\documentclass[a4paper,10pt]{article}

\newif\ifuseprecompiled
\useprecompiledtrue 

\usepackage{graphicx}
\usepackage{pgfplots}
\pgfplotsset{compat=newest}

\ifuseprecompiled 
\else 
\usepackage{tikz}
\usepackage{calc}
\usetikzlibrary{calc}
\usetikzlibrary{arrows.meta} 
\usetikzlibrary{decorations.text}
\usetikzlibrary{shapes}
\usepackage{import}

\usepackage{tikzscale}

\pgfplotsset{plot coordinates/math parser=false}
\usetikzlibrary{plotmarks}
\usepgfplotslibrary{fillbetween}
\usetikzlibrary{patterns}
\usetikzlibrary{intersections}

\usepgfplotslibrary{external}
\fi

\definecolor{RWTHbordeaux}{RGB}{161,16,53}
\definecolor{RWTHblue}{RGB}{0,83,159}
\definecolor{RWTHteal}{RGB}{0,152,161}
\definecolor{RWTHmagenta}{RGB}{138,32,60} 
\definecolor{RWTHsuperlightblue}{RGB}{199,221,242}
\definecolor{RWTHgrey}{RGB}{51,51,51}
\definecolor{RWTHblack}{RGB}{0,0,0}
\definecolor{RWTHorange}{RGB}{246,168,0}
\definecolor{RWTHpetrolgreendark}{RGB}{58,146,59}
\definecolor{RWTHpetrolgreenlight}{RGB}{202,220,197}

\usepackage{subfig}
\usepackage{amsmath,amssymb,amsthm}
\usepackage{amsfonts}
\usepackage{color}
\usepackage{algorithm}
\usepackage{algorithmic}
\usepackage{multirow}
\usepackage{multicol}
\usepackage{relsize,amsmath} 
\usepackage{enumitem} 
\setlist[enumerate]{itemsep=0pt,topsep=0pt,parsep=1pt} 
\setlist[itemize]{itemsep=0pt,topsep=0pt,parsep=1pt} 

\usepackage{placeins} 
\usepackage{mathtools} 
\usepackage{titlesec} 
\titleformat*{\section}{\large \bfseries}
\titleformat*{\subsection}{\normalsize \bfseries}
\titleformat*{\subsubsection}{\normalsize \bfseries}

\usepackage[font={footnotesize,it}]{caption} 
\usepackage{cite} 
\usepackage[colorlinks   = true,linkcolor=RWTHblue,citecolor=RWTHblue]{hyperref}
\makeatletter
\renewcommand*{\eqref}[1]{%
  \hyperref[{#1}]{\textup{\tagform@{\ref*{#1}}}}%
}
\makeatother
\numberwithin{equation}{section} 

\usepackage{thmtools} 
\declaretheoremstyle[
    headfont=\bfseries,
    notefont=\normalfont,
    bodyfont=\itshape,
    postheadhook={},
]{mytheorem}

\declaretheorem[style=mytheorem,numberwithin=section]{theorem}
\declaretheorem[style=mytheorem,numberlike=theorem]{corollary}

\declaretheorem[style=mytheorem,numberlike=theorem]{lemma}

\declaretheorem[style=mytheorem,numberlike=theorem]{example}

\declaretheorem[style=mytheorem,numberlike=theorem]{definition}
\declaretheorem[style=mytheorem,numberlike=theorem]{remark}

\evensidemargin=\oddsidemargin
\newdimen\dummy
\dummy=\oddsidemargin
\ifuseprecompiled 
\else 

\def\plottitlefontsize{\footnotesize}
\def\plottitlefontsizesmaller{\tiny} 

\def\plottickfontsize{\footnotesize}

\newlength\figureheight
\newlength\figurewidth
\fi

\newcommand{\lhb}[1]{\ensuremath{\mathfrak{L}(#1)}}
\newcommand{\rhb}[1]{\ensuremath{\mathfrak{R}(#1)}}

\newcommand{\geomean}[1]{\ensuremath{\langle #1 \rangle_{\operatorname{geo}}}}
\newcommand{\arithmean}[1]{\ensuremath{\langle #1 \rangle_{\operatorname{ar}}}}

\DeclareMathOperator{\diag}{diag}

\DeclareMathSymbol{\mlq}{\mathord}{operators}{``}
\DeclareMathSymbol{\mrq}{\mathord}{operators}{`'}

\newcommand{\Index}{{\mathcal I}}
\newcommand{\argmin}{\mathop{\rm argmin}}

\newcommand{\Ind}{\mathcal{I}}
\newcommand{\R}{\ensuremath{\mathbb{R}}}

\newcommand{\Lf}{\ensuremath{L}}
\newcommand{\Rf}{\ensuremath{R}}
\newcommand{\fL}{\ensuremath{\mathcal{L}}}
\newcommand{\fR}{\ensuremath{\mathcal{R}}}
\newcommand{\fN}{\ensuremath{\mathcal{N}}}

\newcommand{\sigR}{\ensuremath{\theta}}
\newcommand{\SigR}{\ensuremath{\Theta}}
\newcommand{\sigL}{\ensuremath{\gamma}}
\newcommand{\SigL}{\ensuremath{\Gamma}}
\newcommand{\rL}{\ensuremath{r_{\gamma}}}
\newcommand{\rR}{\ensuremath{r_{\theta}}}
\newcommand{\nL}{\ensuremath{n_{L}}}
\newcommand{\nR}{\ensuremath{n_{R}}}
\newcommand{\nN}{\ensuremath{n_{N}}}
\newcommand{\B}{\ensuremath{B}}
\newcommand{\Ps}{\ensuremath{S}}
\newcommand{\gammaL}{\ensuremath{s_1}}
\newcommand{\gammaR}{\ensuremath{s_2}}
\newcommand{\filter}{\ensuremath{\mathcal{F}}}
\newcommand{\lowfil}{\ensuremath{F}}

\newcommand{\N}{\ensuremath{\mathbb{N}}}

\newcommand*\intd{\mathop{}\!\mathrm{d}} 

\newcommand{\ck}{\boxtimes}
\newcommand{\kp}{\otimes}

\def\new#1#2{#2}
\def\moved#1#2{#2}
\def\newwo#1{#1}
\def\movedwo#1{#1}

% ...

\def\oldwo#1{\noindent}
\def\oldwojg#1{\noindent}
\def\oldjg#1#2{\noindent}
\def\old#1#2{\noindent}
\def\oldSTATE#1#2{\noindent}

\def\BC#1{}

\def\BCn#1#2{}

\newcommand*\samethanks[1][\value{footnote}]{\footnotemark[#1]}

\author{Lars Grasedyck\thanks{IGPM, RWTH Aachen University, Templergraben 55, 52056 Aachen \protect\\ {\tt lgr.rwth-aachen.de}, {\tt kraemer@igpm.rwth-aachen.de} \protect\\
Both authors gratefully acknowledge support by the DFG priority programme 1648 under grant GR3179/3-1.} \and Sebastian Kr\"amer\samethanks[1]}

\title{Stable ALS Approximation in the TT-Format for Rank-Adaptive Tensor Completion} 
\date{}
\begin{document}

\maketitle

\begin{abstract}
Low rank tensor completion is a highly ill-posed inverse problem,  
particularly when the data model is not accurate, and
some sort of regularization is required in order to solve it. 
In this article we focus on the calibration of 
the data model.  
For alternating optimization,
we observe that existing rank adaption methods do not
enable a continuous transition between manifolds of different ranks.
We denote this characteristic as \textit{instability (under truncation)}. As a consequence of this property, arbitrarily 
small changes in the iterate can have arbitrarily large influence on the further reconstruction.
We therefore introduce a singular value based regularization to the standard alternating least squares 
(ALS), which is motivated by averaging in microsteps. We prove its \textit{stability}
and derive a natural semi-implicit rank adaption strategy.
We further prove that the standard ALS microsteps for completion problems are only stable on 
manifolds of fixed ranks, and only around points that have what we
define as \textit{internal tensor restricted isometry property, iTRIP}. 
In conclusion, numerical experiments are provided that show improvements of the reconstruction quality 
up to orders of magnitude in the new Stable ALS Approximation (SALSA) compared to standard ALS and the well known Riemannian optimization RTTC.
\\

\smallskip 
\noindent \textbf{Keywords.} tensor completion,  MPS,  tensor train,  TT,  hierarchical Tucker,  HT,  alternating optimization,  ALS,  high-dimensional,  low rank,  SVD,  ill-posedness,  stability \\

\smallskip 
\noindent \textbf{AMS subject classifications.} 15A18, 
15A69, 
65F22, 
90C06, 
90C31 \\
\end{abstract} 
\section{Introduction}
Low rank tensor completion is a highly ill-posed inverse problem.
In order for any recovery to succeed, regularity assumptions are required.
This is either achieved by adding certain penalty terms,
or by using an explicit reduction of degrees of freedom, which in the
context of high-dimensional tensors can be obtained by using low rank representations,
cf. \cite{GrKrTo13_Ali,Ha14_Num,HaSc14_Ten}.
The main goal of this work is to derive a rank adaptive method for tensor completion\BCn{b1_}{b_}
and moreover to discuss the benefits of such as well as the reasons why heuristics tend to be insufficient.
Since the concept of \textit{stability} in the sense of Definition \ref{cowitr} 
has not yet been considered in literature,\BCn{a1}{a}
the initial part is dedicated to the simpler matrix case (Sections \ref{sec:notcons} and \ref{sec:stabmatrixcomp})
in order to provide an easier access\BCn{i1_}{i_}. The subsequent analysis will focus on the importance of these concepts
to least squares tensor completion where the calibration of model complexity
is more challenging.
\subsection{Introduction to Stability for Ill-posed Inverse Problems through the Example of Matrix Completion} 
In the setting of low rank matrix completion, 
the target of recovery is a matrix $M \in \R^{n \times m}$ which is only observable at points \BCn{l1_}{l_}\BCn{c1}{c}
\[ M|_P = \{ M_p \}_{p \in P} \in \R^P \new{36}{\cong \R^{|P|}} \quad \mbox{for} \quad P \newwo{\subset \mathcal{I} :=} \{1,\ldots,n\} \times \{1,\ldots,m\} ,\]
where $P$ is a given, fixed sampling set, which we hence can not enlarge.
One very strict regularity assumption is given by $\mathrm{rank}(M) = r$ for some sufficiently small $r \in \mathbb{N}$, which
leads to the minimization problem
      \begin{align*} \mathrm{minimize\ } & \|A - M\|_{P} \\ 
	\mathrm{subject\ to\ } & A \in \R^{n \times m}, \ \mathrm{rank}(A) \leq r,
      \end{align*} 
$\mbox{where } \|B\|^2_{\rm{P}} := \sum_{i \in P} B^2_i$ for matrices $B$.
A favorable data model for this task is the low rank representation, 
i.e. a function
\[ \tau_r: (X,Y) \mapsto A = XY \in \R^{n \times m} \quad \mbox{for} \quad (X,Y) \in \mathcal{D}_r :=\R^{n \times r} \times \R^{r \times m}.\]
Since every matrix has a unique rank, we can partition
the target space $\R^{n \times m}$ into the disjoint subsets 
\[ \mathcal{T}_r := \{ A \mid \mathrm{rank}(A) = r \}, \quad r = 0,\ldots,\min(n,m).\] 
With $\mathrm{image}(\tau_r) = \bigcup_{\widetilde{r} \leq r} \mathcal{T}_{\widetilde{r}}$ in mind, the optimization is performed on the representation or data space $\mathcal{D}_r$.
In an alternating least squares (ALS) method for example, one then applies two
optimization methods $\mathcal{M}^{(1)}$, $\mathcal{M}^{(2)}$, 
\begin{align} \label{MX} \mathcal{M}^{(1)}_r(X,Y) & := (\argmin_{\widetilde{X}} \|\widetilde{X} Y - M\|_P,Y),\\
\label{MY} \mathcal{M}^{(2)}_r(X,Y) & := (X,\argmin_{\widetilde{Y}} \|X \widetilde{Y} - M\|_P),
\end{align} 
which in this context are called microsteps.
Note that the $\mathrm{argmin}$ is not necessarily unique, and we choose the element minimizing the Frobenius norm $\|X Y\|_F$.
Formally, for each value of the matrix rank $r$, every single $\mathcal{M}_r^{(1)}$, $\mathcal{M}_r^{(2)}$ is a different function.
\\\\
For most realistic applications, it is more reasonable to relax the regularity assumption to $M$ being \textit{nearly} rank $r$,
which means that after $r$ entries, the singular values of $M$ become sufficiently smaller. Ultimately,
if no assumptions are made, the appropriate model complexity is a matter of the quality and magnitude of $P$ with
respect to $M$.
Yet in the general case, the missing structure of given data hardly allows to obtain knowledge about this relation.\BCn{a2_}{a_}
Therefor, since overestimating the model complexity (i.e. the rank) ultimately leads to flawed results,
a cautious learning process is required to adapt such, as rank increasing strategies already suggest.
Due to the difficult nature of the problem, we do not expect to be able to find the global minimizer, but instead focus\BCn{d2_}{d_}
on single aspects that are likely to improve the approximation quality.
\\\\
Each adaption of the rank during the optimization will cause the algorithm to change between data spaces $\mathcal{D}_r$. 
Intuitively, considering that the 
generated spaces $\mathcal{T}_r$ have pairwise distance $0$ within $\R^{n \times m}$,
one would want
that a change of rank does not have large impact, given the problematic nature of overfitting. This, however, is \BCn{e1}{e}
not true for both $\mathcal{M}^{(1)}, \mathcal{M}^{(2)}$, while 
arbitrarily small perturbations of the iterate $A = \tau_r(X,Y)$ may change its rank.
For ill-posed inverse problems\BCn{a1_}{a_}\BCn{c1_}{c_}, we hence propose the following concept:
\begin{definition}[Stability]\label{cowitr}
  Let $\mathcal{M}$ be a method that
  maps any rank $r$ to a function $\mathcal{M}_r: \mathcal{D}_r \rightarrow \mathcal{D}_r$
  (the optimization method for fixed rank).
  We define the following properties:
  \begin{itemize}
    \item $\mathcal{M}$ is called \textit{representation independent}, if $\tau_r(\mathcal{M}_r(G)) = \tau_r(\mathcal{M}_r(\widetilde{G}))$ 
    for all $r$ and $G,\widetilde{G} \in \mathcal{D}_r$ with $\tau_r(G) = \tau_r(\widetilde{G})$.
    We then define $\tau_r^{-1}$ to map to one possible representation (we want to circumvent the
    use of equivalence classes).
    \item $\mathcal{M}$ is called \textit{\old{2}{fix-rank}\new{2}{fixed-rank} stable}, if it is representation independent\BCn{b1}{b}
    and for any fixed rank $r$, the map 
    $\tau_r \circ \mathcal{M}_r \circ \tau^{-1}_r: \mathcal{T}_r \rightarrow \R^\Ind$ is continuous.
    \item $\mathcal{M}$ is called \textit{stable}, if it is representation independent
    and the function 
   \begin{align}
    \label{fM} f_{\mathcal{M}}: \R^\Ind \rightarrow \R^\Ind,\quad f_{\mathcal{M}}(A) := \tau_{r(A)} \circ \mathcal{M}_{r(A)} \circ \tau^{-1}_{r(A)}(A),
   \end{align}
  where $r(A)$ is the rank of $A$, is continuous. 
  \end{itemize}
\end{definition}
\begin{figure}[htb]
  \begin{center}
      \ifuseprecompiled
      \includegraphics[width=0.98\textwidth]{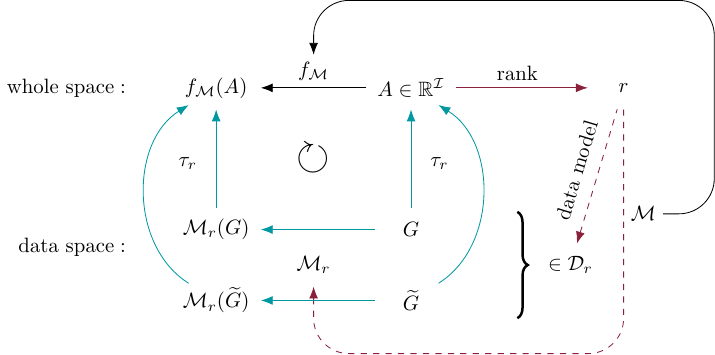}
      \fi
  \end{center}
  \caption{\label{stability_figure} The diagram depicting Definition \ref{cowitr}. Magenta part: depending on the rank of $A$, the  
  method $\mathcal{M}$ provides a specific mapping $\mathcal{M}_r$ to be applied to equivalent representations $G,\widetilde{G} \in \mathcal{D}_r$.
  Teal part: representation independent states
  that $f_{\mathcal{M}}$ is well-defined since both lower paths from $A$ along the data space result in the same output within the whole space. Stability
  requires that this function, the upper path, is continuous.}
\end{figure}
This definition of stability is not restricted to matrix completion. Data spaces, the rank 
(possibly generalized to any model complexity) and the method can be replaced
by any appropriate type, in particular tensor completion in hierarchical tensor formats (cf. Section \ref{sec:gentohighd}). The only 
assumption should usually be a nestedness of spaces, i.e. $\mathcal{T}_r \subset \mathcal{T}_{\widetilde{r}}$
whenever $r \preceq \widetilde{r}$ (entry-wise inequality). \\
Properly calibrating the rank $r$ for unstable methods in the context of ill-posed inverse problems may lead to complications.
Most of the operators applied to representations are stable, e.g. truncations based on 
matrix singular values. The situation however changes if we apply the partial optimization (or micro-)
step $\mathcal{M}^{(1)}$ or $\mathcal{M}^{(2)}$ on a low rank representation: 
%
\begin{example}[Instability of alternating least squares matrix completion steps]\label{inofallesqco}
  Let $a\in\mathbb{R}\setminus\{0,1\}$ be a possibly very small parameter. 
  We consider the target matrix $M$ and an $\varepsilon$-dependent initial approximation 
  $A=A(\varepsilon)$
  \[
  M := 
  \begin{pmatrix}
    \fbox{ ? }   & 1.1 & 0.9 \\
    1   & 1   & 1.1 \\
    1.1 & 1   & 1
  \end{pmatrix}, \quad
  A(\varepsilon) 
  := 
  \begin{pmatrix}
    1 & 1 & 1 \\
    1 & 1 & 1 \\
    1 & 1 & 1 
  \end{pmatrix} 
  + \varepsilon
  \begin{pmatrix}
    0.5+a   & 0.5+a & -a   \\
    1+a & 1+a & -1-a \\
    1-a & 1-a & -1+a
  \end{pmatrix},
  \]
  where the entry $M_{1,1}$ (the question mark above) is not known or given. The matrix $M$ is of
  rank $3$ and $A(\varepsilon)$ is of rank $r = 1$ for $\varepsilon = 0$ and of rank $r = 2$ otherwise.
  We seek a best approximation of (at most) rank $2$ in the least squares sense for the known entries
  of $M$. 
  In a single ALS step, as defined by \eqref{MY}, we replace $Y(\varepsilon)$ of the low rank 
  representation $A(\varepsilon) = X(\varepsilon) Y(\varepsilon)$ by the local minimizer, where in this case
  \[
  A(0) = \begin{pmatrix}
    1 \\ 1 \\ 1
  \end{pmatrix}
  \begin{pmatrix}
    1 & 1 & 1
  \end{pmatrix}
  ,\qquad 
  A(\varepsilon) = \begin{pmatrix}
    1 & 0.5+a   \\
    1 & 1+a \\
    1 & 1-a 
  \end{pmatrix} 
  \begin{pmatrix}
    1 & 1 & 1 \\
    \varepsilon & \varepsilon & -\varepsilon
  \end{pmatrix}
  \mbox{ if } \varepsilon > 0.
  \]
  This optimization yields a new matrix, $B(\varepsilon) = f_{\mathcal{M}^{(2)}}(A(\varepsilon)) = \tau_r \circ \mathcal{M}^{(2)}_r \circ \tau^{-1}_r(A(\varepsilon))$ (independently of the chosen representation),
  given by
  \[
  B(0) = 
  \begin{pmatrix}
    1.05 & * & *\\
    1.05 & * & *\\
    1.05 & * & *
  \end{pmatrix}
  ,\qquad 
  B(\varepsilon) = 
  \begin{pmatrix}
    1+\frac{1}{40a} & * & *\\
    1.0  & * & *\\
    1.1  & * & *
  \end{pmatrix}
  \mbox{ if } \varepsilon > 0.
  \quad (* \mbox{ is some value})
  \]
  Now let $a$ be fixed and let $\varepsilon$ tend to zero so that the initial guess 
  $A(\varepsilon)\to A(0)$. However, $B(\varepsilon) \nrightarrow B(0)$, thus violating the stability. 
  Furthermore, the rank two approximation $B(\varepsilon)$, given an arbitrary, fixed $\varepsilon > 0$, diverges as $a\to 0$,\BCn{d1}{d}
  in particular it is not convergent although the initial guess $A(\varepsilon)$ converges
  to a rank two matrix as $a\to 0$. Thus, the microstep is not even stable for fixed rank. 
  We want to stress that the initial guess is bounded for all $\varepsilon,a\in(0,1)$,
  but the difference between $B(0)$ and $B(\varepsilon)$ is unbounded for $a\to 0$ (cf. Definition \ref{iTRIP}). The
  unboundedness can be remedied by adding a regularization term in the least squares
  functional, e.g. $+\|X Y\|$, but the ALS step remains unstable.
\end{example}
%
\BCn{a3_}{a_}%
This example likewise demonstrates that ALS for tensor completion is not stable \old{25.5}{, and this is fatal, especially}\new{25.5}{and thus, as discussed before, problematic} when adapting
the rank (cf. Section \ref{sec:gentohighd}).
\old{38}{It is easy to see}\new{38}{We will further show} that this is not a marginal phenomenon, but occurs systematically
during any rank change (cf. Example \ref{alsanadfarnocowitr}).

\subsection{Relation to Other Matrix and Tensor Methods}\label{reltother}
%
Whenever a tensor is point-wise available, algorithms such as the TT-SVD \cite{Os11_Ten}
can just establish the exact rank based on its very definition or
a reliable rank estimate as well as representation can be obtained through cross-approximation methods, a setting in which
the subset of used entries can be chosen freely \cite{OsTy10_TTc,BaGrKl13_Bla}.\\
If only indirectly given, adapting the rank of the sought low rank tensor can still be straight-forward, e.g. when 
the rank has to be limited \textit{only} due to computational complexity, while in principle the exact solution
is desired \cite{BeMo13_Num,BaScUs16_Ten}. Here, an optimal regulation of thresholding parameters becomes most important.
This mainly includes classical problems that have been transferred to large scales.
These may for example be solved with iterative methods \cite{BaGr13_Apr,BaSc16_Ite,MaZa12_Sol}, which naturally increase the
rank and rely on subsequent reductions,
or also by rank preservative optimization, such as alternating optimization \cite{HoRowSc12_The,DoSa14_Alt,RoUs13_OnL,EsKh15_Con}, possibly combined
with a separate rank adaption. \\
Provided that the tensor restricted isometry property holds, the task may be interpreted as distance minimization with respect to
a norm that is sufficiently similar to the Frobenius norm and analyzed based on compressed sensing \cite{RaScSt15_Ten}.
Black box tensor completion for a fixed sampling set, however, requires a certain solution to a positive-semi definite linear system.
Hence neither an exact solution is reasonable nor does any norm equivalence hold.
Thus, the available data is easily misinterpreted, the more so if the rank is overestimated, and 
truncation based algorithms, including DMRG \cite{HoRowSc12_The,Je02_Dyn}, are misled. \\
Nuclear norm minimization, being closely related to compressed sensing as well, has a very strong
theoretical background \cite{CaRe09_Exa,CaTa10_The,Gr11_Rec,Re11_ASi} for the matrix case. These approaches rely on a direct adaption of the target function, that is convex relaxation. 
Yet it appears that they are outperformed in practice by alternating least squares approaches \cite{Ja13_Low} and
the simplifications required for an adaption to tensors \cite{LiSh13_AnE,SiTrLaSu14_Lea,GaReYa11_Ten}
do not seem to allow for an appropriate generalization \cite{MuHuWrGo14_Squ}.
Also the approaches which themselves retreat to alternating least squares \cite{Ha15_Mat} treat the iterations as necessity
for the minimization of an objective function with regularization term.
The penalty term is, as usual,
based on the singular values of the output of a microstep (a posteriori), 
as it is also the case for the work \cite{St16_Rie} on tensor completion through Riemannian optimization.\BCn{j12}{j}
Although their term may appear similar to the term we derive (cf. Theorem \ref{miofthalsrerefufoma}), the stability property, on the contrary, requires that the penalty term 
depends on the current singular values before the microstep (a priori), which is an essential difference: we explicitly allow and exploit small singular values instead of penalizing such.
In that sense, we treat each update and adaption as part of a \textit{learning progress}, where the magnitudes of singular values indicate in some
respects an uncertainty of approximation. To the best of our knowledge, this point of view has not yet been considered in literature and
hence a concept of stability as we define it has not been investigated.
\\
For fixed or uniform rank, there have been least squares based proposals in hierarchical tensor formats \cite{KrStVa14_Low,SiHe15_Opt} as well.
The essential adaption of the rank however, including similar matrix approaches, is rarely considered, all the less in numerical tests, and 
remains an open problem in this setting. 
A mentionable approach so far is the rank increasing strategy 
\cite{WeYiZh12_Sol,GrKlKr15_Var} and 
its regularization properties are a first starting point for this article.
\\\\
\textbf{The rest of the article is organized as follows}: 
In Section \ref{sec:notcons}, we further investigate instability
and exemplarily analyze approaches towards it in the matrix case. Based on this insight, we motivate
a variational residual function, derive its minimizer and present a stable algorithm for matrix completion in Section \ref{sec:stabmatrixcomp}.
In Section \ref{sec:gentohighd}, we begin to generalize former results to high dimensional tensors, yet essentially work in three dimensions.
In the main Section \ref{sec:moofthsicorefu}, we then derive its minimizer and prove stability (Theorem \ref{coofthfiadfmist}) for the
thereby obtained regularized microsteps\newwo{, further analyzing these in Section \ref{sec:filter}}. Subsequently, in Section \ref{sec:backtod}, these results
are transferred back to arbitrarily dimensional tensors. 
Section \ref{sec:seimannounraad} finishes with the necessary details for the algorithm, including its rank adaption
as it is naturally given through stable alternating least squares. 
Comprehensive numerical tests (exclusively for unknown ranks) are provided in Section \ref{sec:ne}.
Detailed tables with the numerical values shown in figures can be found in the appendix.
\section{Instability and Approaches to Resolve the Problem}\label{sec:notcons}
%
\old{28}{As previously mentioned, instability poses a systematic flaw in ALS, or for that matter, in any such range based optimization}%
\new{28}{The following example shows that instability can be observed systematically during rank changes in ALS, or more general, in any such range based optimization.
In that sense, the implied complications for ill-posed, inverse problems may frequently occur:} 
\begin{example}[ALS for ill-posed, inverse problems is unstable]\label{alsanadfarnocowitr}
  Consider the microstep $\mathcal{M}^{\newwo{{(2)}}}$ as in \eqref{MY}.
  Let $U \in \R^{n \times r}$, $V \in \R^{m \times r}$ be orthogonal, such that $U \Sigma V^T$ is a truncated SVD\BCn{h1}{h} of a rank $r$ matrix $A = \tau_r(U,\Sigma V^T) \in \R^{n \times m}$.
  We now let $\sigma_r \rightarrow 0$, $\sigma_r > 0$, such that in the limit $A^{\ast} := A|_{\sigma_r = 0}$ has rank $r-1$.
  The update is independent of this last singular value though:
  \begin{align} \label{limit} f_{\mathcal{M}^{\newwo{{(2)}}}}(A) = \tau_r(\mathcal{M}^{\newwo{{(2)}}}_r(U, \Sigma V^T)) = U \argmin_{Y} \|U Y - M\|_P = \lim_{\varepsilon \searrow 0} f_{\mathcal{M}^{\newwo{{(2)}}}}(A|_{\sigma_r = \varepsilon}) \end{align}
  However, if $\sigma_r = 0$, then $A|_{\sigma_r = 0}$ has rank $r-1$ and a truncated SVD $U_c \Sigma_c V_c^T$. Hence, the update
  \[ f_{\mathcal{M}^{\newwo{{(2)}}}}(A|_{\sigma_r = 0}) = \tau_{r-1}(\mathcal{M}^{\newwo{{(2)}}}_{r-1}(U_c, \Sigma_c V_c^T)) = U_c \argmin_{Y} \|U_c Y - M\|_P \]
  is in general different from the limit \eqref{limit}\BCn{n1_}{n_}\new{38}{, given that the range of $U$ is different from $U_c$}. The same holds for an analogous update $\mathcal{M}^{\newwo{{(1)}}}$ of $X = U \Sigma$.
  Note that these updates are indeed representation independent.
\end{example}
The microsteps of ALS in the tensor case behave in the same way\BCn{o1_}{o_}. The only difference is that there are two
tuples of singular values $\sigma^{(\mu-1)}$ and $\sigma^{(\mu)}$ adjacent to a core $G_\mu$ (cf. Lemma \ref{stre}). 
There may be many ways to stabilize the microsteps. However, we aim for an as little distorting as practical way to do this.
A quite successful approach for completion has been the rank increasing strategy, e.g. \cite{WeYiZh12_Sol}.
\old{38}{By the given limitation to all ranks, a regularization is
introduced to the target function. High frequencies, corresponding to small singular values, are
excluded up to a certain progress.}%
\new{38}{Therein, the model complexity is slowly increased, step by step attempting a more distinct approximation. Thereby, local minima
that correspond to overfitting are avoided.}
\\
A similar kind of effect can be achieved by assuming an uncertainty on the current iterate, or,
equivalently, averaging the tensor update function. That way, the level of regularization can be adapted
continuously and is less dependent on the \old{38}{\textit{technical} rank that is currently used.}\new{38}{discrete rank but the more meaningful, real-valued singular values.}  
We will first view this in a minimal fashion for the matrix case and the method $\mathcal{M}^{\newwo{{(2)}}}$ defined by \eqref{MY} (for the remainder
of this section called $\mathcal{M}$ for simplicity).
With this approach, we can motivate an algorithm that is stable under truncation and
allows to straightforwardly adapt ranks\old{38}{ (as well as nonuniformly for tensors)}. \old{38}{It optimizes, in a loose sense,
continuously between manifolds of different ranks.}\new{38}{While other methods in fact keep distance from the border of
one manifold of a fixed rank, we aim to optimize, in a certain sense, continuously between manifolds of different ranks.}\\
\textit{Assuming} local integrability of $f_{\mathcal{M}}$ (as defined in \eqref{fM}), we obtain that the variational function 
\begin{align} 
  \label{fastM} f^{\ast}_{\mathcal{M}}(A) & := \frac{1}{|\mathbb{V}_{A,\omega}|} \int_{\mathbb{V}_{A,\omega}} f_{\mathcal{M}}(H) \intd  H \\
  \nonumber \mathbb{V}_{A,\omega} & := \{H \in image(\tau_{\widetilde{r}}) \mid \|H-A\|_{\new{6}{F}}\leq \omega\}
\end{align}
is continuous within $image(\tau_{\widetilde{r}})$, where $\|\cdot\|_F$ is the Frobenius norm and $\widetilde{r}$ may be considered an upper bound to the rank (cf. Figure \ref{fancy}).
\begin{figure}
  \centering
  \ifuseprecompiled
\includegraphics[width=0.98\textwidth]{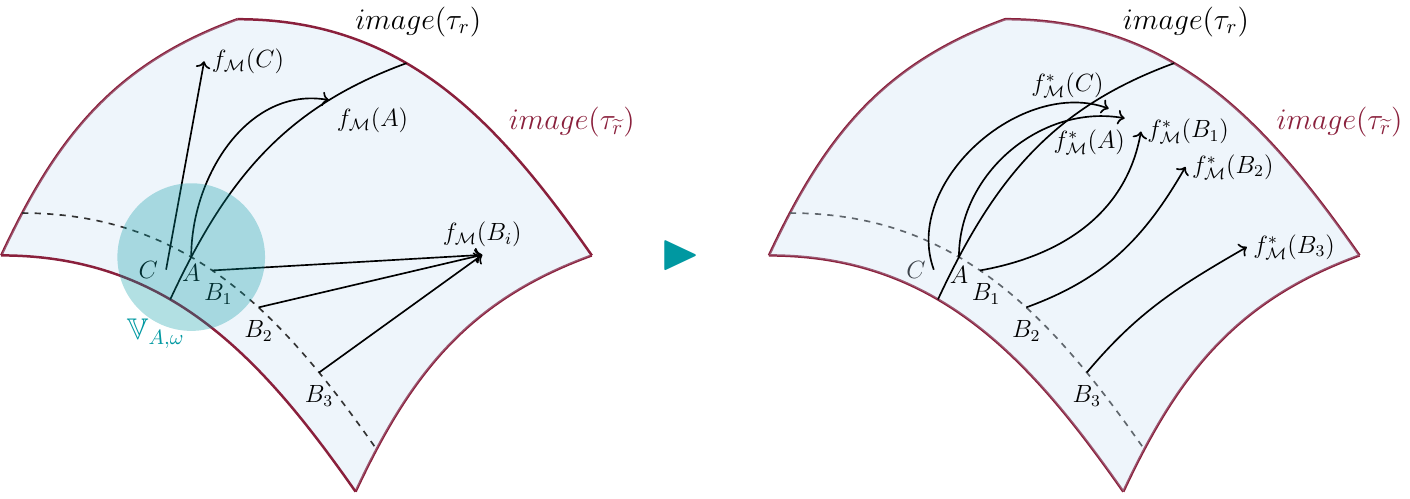}
\fi
\caption{\label{fancy} The schematic display of the unstable function $f_\mathcal{M}$ (left) and the variational, stable $f^{\ast}_\mathcal{M}$ (right). 
In both pictures, the image of $\tau_r$ is depicted as black curve contained in the image of $\tau_{\widetilde{r}}$ shown
as blue area (with magenta boundary). $A$ is a rank $r$ \old{39}{tensor}\new{39}{element}, while $C$ and each $B_i$ has rank $\widetilde{r}$. 
Left: Regardless of their distance to $A$, the tensors $B_1, B_2$ and $B_3$ (and any other point of the dotted line except the lower rank \old{39}{tensor}\new{39}{element} $A$)
are mapped to the same point $f_{\mathcal{M}}(B_i)$. Likewise, $C$ is, although as close to $A$ as $B_1$, mapped
to a completely different point. The teal circle exemplarily shows one possible range of averaging at the point $A$.
Right: If \old{39}{a tensor}\new{39}{an element} (such as $B_1$ and $C$) is close to $A$, then this also holds for their function values.
However, \old{39}{the }$f^{\ast}_{\mathcal{M}}(A)$ is not rank $r$ anymore 
(in fact, the image of $f^{\ast}_{\mathcal{M}}$ is generally not even rank $\widetilde{r}$).}
\end{figure}
However, this function does not preserve low rank structure\new{38}{, apart from appearing to be to complicated to evaluate,} and therefore 
we cannot find a method $\mathcal{M}^{\ast}$ for which $f_{\mathcal{M}}^{\ast} = f_{\mathcal{M}^{\ast}}$. 
Still, all simplifications which we will make remain subject to this motivation.
\subsection{Investigations into the Connection between Averaging and Stability}\label{sec:conect}
\old{7}{Consider instead a scenario in which we limit the perturbation that the left singular vectors $U$
receive due to the variation of $A$ to only one component (as limit case of $\sigma_1 \gg \sigma_2 \approx \omega$). 
From this, we will observe important consequences.}%
\new{7}{Although the following examples are heavily narrowed down, we observe which kind of simplifications we may apply in order to obtain a feasible, but still stable method (cf. Section \ref{sec:stabmatrixcomp}). \BCn{h1_}{h_}\BCn{g1}{g}
First, we consider a scenario in which we limit the perturbation that the left singular vectors $U$
receive due to the variation of $A$ to only one component as it is approximately the case if $\sigma_1 \gg \sigma_2 \approx \omega$.}
\begin{lemma}[Variational low rank matrix approximation]\label{coloramaap}
  Let $\mathcal{M}$ be defined by \eqref{MY} and $P = \Ind$ (full sampling).
  Let further $A = U \Sigma V^T \in \R^{n \times m}$ be of rank two, given by its SVD components $U = (u_1 \mid u_2)$, $\Sigma = diag(\sigma_1,\sigma_2)$ and $V$ as well as $M \in \R^{m \times m}$ arbitrary
  and $0 < \omega < \sqrt{2}\sigma_2$. Then
  \begin{align}\label{conv}
    \nonumber & \widehat{f}_{\mathcal{M}}(A) := \frac{1}{|V_{\omega}|} \int_{V_{\omega}} f_{\mathcal{M}} \left( (u_1 \mid u_2 + \Delta u_2) \Sigma V^T \right) \intd  \Delta u_2 \\
    &  = \underbrace{u_1 u_1^T M}_{\mbox{optimization}} \ + \ \underbrace{(1-\alpha_{\omega})^2  u_2 u_2^T M}_{\mbox{regularization}} \ + \ \underbrace{\frac{2 \alpha_{\omega} - \alpha_{\omega}^2}{m-2} (I_m - u_1 u_1^T - u_2 u_2^T) M}_{\mbox{replenishment}}\\
    \nonumber &  V_{\omega} := \left\{ \Delta u_2 \mid \|(u_1 \mid u_2 + \Delta u_2) \Sigma V^T - A\|_{\new{6}{F}} = \omega, \ (u_1 \mid u_2 + \Delta u_2) \mbox{ has orthonormal columns} \right\}
  \end{align}
  for $\alpha_{\omega} = \frac{\omega^2}{2 \sigma_2^2}$, such that $\alpha_{\omega} \rightarrow 1$ if $\omega \rightarrow \sqrt{2}\sigma_2$. 
  Alternatively, considering complete uncertainty concerning the second singular vector, we obtain
  \begin{align*}
    & \frac{1}{|V_{\omega}|} \int_{V_{\omega}}  f_{\mathcal{M}} \left( (u_1 \mid \Delta u_2) \Sigma V^T \right) \intd  \Delta u_2 
    = u_1 u_1^T M + \frac{1}{m-1} (I_m - u_1 u_1^T) M, \\
    & \mbox{where here} \quad V_{\omega} := \{ \Delta u_2 \mid (u_1 \mid \Delta u_2) \mbox{ has orthonormal columns} \}.
  \end{align*}
\end{lemma}
\begin{proof}
  We parameterize $V_{\omega}$. First, $\omega = \|(u_1 \mid u_2 + \Delta u_2) \Sigma V^T - A\|_{\new{6}{F}} = \|\Delta u_2\|_{\new{6}{F}} \ \sigma_2$ and hence $\|\Delta u_2\|_{\new{6}{F}} = \frac{\omega}{\sigma_2}$.
  By orthogonality conditions, we obtain $\Delta u_2 = - \alpha_{\omega} u_2 + \Delta u^{\perp}_2$ with $\Delta u^{\perp}_2 \perp \mbox{range}(U)$ for a
  fixed $\alpha_{\omega} = \frac{\omega^2}{2 \sigma_2^2}$. Hence, $V_{\omega}$ is an $(m-3)$-sphere of radius 
  $\beta_{\omega} = \sqrt{\frac{\omega^2}{\sigma_2^2} - \alpha_{\omega}^2}$, that is $\beta_{\omega} S^{m-2}$. The update for each instance of $\Delta u^{\perp}_2$ is given by
  \begin{align*} 
    f_\mathcal{M}( (u_1 \mid u_2 + \Delta u_2) \Sigma V^T ) = (u_1 \mid u_2 + \Delta u_2) (u_1 \mid u_2 + \Delta u_2)^T M.
  \end{align*}
  We integrate this over $V_{\omega}$ and obtain
  \begin{align*}
    \int_{V_{\omega}} f_\mathcal{M} = \int_{V_{\omega}} u_1 u_1^T M + \int_{V_{\omega}} (1-\alpha_{\omega})^2 u_2 u_2^T M + \int_{V_{\omega}} \Delta u^{\perp}_2 {\Delta u^{\perp}_2}^T M
  \end{align*}
  since all integrals of summands which contain $\Delta u^{\perp}_2$ exactly once vanish due to symmetry.
  We can simplify the last summand with Lemma \ref{inovalva} to 
  \begin{align*}
    \int_{V_{\omega}} \Delta u^{\perp}_2 {\Delta u^{\perp}_2}^T M = \int_{\beta_{\omega} S^{m-2}} (Hx) (Hx)^T M \intd x = H H^T \frac{2 \alpha_{\omega} - \alpha_{\omega}^2}{m-2} |V_{\omega}| M
  \end{align*}
  for a linear, orthonormal map $H$ that maps $x \in \beta_{\omega} S^{m-2}$ to $\Delta u^{\perp}_2$, that is, embeds it into $\R^m$.
  One can then conclude that $H H^T = I_m - u_1 u_1^T - u_2 u_2^T$, since the rank of $H$ is $m-2$ and $range(H) \perp range(U)$.
  The division by $|V_{\omega}|$ then finishes the first part. The second part is analogous.
\end{proof}
We can observe that, in this case, 
choosing $\omega$ close to $\sigma_2$, or in that sense a small $\sigma_2$, will \textit{filter} out influence of $u_2$. 
This is indeed in agreement to the update which the rank $1$ best-approximation to $A$ would yield\footnote{Note that we fixed $\|\Delta u_2\|_{\new{6}{F}} = \omega$ for simplicity as well as that for $\omega > \sqrt{2} \sigma$, Example \ref{coloramaap} does not make sense. 
Allowing perturbations \textit{up to} a magnitude $\omega$ will prohibit that the influence of
$u_2$ vanishes completely, hence $u_2$ is never actually truncated.}. 
More importantly, the result $\widehat{f}_{\mathcal{M}}(A)$ in \eqref{conv} is not low rank, yet
close to the rank $2$ approximation $U (u_1^T M \mid (1-\alpha_{\omega})^2 u_2^T M)$, in which the first component
$U$ has remained the same.
While the variational model as in \eqref{fastM} remains the basic idea, it is too complicated
to be used for the derivation of a stable method $\mathcal{M}^{\ast}$. We instead consider
a slightly modified approach (which will be used in the following Sections \ref{sec:stabmatrixcomp} and \ref{sec:moofthsicorefu}).
\begin{lemma}[Low rank matrix approximation using a variational residual function]\label{averobj}
  In the situation of Lemma \ref{coloramaap}, we have
  \begin{align}\label{unc}
    \argmin_{\widetilde{V}} \ \frac{1}{|V_{\omega}|} \int_{V_\omega} \|(u_1 \mid u_2 + \Delta u_2) \widetilde{V} - M\|_{\new{6}{F}}^2 \intd \Delta u_2 = (u_1^T M \mid (1 - \alpha_{\omega}) u_2^T M)
  \end{align}
\end{lemma}
\begin{proof}
Let $\widetilde{V}$ be the corresponding minimizer.
  With the same derivation as in Lemma \ref{coloramaap}, we obtain
  {\setlength{\belowdisplayshortskip}{-2pt}\setlength{\belowdisplayskip}{-2pt}\begin{align*}
    |V_{\omega}| \widetilde{V} & = \int_{V_\omega} (u_1 \mid u_2 + \Delta u_2)^T M \intd  \Delta u_2 \\
    & = \left( u_1^T M |V_{\omega}| \mid u_2^T M |V_{\omega}| + \int_{V_\omega} -\alpha_{\omega} u_2^T M + \Delta u^{\perp}_2 M \mbox{d} \Delta u_2 \right) \\
    & = |V_{\omega}| \left( u_1^T M \mid (1 - \alpha_{\omega}) u_2^T M \right).
  \end{align*}}
\end{proof}
Comparing this to the rank $2$ approximation of the previous result \eqref{conv}, we observe that solely
$(1 - \alpha_{\omega})^2$ has been replaced by $1 - \alpha_{\omega}$. For our purpose, these terms are sufficiently similar for small $\alpha_{\omega} \geq 0$.\\
\begin{remark}[Replenishment and lower limit]\label{repllowerlimit}
\old{25}{The replenishment term in \eqref{conv} is crucial. Without this term, some parts of the iterates will simply converge to zero for fixed $\omega$ (cf. Appendix B).
We later bypass this problem by setting a lower limit to all occurring singular values.}
\new{25}{The so called replenishment term in \eqref{conv} points at an important aspect which we analyze in Section \ref{sec:filter}.
We later bypass related problems through an additional manipulation of singular values, the intensity of which is proportional to the current residual.}
\end{remark}
%
We here refer to a Matlab implementation of a (superficially random) Monte Carlo approach
to the unsimplified variational microstep $f^{\ast}_{\mathcal{M}}$ as in \eqref{fastM} for
matrix completion\newwo{, which is linked on the personal webpage of the author Sebastian Kr\"amer\footnote{by the time the paper is written, the address is {\tt www.igpm.rwth-aachen.de/team/kraemer}}}.
Likewise, an implementation of the final algorithm \textit{SALSA} (Algorithms \ref{CALS_alg_matrix} and \ref{CALS_alg}),
which is developed from the idea in Lemma \ref{averobj}, can be found for the matrix case
as well as for the tensor case.
\section{Stable Alternating Least Squares Microsteps for Matrix Completion}\label{sec:stabmatrixcomp}
%
In this section, we adapt the target function of each microstep $\mathcal{M}$ in order to
obtain a stable method $\mathcal{M}^{\ast}$.
When performing ALS, instead of just one specific iteration point \newwo{$A \in \mathcal{T}_r$}, the results in Section \ref{sec:conect} suggest to instead
%
consider a set $V_{\omega}(A)$ of variations, or uncertainty, $\Delta A$ along the manifold $\mathcal{T}_r$ of rank $r$ matrices: 
\[ V_{\omega}(A) := \{ \Delta A \mid A + \Delta A \in \mathcal{T}_r, \ \|\Delta A\|_{\new{6}{F}} \ \leq \omega\}, \quad r = \mbox{rank}(A) \]
The initial idea is slightly similar to gradient sampling (e.g. \cite{Bu05_ARo}), but is then pursued differently.
Let $A = \tau_r(X,Y)$ and $(\Delta X, \Delta Y)$ such that $A + \Delta A = \tau_r(X + \omega \Delta X,Y + \omega \Delta Y)$.
Then
\begin{align}\nonumber
  \|\Delta A\|^2_F & = \|(X + \omega \Delta X) (Y + \omega \Delta Y) - X Y \|^2_F \\
  & = \|\omega (\Delta X Y + X \Delta Y) \|^2_F  + \left(\mathcal{O}(\omega^2)\right)^2 \label{varresder}
\end{align}
The term $\| \Delta X Y + X \Delta Y \|^2_F$ can be approximated, assuming 
the angles between the three summands are small, 
by $\| \Delta X Y \|^2_F + \| X \Delta Y \|^2_F$. This and Lemma \ref{averobj} then motivate the following definition.
\begin{definition}[Variational residual function]\label{sicorefuinmosmo2matrix}
  Let $\omega \geq 0$, $M$ the target matrix, $P$ the sampling set and $A = X Y$ the
  current iterate. 
  We define the variational residual function $C := C_{M,P,X,Y}: \mathcal{D}_r \rightarrow \R$ by
  \begin{align} \label{corfmatrix}
    C(\widetilde{X},\widetilde{Y}) & := \int_{\mathbb{V}_{\omega}(X,Y)} \| (\widetilde{X} + \Delta X) (\widetilde{Y} + \Delta Y) - M\|^2_{P} \intd   \Delta X \intd  \Delta Y,\\
%
  \mathbb{V}_{\omega}(X,Y) & := \{(\Delta X,\Delta Y) \mid \|\Delta X Y\|^2_F + \|X \Delta Y\|^2_F \leq \omega^2 \}. \nonumber
  \end{align}
\end{definition}

\subsection{Minimizer of the Variational Residual Function for Matrices}
We define our modified methods as
\begin{align} \label{methodX} (\mathcal{M}^{(1)})^{\ast}(X, \ Y) 
  & := (\argmin_{\widetilde{X}}\ C(\widetilde{X},Y), \ Y) \\
  \label{methodY} (\mathcal{M}^{(2)})^{\ast}(X, \ Y) 
  & := (X, \ \argmin_{\widetilde{Y}}\ C(X,\widetilde{Y}))
\end{align}
with $C =C_{M,P,X,Y}$ as in \eqref{corfmatrix}. 
It should further be noted that $\mathbb{V}_{\omega}$ does not depend on the unknown $\widetilde{X}$, $\widetilde{Y}$, respectively, but on the current iterate.
%
We will later see that the minimizers are unique, but for formality we again use the minimization of Frobenius norm of the
iterate as secondary criterion.
\begin{lemma}
 The two methods $(\mathcal{M}^{(1)})^{\ast}$ \eqref{methodX} and $(\mathcal{M}^{(1)})^{\ast}$ \eqref{methodY} are representation independent.
\end{lemma}
\begin{proof}
 We later prove this for the generalized tensor case (cf. Lemma \ref{in}).
\end{proof}
%
Here, and throughout the remainder of the paper, we use the following, convenient notations, since we often have to reshape, restrict or project objects.
\begin{definition}[Restrictions]\label{restrictions}\BCn{o2}{o}
  For any object $A \in \R^{I}$ and index set $S \subset I$, we 
  use $A|_S \in \R^S$ as restriction.
  For a matrix $M$, let $M_{:,i}$ be its $i$-th column and $M_{i,:}$ be its $i$-th row.
  Furthermore, whenever we apply a restriction to an object or reshape it, we also use the same notation
 to correspondingly modify index sets (cf. Theorem \ref{miofthalsrerefufoma}).
\end{definition}
\begin{lemma}[Integral over all variations]\label{inovalva}
  Let $n,m \in \mathbb{N}$, $\omega \geq 0$ and $H \in \R^{n \times n}$ be a matrix as well as 
  \[ V_{\omega}^{(n,m)} := \{ X \in \R^{n \times m} \mid \|X\|_F = \omega\}. \]
  Then
  \[ \int_{V_{\omega}^{(n,m)}} X^T H X \intd X =  \frac{\omega^2|V_{\omega}^{(n,m)}|}{nm} \mbox{tr\newwo{ace}}(H)  I_m, \quad |V_{\omega}^{(n,m)}| := \int_{V_{\omega}^{(n,m)}} 1. \]
\end{lemma}%
\begin{proof}\BCn{g1_}{g_}Let $Y$ be the result of the above integral. Then
  \begin{align*}
   Y_{ij} & = \mathrm{trace}(Y_{ij}) = \int_{V_{\omega}^{(n,m)}} \mathrm{trace}(X_{:,i}^T H X_{:,j}) \intd X 
    = \mathrm{trace}(H \int_{V_{\omega}^{(n,m)}} X_{:,j} X_{:,i}^T) \intd X. 
  \end{align*}
  Due to symmetry, for some $\alpha \in \R$, we have 
  \begin{align*}
   \int_{V_{\omega}^{(n,m)}} \mathrm{vec}(X) \mathrm{vec}(X)^T  \intd X & = \alpha I_{nm}, \quad
   \int_{V_{\omega}^{(n,m)}} \mathrm{vec}(X)^T \mathrm{vec}(X) \intd X = \omega^2 |V_{\omega}^{(n,m)}|.
  \end{align*}
  Since the second term is the trace of the first one, it follows that $\alpha = \omega^2 |V_{\omega}^{(n,m)}|/(nm)$. We can hence simplify 
  \begin{align*}
   Y_{ij} = \begin{cases}
             \omega^2 |V_{\omega}^{(n,m)}|/(nm) \ \mathrm{trace}(H) & \mbox{ if } i = j, \\
             0 & \mbox{ otherwise }, \\
            \end{cases}
  \end{align*}
  which is the to be proven statement.
\end{proof}
We now derive the minimizer of the variational residual function for matrices \eqref{corfmatrix}.
One can use an SVD of the current iterate $A = U \Sigma V^T$ for simplification. 
In this case, $\mathbb{V}_{\omega}$ as in Definition \ref{sicorefuinmosmo2matrix} takes the easier forms
\begin{align*} \mathbb{V}_{\omega}(U \ \Sigma, \ V^T) & = \{(\Delta X,\Delta Y) \mid 
  \|\Delta X V^T\|^2_F + \|U \Sigma \Delta Y\|^2_F \leq \omega^2 \} \\
  & = \{(\Delta X,\Delta Y) \mid \|\Delta X\|^2_F + \|\Sigma \Delta Y\|^2_F \leq \omega^2 \}, \\
  \mathbb{V}_{\omega}(U, \ \Sigma \ V^T) & = \{(\Delta X,\Delta Y) \mid 
  \|\Delta X \Sigma\|^2_F + \|\Delta Y\|^2_F \leq \omega^2 \}.
\end{align*}

\begin{theorem}[Minimizer of the ALS variational residual function for matrices]\label{miofthalsrerefufoma}
  \new{8}{Let $A \in \R^{n \times m}$ be the current iterate with $r = \mathrm{rank}(A)$ and let $U \Sigma V^T$ be a truncated SVD, $U \in \R^{n \times r}$, $\Sigma \in \R^{r \times r}$, $V^T \in \R^{r \times m}$, of $A$.
  Let further $C$ be the variational residual function as in \eqref{corfmatrix}.}\BCn{h2}{h}%
  \old{8}{Let $U \Sigma V^T$ be the SVD of a matrix $A \in \R^{n \times m}$ and $C$ the variational residual function as in \eqref{corfmatrix}.}
  The minimizer $X^+$ of $\widetilde{X} \mapsto C_{M,P,U \Sigma, V^T}(\widetilde{X},V^T)$  is given by
  \begin{align*}
  X^+_{i,:} = \argmin_{\widetilde{X}_{i,:}} \ \underbrace{\|\widetilde{X}_{i,:} V^T - M_{i,:} \|_{P_{i,:}}^2}_{\mbox{standard ALS}}  & + \underbrace{\frac{|{P_{i,:}}|}{m} \ \omega^2 \ \zeta_2 \ \|\widetilde{X}_{i,:} \Sigma^{-1}\|_F^2}_{\mbox{regularization}},
  \end{align*}
  where $P_{i,:} := \{ p_2^{(k)} \mid p_1^{(k)} = i, \ k = 1,\ldots,|P|\}$ is the corresponding part of the index set $P$.
    The minimizer $Y^+$ of $\widetilde{Y} \mapsto C_{M,P,U, \Sigma V^T}(U,\widetilde{Y})$ is given by
  \begin{align*}
    Y^+_{:,j} = \argmin_{\widetilde{Y}_{:,j}} \ \ \underbrace{\|U \ \widetilde{Y}_{:,j} - M_{:,j}\|^2_{P_{:,j}}}_{\mbox{standard ALS}} 
    + \underbrace{\frac{|P_{:,j}|}{n} \ \omega^2 \ \zeta_1 \ \|\Sigma^{-1} \ \widetilde{Y}_{:,j}\|^2_F}_{\mbox{regularization}},
  \end{align*}
  where $P_{:,j} := \{ p_1^{(k)} \mid p_2^{(k)} = j, \ k = 1,\ldots,|P|\}$.
  The constants $\zeta_{1}$ and $\zeta_2$ only depend on the proportions of the representation and sampling set (cf. Remark \ref{constants_matrices}).
\end{theorem} 
\moved{30}{
The factors $\frac{|{P_{i,:}}|}{m}$ and $\frac{|P_{:,j}|}{n}$ normalize the penalty terms 
to the particular magnitudes of the corresponding shares of the sampling set $P$ and hence the standard ALS part. The factors $\zeta_1,\zeta_2 \in (0,1)$ are in turn
independent of $i,j$, respectively. The ratio of both terms equals the ratio of the mode sizes $n,m$.
The reason for the latter scaling will become apparent in the tensor case (see comments to Theorem \ref{miofthalsrerefu}).}
\begin{proof}\BCn{f1_}{f_}%
\old{30}{The proof is rather technical and can be found in Appendix A.}
\moved{30}{
We search for $Y^+ := \argmin_{\widetilde{Y}} C_{M, P, U, \Sigma V^T}(U,\widetilde{Y})$
(the counterpart for $X^+$ is analogous).
Substituting
\[ 
 (\Delta X,\Delta Y) \rightarrow (\Delta X \Sigma^{-1}, \Delta Y)
\]
we can (up to a constant Jacobi determinant) restate $C$ as
\begin{align}
C_{M, P, U, \Sigma V^T}(U,\widetilde{Y}) & \ \propto \int_{\mathbb{V}_{\omega}} \|
(U +  \Delta X \Sigma^{-1}) (\widetilde{Y} + \Delta Y)  
\nonumber  - M\|^2_{P} \intd \Delta X \intd \Delta Y, \\
\mathbb{V}_{\omega} & = \{(\Delta X,\Delta Y) \mid 
\|\Delta X \|^2 + \|\Delta Y \|^2\leq \omega^2 \}.
\end{align}
Each of the independent columns in the minimizer is restated as
\begin{align}
 Y^+_{:,j} & = \argmin_{\widetilde{Y}_{:,j}} \int_{\mathbb{V}_\omega}
 \| (U +  \Delta X \Sigma^{-1}) (\widetilde{Y}_{:,j} +  \Delta \widetilde{Y}_{:,j})  - M_{:,j} \|^2_{P_{:,j}} \intd  \Delta X \intd \Delta Y. 
\end{align}
Let $j$ be arbitrary but fixed from now on.
For any vector $x$, it is $\|x\|_{\mbox{vec}(P_{:,j})} = \|H(j) x\|_{F} = x^T H(j) x$ for a diagonal, 
square matrix $H(j) \in \R^{n \times n}$ 
with \\ $H(j)_{(s),(s)} = \delta_{s \in P_{:,j}}$ (hence $H(j)^2 = H(j)$).
Using the normal equation, we obtain $Y^+_{:,j} = W^{-1} b$, where
\begin{align*} 
W = & \int_{\mathbb{V}_\omega} (U +  \Delta X \Sigma^{-1})^T \ H(j) \ (U +  \Delta X \Sigma^{-1}) \intd \Delta X \intd \Delta Y
\end{align*}
and
\begin{align*} 
b = & \int_{\mathbb{V}_\omega} (U +  \Delta X \Sigma^{-1})^T \ H(j) \ (M_{:,j} -  \Delta \widetilde{Y}_{:,j}) \intd \Delta X \intd \Delta Y.
\end{align*}
In both $W$ and $b$, any perturbation that appears only one-sided in the expanded product vanishes due to symmetry of $\mathbb{V}_\omega$. 
Thus $b = |\mathbb{V}_\omega| \ U_{P_{:,j},:}^T \ M_{P_{:,j},j}$
and
\begin{align*} W = & \int_{\mathbb{V}_\omega} U^T \ H(j) \ U \intd \Delta X \intd \Delta Y \\
 + & \int_{\mathbb{V}_\omega} (\Delta X \Sigma^{-1})^T \ H(j) \ (\Delta X \Sigma^{-1}) \intd \Delta X \intd \Delta Y
\end{align*}
Now, let $\ell = \#_{X} = nr,\ k = \#_{Y} = rm$. Since $\mathbb{V}$ is a version of the $(\ell + k)$-sphere,
we can use the following integration formula: 
For any $n,m,k \in \mathbb{N}$ let $f: \R^{n+m} \rightarrow \R^k$ be a sufficiently smooth function and $S^{v-1}_{\omega}$ be the $v$-sphere of radius $\omega$.
 Then
 \[ \int_{S^{n+m-1}_{\omega}} f(x_n,x_m) \intd x = \int_0^{\pi/2} \omega \int_{S^{n-1}_{\omega \sin(u)}} \int_{S^{m-1}_{\omega \cos(u)}} f(x_n,x_m) \intd  x_m \intd  x_n \intd  u. \]
For a function $f$ we then obtain
\begin{align*}
 \int_{\mathbb{V}} f\intd \Delta X \intd \Delta Y & = 
 \int_{\lambda = 0}^{\omega} \int_{S_{\lambda}^{\ell + k-1}} f\intd \Delta X \intd \Delta Y \intd \lambda \\
 & = \int_{\lambda = 0}^{\omega} \lambda \int_{g = 0}^{\pi/2} \int_{S_{\lambda \sin(g)}^{k -1}} \int_{S_{\lambda \cos(g)}^{\ell -1}} f\intd \Delta X \intd \Delta Y \intd g \intd \lambda
\end{align*}
When $f$ is independent of $\Delta Y$, this simplifies to
 \begin{align*} 
 = \int_{\lambda = 0}^{\omega} \lambda \int_{g = 0}^{\pi/2} |S_{\lambda \sin(g)}^{k -1}| \int_{S_{\lambda \cos(g)}^{\ell-1}} f\intd \Delta X \intd g \intd \lambda
\end{align*}
We further use the identity 
\[ \int_0^{\pi/2} \cos(x)^p \sin(x)^q \intd x = \frac{\Gamma((p+1)/2) \ \Gamma((q+1)/2)}{2 \Gamma((p+q+2)/2)} =: \nu(p,q) \]
We apply these and Lemma \ref{inovalva} for different $f_0 = U^T \ H(j) \ U$ ($\delta = 0$) and $f_1 = (\Delta X \Sigma^{-1})^T \ H(j) \ (\Delta X \Sigma^{-1})$ ($\delta = 1$).
For the summands $W(0) + W(1) = W$ this then yields
\begin{align*} W(\delta) = &
\int_{\lambda = 0}^{\omega} \lambda \int_{g=0}^{\pi/2} \frac{2 \pi^{k/2} (\lambda \sin(g))^{k-1}}{\Gamma(k/2)}
(\lambda^2 \cos(g))^{\delta} \frac{2 \pi^{\ell/2} (\lambda \cos(g))^{\ell-1}}{\Gamma(\ell/2)} C_H(\delta) \\
& = c \frac{\omega^{k+\ell+2\delta}}{k+\ell+2\delta} \ \nu(\ell-1+2\delta,k-1) C_H(\delta)
\end{align*}
for $c = \frac{4 \pi^{(k+\ell)/2}}{\Gamma(\ell/2)\Gamma(k/2)}$. The constant matrices $C_H$ are given by
\begin{align*}
 C_H(0) & =  \widetilde{C}_H(0) = U_{P_{:,j},:}^T \ _{P_{:,j},:} \\
 |P_{:,j}|^{-1} n r C_H(1) & = \widetilde{C}_H(1) = \Sigma^{-2}
\end{align*} 
Furthermore, it is $|\mathbb{V}_\omega| = c \frac{\omega^{\ell+k}}{\ell+k} \nu(\ell-1,k-1)$. 
Factoring out this base volume in $W = |\mathbb{V}_\omega| \widetilde{W}$
by using properties of the $\Gamma$ function, one derives:
\begin{align*}
 \widetilde{W}(0) = \widetilde{C}_H(0),\  \widetilde{W}(1) = \frac{|P_{:,j}|}{n} \omega^2 \zeta_{1} \widetilde{C}_H(1), \quad \zeta_{1} = \frac{\ell}{r (k + \ell + 2)}
\end{align*}
Restating the result again as a least squares problem finishes the proof.}
\end{proof}

\begin{remark}[Specification of constants]\label{constants_matrices}
 Let $\#_X := nr,\ \#_{Y} := rm$ be the sizes of the components in the matrix decomposition.
  The constants in Theorem \ref{miofthalsrerefufoma} are given by 
  \begin{align*}
  \zeta_2 = \frac{\#_{Y}}{r (\#_X + \#_{Y} + 2)}, \quad \zeta_1 = \frac{\#_X}{r (\#_X + \#_{Y} + 2)}.
  \end{align*}
  However, when changing the rank, this would impose a slight offset in continuity of both $f_{{\mathcal{M}^{(1)}}^{\ast}}$ and $f_{{\mathcal{M}^{(2)}}^{\ast}}$.
  This problem is simply resolved by substituting $\omega$ by $\widetilde{\omega}$ properly for each value $r$, such that we can set and normalize
  \begin{align*}
  \omega^2 \zeta_2 = \widetilde{\omega}^2 \frac{m}{n+m}, \quad \omega^2 \zeta_1 = \widetilde{\omega}^2 \frac{n}{n+m}.
  \end{align*} 
  We will still just write $\omega$ although we replace $\zeta_1$ and $\zeta_2$ by the adapted values.
\end{remark}

Manipulating $\omega$ does not change the representation independency of the two
methods $(\mathcal{M}^{(1)})^{\ast}$ and
$(\mathcal{M}^{(2)})^{\ast}$, since for fixed $r$,\newwo{\ the value} $\omega > 0$ just remains an arbitrary constant (cf. Lemma \ref{in}).
We arrive at the main result for the matrix case:

\begin{theorem}[Stability of the variational matrix methods]\label{stablematrixmethods}
 The methods $(\mathcal{M}^{(1)})^{\ast}$ \eqref{methodX} and $(\mathcal{M}^{(2)})^{\ast}$ \eqref{methodY}
 are stable. 
\end{theorem}
\begin{proof}
 Follows as special case from the proof of the tensor version, Theorem \ref{coofthfiadfmist}.
\end{proof}

\old{9}{
The validity of the theorem boils down to one property. Consider a very small last singular value
$\sigma_r$. Then the corresponding parts of the minimizer are heavily penalized. If the value
converges to zero, the penalty tends to $+\infty$, which in limit leads to the same result as if
the representation had been truncated beforehand. The systematic discontinuity as in Example \ref{alsanadfarnocowitr} does no
longer occur. The actual proof is quite technical, yet this remains the base argument.
}\BCn{i1}{i}%
Despite the technicalities involved in the proof, the simplicity of the idea becomes apparent by setting $P = \mathcal{I}$ (being
analogous to the argumentation in section \ref{sec:conect}).
For a certain constant $c \in \R$, in the setting of Theorem \ref{miofthalsrerefufoma}, we then have
\begin{align}
 f_{{\mathcal{M}^{(2)}}^{\ast}}(U \Sigma V^T) & \quad = \quad U \quad \cdot \quad \underbrace{(I + c \Sigma^{-2})^{-1}}_{\mbox{regularization}} \quad \cdot \quad \underbrace{\vphantom{(I + c \Sigma^{-2})^{-1}} U^T \ M}_{\mbox{standard ALS}} \label{filterex} \\
 & \quad = \quad (\ (1+c\sigma_1^{-2})^{-1} \cdot U_{:,1} U_{:,1}^T M,\ \ldots,\ (1+c\sigma_r^{-2})^{-1} \cdot U_{:,r} U_{:,r}^T M\ ). \nonumber
\end{align}
If now $\sigma_r \rightarrow 0$, then also $(1+c\sigma_r^{-2})^{-1} \rightarrow 0$ and we obtain the same result as if we would have truncated the representation $(X,Y) = (U,\Sigma V^T)$ to rank $r-1$ beforehand.
We also denote these additional factors as \textit{filter}, as they filter out influence corresponding to low singular values.
Note that obtaining small singular values is not penalized, but using components corresponding to small ones is.
\\\\
Algorithm \ref{MatrixCompletion} performs one stable ALS approximation (SALSA) sweep, that is it applies 
the two stable methods from Theorem \ref{stablematrixmethods}. \old{13}{To substitute the perished replenishment term \eqref{conv}
we limit all singular values from below by $\sigma_{\mathrm{min}} > 0$ (cf. Remark \ref{repllowerlimit}, Appendix B),}\BCn{k1}{k}%
\new{13}{Before each update, any current singular value $\sigma_i < \sigma_{\mathrm{min}}$ is replaced by $\sigma_{\mathrm{min}}$,
which in turn is set as fraction $f_{\min} \ll 1$ of the current residual (Algorithm \ref{CALS_alg_matrix}). 
Although the influence of this manipulation on the subsequent step is thereby marginal, it
is necessary since otherwise $\sigma_i$ may irreversibly converge to zero (cf. Remark \ref{repllowerlimit}, for more details, see Section \ref{sec:filter}).
}

\begin{algorithm}
  \caption{Stable Matrix Completion \label{MatrixCompletion}}
  \begin{algorithmic}[1]
    \REQUIRE limit $\sigma_{\mathrm{min}}$, parameter $\omega$, initial guess $A = X Y \in \R^{n \times m}$ such that \old{13}{$Y$ contains the right}\new{13}{$X$ contains the left} singular vectors of $A$, and data points $M|_P$
    \STATE \new{13}{compute the SVD $U \Sigma V^T := Y$ and update $\sigma_i := \mbox{max}(\sigma_i,\sigma_{\mathrm{min}}),\ i = 1,\ldots,r$}
    \STATE \new{13}{set $X := X U \Sigma$ and $Y := V^T$}
    \STATE for $i = 1,\ldots,n$ update
        \begin{align} \label{stmethX}
     X_{i,:} := \argmin_{\widetilde{X}_{i,:}} \|\widetilde{X}_{i,:} Y - M_{i,:} \|_{P_{i,:}}^2 & + \frac{|{P_{i,:}}|}{m} \frac{\omega^2 m}{n + m} \|\widetilde{X}_{i,:} \Sigma^{-1}\|_F^2
    \end{align}
    \oldSTATE{13}{compute the SVD $U \Sigma V^T := X$ and update $\sigma_i := \mbox{max}(\widetilde{\sigma}_i,\sigma_{\mathrm{min}}),\ i = 1,\ldots,r$}
    \oldSTATE{13}{set $X := U$ and $Y := \Sigma V^T Y$}
     \STATE \new{13}{compute the SVD $U \Sigma V^T := X$ and update $\sigma_i := \mbox{max}(\sigma_i,\sigma_{\mathrm{min}}),\ i = 1,\ldots,r$}
    \STATE \new{13}{set $X := U$ and $Y := \Sigma V^T Y$}    
     \STATE for $j = 1,\ldots,m$ update
    \begin{align} \label{stmethY}
     Y_{:,j} := \argmin_{\widetilde{Y}_{:,j}} \|X \widetilde{Y}_{:,j} - M_{:,j} \|_{P_{:,j}}^2 & + \frac{|{P_{:,j}}|}{n} \frac{\omega^2 n}{n + m} \|\Sigma^{-1} \widetilde{Y}_{:,j}\|_F^2
    \end{align}
   \oldSTATE{13}{compute the SVD $U \Sigma V^T := Y$ and update $\sigma_i := \mbox{max}(\widetilde{\sigma}_i,\sigma_{\mathrm{min}}),\ i = 1,\ldots,r$}
    \oldSTATE{13}{set $X := X U \Sigma$ and $Y := V^T$} 
  \end{algorithmic}
\end{algorithm}
\subsection{Rank Adaption}\label{sec:matrixra}
%
%
The key aspect of stability is that it rendered explicit rank adaption near unnecessary, since the optimization relies on \BCn{k2}{k}%
the magnitude of singular values, as \eqref{filterex} suggests. Starting from an initial representation \BCn{l1}{l}%
and a large value $\omega$ proportional to the norm of the iterate $A$, the parameter $\omega$ is decreased after each iteration. The singular values are decided into two types:
\begin{definition}[Stabilized rank and minor singular values]\label{virtstabsv}
A singular value $\sigma_i$ is called
\textit{stabilized}, if it is larger than a certain fixed fraction of $\omega$ (meaning any corresponding terms have an increased influence, cf. \eqref{filterex}).
Otherwise, it is called minor (as a removal of such does
not notably change the next steps).
The \textit{stabilized rank} only counts the number of stabilized singular values.
\end{definition}
The actual rank is only modified as to make sure that there is always a fixed, small amount of minor singular values, i.e.\BCn{k3}{k}%
\begin{align}
 |\{ i \mid 0 < \sigma_i < f_{\mathrm{minor}} \cdot \omega \}| \overset{!}{=} k_{\mathrm{minor}}, \label{matrixrankincr}
\end{align}
for constants $f_{\mathrm{minor}} < 1$ and $k_{\mathrm{minor}} \in \mathbb{N}$.
This states the basic concept of implicit rank adaption and we will provide a more detailed discussion in the later section \ref{sec:seimannounraad} for the elaborate tensor case.
For performance evaluation, a validation set may be used:
\movedwo{
\begin{definition}[\oldwo{Control}\newwo{Validation} set]\label{cose}
For a given \oldwo{index set }$P$\newwo{, the sampling or training set}, we define $P_2 \subset P$ as \oldwo{control}\newwo{validation} set. This set may be chosen
randomly or specifically distributed. The actual set used for the optimization is replaced by $P \leftarrow P \setminus P_2$ (keeping
the same symbol). 
\end{definition}}
%
%
%
\begin{algorithm}
\caption{SALSA Algorithm \label{CALS_alg_matrix}}
 \oldjg{13}{
  \begin{algorithmic}[1]  \label{ra_matrix}
  \REQUIRE $P \subset \Index$, $M|_P$, $r_{\mathrm{start}}$ (and parameters)
  \STATE initialize $X, Y$ s.t. $|P| \|\tau_r(X,Y)\|_F^2 = |\Ind|\|M|_P\|_P^2$ for $r = r_{\mathrm{start}}$
  \STATE set $\omega = 1/2 \|\tau_r(X,Y)\|_F^2$
  \STATE optional: split off a small control set $P_2 \subset P$ for performance evaluation
  \FOR{${\tt iter}=1,2,\ldots$}
    \STATE proceed SALSA sweep$^{\ast}$ (Algorithm \ref{MatrixCompletion})   
    \STATE$^{\ast}$: and adapt lower limit $\sigma_{\mathrm{min}}$
    \STATE$^{\ast}$: and decrease $\omega$ if residual decreases too slow
    \IF{$^{\ast}$: $\sigma_r$ becomes stabilized}
      \STATE increase rank using a marginal singular value
    \ENDIF
    \IF{final breaking criteria apply}
      \STATE terminate algorithm
    \ENDIF
  \ENDFOR
  \end{algorithmic}
  }
  \new{13}{
  \begin{algorithmic}[1]  \label{ra_matrix}
  \REQUIRE $P \subset \Index$, $M|_P$
  \STATE initialize $X, Y$ s.t. $\tau_r(X,Y) \equiv \|M|_P\|_1/|P|$ for $r \equiv 1$ and $\omega = 1/2 \cdot \|\tau_r(X,Y)\|_F$
  \STATE split off a small validation set $P_2 \subset P$ for performance evaluation
  \FOR{${\tt iter}=1,2,\ldots$}
    \STATE proceed SALSA sweep$^{\ast}$ (Algorithm \ref{MatrixCompletion})   
    \STATE$^{\ast}$: and renew lower limit $\sigma_{\mathrm{min}} := f_{\min} \cdot \frac{|\Index|}{|P|} \|\tau_r(X,Y) - M\|_P$
    \STATE$^{\ast}$: and decrease $\omega$
    \STATE adapt rank according to \eqref{matrixrankincr} (start this when the first few iteration have passed)
    \IF{a stopping criterion applies}
      \STATE terminate algorithm
      \STATE \textbf{return} iterate for which $\|\tau_r(X,Y) - M\|_{P_2}$ was lowest
    \ENDIF
  \ENDFOR
  \end{algorithmic} }
\end{algorithm}
The \old{45}{final breaking}\new{45}{stopping} criteria in Algorithm \ref{CALS_alg_matrix} may depend on the
behavior of $P_2$, or may simply be based on a rank bound, e.g. $r \leq |P|/(n+m)$. The latter criterion, however, only suffices in the matrix case.
\old{13}{Note that through the stability, a rank overestimation is not harmful, albeit increases the computational complexity.}
\begin{remark}[On convergence estimates]\BCn{j11}{j}%
Due to the specific dependency of the regularization on the current singular values of the iterate,
it may be impossible to restate the problem as minimization of a modified, global cost function.
Furthermore, the iterate does not remain on a fixed manifold of low rank. Given that also rank deficient points may be included in the optimization
due to the stability, theoretical convergence results so far remain subject to future work.
\end{remark}
\section{Generalization to High-Dimensional Tensors}\label{sec:gentohighd}
For the rest of this article we consider the problem of (approximately) reconstructing a \textit{tensor} $M \in \R^{\Ind}$
from a given data set $M|_P = \{M_p\}_{p \in P}$, where now \old{3}{$P \subset \Ind$}\BCn{c2}{c}%
\begin{align}
P \subset \Ind := \bigtimes_{\mu = 1}^d \mathcal{I}_\mu, \quad \mathcal{I}_\mu := \{1,\ldots,n_\mu\}, \ \mu = 1,\ldots,d.
\end{align}
We further generalize the 
representation map and data space $\tau_r: \mathcal{D}_r \rightarrow \mathcal{D}_r$ as well as the rank $r$ together
with the spaces $\mathcal{T}_r$ to the tensor train (TT-)format (Definition \ref{TTformat}, \cite{OsTy09_Bre,Os11_Ten}, also called Matrix Product States
(MPS) \cite{Wh92_Den,Vi03_Eff} or interpreted as a special case of the Hierarchical Tucker format \cite{HaKu09_ANe,Gr10_Hie}).
Stability for tensor methods is thereby defined through Definition \ref{cowitr} as well. 
For the underlying tensor $M$ it is now assumed that its TT-singular values (cf. Definition \ref{ttsiva}) decline sufficiently fast
in order to yield an approximation $M_{\varepsilon}\in \mathcal{T}_r$.
\begin{definition}[TT-singular values and TT-rank]\label{ttsiva}
  Analogously to the TT-ranks, we define the TT-singular values $\sigma=(\sigma^{(1)},\ldots,\sigma^{(d-1)})$ of a tensor $A$ 
  as the unique singular values of the corresponding matricizations 
  $A^{(\newwo{\{}1,\ldots,\mu\newwo{\}})}\in\mathbb{R}^{(\Ind_1\times\cdots\times \Ind_{\mu}) \times (\Ind_{\mu+1}\times\cdots\times \Ind_{d})}$ with entries
  \[A^{(\newwo{\{}1,\ldots,\mu\newwo{\}})}((i_1,\ldots,i_\mu),\, (i_{\mu+1},\ldots,i_d)) := A(i)\]
  of $A$, such that $\sigma^{(\mu)}$ contains the ordered singular values of $A^{(\newwo{\{}1,\ldots,\mu\newwo{\}})}$, $\mu = 1,\ldots,d-1$. 
  The TT-rank $r_\mu$ is the number of nonzero TT-singular values in $\sigma^{(\mu)}$. 
  We also call $\sigma^{(\mu)}$ the $\mu$-th singular values.
\end{definition}
\begin{definition}[TT-format]\label{TTformat} A tensor $A\in\mathbb{R}^{\Ind}$ is in the set $\new{14}{\mathcal{T}_r}$, often also denoted $\mathrm{TT}(r)$, $r\in\mathbb{N}^{d-1}$\BCn{n1}{n},
  if for $\mu=1,\ldots,d$ and $i_\mu\in\Index_{\mu}$ there exist 
  $G_\mu(i_\mu)\in\mathbb{R}^{r_{\mu-1}\times r_\mu}$\, ($r_0=r_d=1$) such that\BCn{r1}{r}
  \[ \new{18}{A = \tau_r(G),} \quad A(i_1,\ldots,i_d) := G_1(i_1)\cdots G_d(i_d),\qquad i\in\Ind.\]
  The representation of $A$ in this form is shortly called the TT format. \new{14}{The single $G_\mu$, as well as
  similarly structured objects, are called cores (or sometimes nodes)}.
  \old{14}{If we want to stress the dependency
  of $A$ on the so-called cores $G_\mu$ then we write  
  $A = \tau_r(G) := G_1 \ck \cdots \ck G_d$, 
  where we define $\ck$ for two cores 
  $H_1, H_2$ as $(H_1 \ck H_2)((i,j)) := H_1(i) \ H_2(j)$ (interpreting TT-cores as vectors of matrices).}
\end{definition}
The TT-SVD \cite{Os11_Ten} provides that if $r$ is the TT-rank of $A$, then it follows $A \in \mathcal{T}_r$. 
So $G$ is also called low-rank representation of $A$.
  The function $\tau_r$\BCn{o1}{o} will appear in different forms, each denoting the mapping from a type of representation of rank $r$ \BCn{r2}{r}%
  to the represented object.
  \begin{definition}[Unfoldings]\label{unfoldings}
\moved{14}{
  For a core $H$ (possibly a product of smaller cores in the TT representation) with $H(i) \in \R^{k_1 \times k_2}$, $i = 1,\ldots,n$,
  we denote the left unfolding $\lhb{H}\in\mathbb{R}^{k_1 \cdot n \times k_2}$, in which the matrices contained in the core are stacked below each other, and right unfolding  $\rhb{H}\in\mathbb{R}^{k_1 \times k_2 \cdot n}$,
  in which they are stacked side by side, by
  \begin{align*}
    \left(\lhb{H}\right)_{(\ell,j), q}:= \left(H(j)\right)_{\ell,q}, \quad \left(\rhb{H}\right)_{\ell,(q,j)}
    &:= \left(H(j)\right)_{\ell,q},\quad 
  \end{align*}
  for $1\le j\le n$, $1\le\ell\le k_1$ and $1\le q\le k_2$. }
  \end{definition}
Our goal is to determine the ranks adaptively.
We will demonstrate why this can be even more troublesome than in the matrix case in the following.\\
%
%
%
\new{16,46}{For many tensors that stem from practical applications one observes exponentially decaying singular values, %
yet the rate of decay may strongly vary for the $d-1$ different matricizations. 
Theoretically, there is no non-trivial limitation to the shape of the TT-singular values as described in following lemma.}
\begin{lemma}[\newwo{Feasibility of TT-singular values \cite{GrKrxx_The}}]\label{TTfeas}\BCn{p1}{p}\BCn{w1_}{w_}
Let $\sigma = (\sigma^{(1)},\ldots,\sigma^{(d-1)})$ be a $(d-1)$-tuple of weakly decreasing $r_\mu$-tuples, $\mu = 1,\ldots,d-1$,
for which $\|\sigma^{(i)}\|_2 = \|\sigma^{(j)}\|_2$ for all $i,j = 1,\ldots,d-1$.
Then there exist mode sizes $n_1,\ldots,n_d \in \mathbb{N}$ and a tensor $A \in \R^{\mathcal{I}}$ such that this tensor has TT-singular values $\sigma$.
%
\end{lemma}
\new{16,46}{The following exemplary tensor emphasizes the problems heuristics encounter in the rank adaption
of a tensor with inconvenient singular values.}
\begin{example}[Rank adaption test tensor]\label{raadtete}
  For $k \in \mathbb{N}$, let $Q \in \R^{n_1 \times \ldots \times n_4}$ be an orthogonally decomposable $4$-dimensional Tensor with
  TT-rank $(k,k,k)$ and uniform singular values $\sigma^{(1)}=\sigma^{(2)}=\sigma^{(3)} = (\alpha,\alpha,\ldots)$ 
  as well as $B \in \R^{n_5 \times n_6}$ be a rank $2k$ matrix with exponentially decaying singular values $\sigma^{(5)} \propto (\beta^{-1},\beta^{-2},\ldots)$ for some $\alpha,\beta > 0$.
  Then the separable tensor $M \in \R^{n_1 \times \ldots \times n_6}$ defined by $M(i) = Q(i_1,\ldots,i_4) \cdot B(i_5,i_6)$
  has singular values $\sigma$ and rank $r^{(M)} = (k,k,k,1,2k)$.
\end{example}
By its definition, $M$ is separable into a $4$- and a $2$-dimensional tensor ($Q$, $B$).
Knowing this would of course drastically simplify the problem. 
We now consider the performance of two very basic rank adaption ideas.
\begin{enumerate}
  \item \textit{Greedy, single rank increase}: We test for maximal improvement by
  increase of one of the ranks $r_\mu$ ($\mu=1,\ldots,d-1$) of the iterate starting from $r \equiv 1$. 
  Solely increasing either of $r_2$, $r_3$ or $r_4$ will give close to no improvement.
  As further shown in \cite{EsKh15_Con},
  the approximation of orthogonally decomposable tensors with lower rank can be problematic.
  In numerical tests \new{17}{(cf. Section \ref{sec:recratt})}, we can observe that $r_5$ is often increased to a maximum first. 
  Thereby, extremely small singular values are involved that lie far beneath the current 
  approximation error, although the rank is not actually overestimated. 
  \item \textit{Uniform rank increase and coarsening}: We increase every rank  
  $r_\mu$ ($\mu=1,\ldots,d-1$) starting from $r \equiv 1$ and decrease ranks when the corresponding
  singular values are below a threshold.
  The problem with this strategy is quite \oldwo{obvious}\newwo{plain}, namely that
  \new{17}{for the target tensor $M$, it holds} $r^{(M)}_5 = 1$\BCn{q1}{q}. If this rank is overestimated, the observed sampling points will be misinterpreted
  (overfitting) and it does not matter how small corresponding singular values become 
  (see Lemma \ref{alsanadfarnocowitr}). 
\end{enumerate}
These indicated difficulties gain more importance with high dimension, but for one microstep at a time, can
be resolved by regarding only three components of a tensor. We describe this in the following Section \ref{sec:3d}.\\\\
  \newwo{The tensor from Example \ref{raadtete} can also be constructed explicitly.}
  \moved{15}{\BCn{f2_}{f_}%
We define a representation $G$ for $A = \tau_r(G)$ via left and right unfoldings by
\begin{align*}
 \lhb{G_1} & := Q_1, \\
 G_2(i_2)  & = G_3(i_3) := I_k, \qquad 1\le i_2 \le n_2,\quad 1\le i_3 \le n_3, \\
 \rhb{G_4} & := Q_4^T, \ \lhb{G_5} := Q_5,\ \rhb{G_6} := \Sigma_5 Q_6^T,
\end{align*}
for (column-) orthogonal matrices $Q_1 \in \R^{n_1 \times r_1}, \ Q_4 \in \R^{n_4 r_4 \times r_3}, \ Q_5 \in \R^{r_4 n_5 \times r_5}
\ Q_6 \in \R^{n_6 \times r_5}$ and $(\sigma_5)_i \propto \beta^{-i}$, $\beta > 1$. 
This tensor has exactly the properties postulated in the example.}
\subsection{Notations and Reduction to Three Dimensions}\label{sec:3d}
For necessary simplicity, we reduce the $d$ dimensional setting to a three dimensional one:
\begin{definition}[Interface matrices]\label{interfacematrix}
  Let $\mu \in \{1,\ldots,d\}$. For a representation $G$, we \BCn{o4}{o}%
  define the left interface matrix $G^{<\mu} \in \R^{n_1\ldots n_{\mu-1} \times r_{\mu-1}}$ (cf. Definitions \ref{TTformat}) via
    \begin{align*}
    G^{<\mu}_{(i_1,\ldots,i_{\mu-1}),:} & := G_1(i_1) \ldots G_{\mu-1}(i_{\mu-1}),
    \end{align*}
    as well as the right interface matrix $G^{>\mu} \in \R^{r_\mu \times n_{\mu+1}\ldots n_d}$ via
    \begin{align*}
    G^{>\mu}_{:,(i_1,\ldots,i_{\mu-1})} & := G_1(i_{\mu+1}) \ldots G_{d}(i_{d}).
  \end{align*} 
We further define the core $A_{(\mu)} \in \left(\R^{n_1\ldots n_{\mu-1} \times n_{\mu+1} \ldots n_{d}}\right)^{\Ind_\mu}$
  as core unfolding with respect to mode $\mu$ of a tensor $A$ by
  \begin{align}
   A_{(\mu)}(i_\mu)_{(i_1,\ldots,i_{\mu-1}),(i_{\mu+1},\ldots,i_d)} = A(i). \label{coreunfold}
  \end{align}
  For any representation it hence holds $(\tau_r(G))_{(\mu)}(j) = G^{<\mu} \ G_\mu(j)\  G^{>\mu}$, $j = 1,\ldots,n_\mu$.
  Multiplication of a core $H$ with a matrix $B$ yields again cores, $HB$ and $BH$, given by
  \begin{align*}
   (HB)(j) := H(j) B \quad \mbox{and} \quad (BH)(j) := B H(j),
  \end{align*}
 respectively, for all possible $j$.
  The previous notations then allow us to compactly write
  \begin{align} \label{thr} (\tau_r(G))_{(\mu)} = G^{<\mu} \ G_\mu\  G^{>\mu}. \end{align}
  This relation is displayed in Figure \ref{LNR}.
\end{definition}
\begin{figure}
\centering
\ifuseprecompiled
\includegraphics[width=0.98\textwidth]{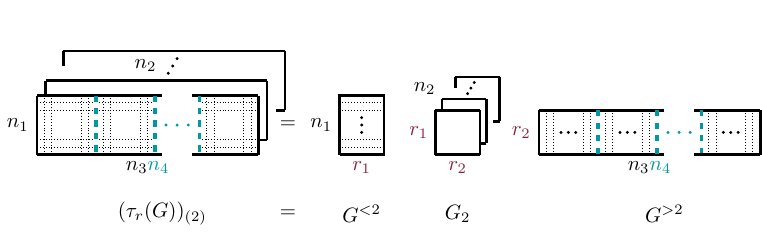}
\fi
       \caption{\label{LNR} The decomposition of a core unfolding with respect to $2$ of a four dimensional tensor into the left and right
       interface matrices as well as the intermediate core.}
\end{figure}%
\newwo{
The interface matrices equal left and right unfoldings, respectively:
  \begin{align}
    G^{<\mu} & = \lhb{G_{1,\ldots,\mu-1}}, \quad G_{1,\ldots,\mu-1}((i_1,\ldots,i_{\mu-1})) :=  G_1(i_1) \ldots G_{\mu-1}(i_{\mu-1}), \label{superdupercores} \\
    G^{>\mu} & = \rhb{G_{\mu+1,\ldots,d}}, \quad G_{\mu+1,\ldots,d}((i_{\mu+1},\ldots,i_{d})) :=  G_1(i_{\mu+1}) \ldots G_{d}(i_{d}). \nonumber
  \end{align} }
%
%
In terms of Definition \ref{restrictions}, for $P = \{ p^{(i)} \mid i = 1,\ldots,|P| \}$ and the operation $(\cdot)_{(\mu)}$ \eqref{coreunfold},
which is used to combine components $s = 1,\ldots,\mu-1$ as well as $s = \mu+1,\ldots,d$, let
\[ P_{(\mu)} = \{((p^{(i)}_1,\ldots,p^{(i)}_{\mu-1}),p^{(i)}_\mu,(p^{(i)}_{\mu+1},\ldots,p^{(i)}_d)) \mid i = 1,\ldots,|P| \}. \]
Thereby, $A|_P$ contains the same entries as $(A_{(\mu)})|_{P_{(\mu)}}$.
For the selection of one slice, $(\cdot)_{(\mu)}(j)$, we denote
\begin{align} P_{(\mu)}(j) = \{((p^{(i)}_1,\ldots,p^{(i)}_{\mu-1}),(p^{(i)}_{\mu+1},\ldots,p^{(i)}_d)) \mid p^{(i)}_\mu = j,\ i = 1,\ldots,|P| \}. \label{rrpex} \end{align}
Likewise, the vectorization of an index set $S \subset \R^{n \times m}$ is defined by $\mbox{vec}(S) = \{s_1 + n(s_2-1) \in \R \mid s \in \R^2 \}$.

Without loss of generality we can restrict our consideration to three dimensional tensors that correspond to the
left and right interface matrices as well as the respective intermediate cores (cf. \oldwo{Remark}\newwo{Definition} \ref{unfoldings}):
\begin{remark}[Reduction to three dimensions]\label{reduction}
  When $\oldwojg{\mu \in D}\newwo{\mu \in \{1,\ldots,d\}}$ is fixed, we will only use the short notations
  \begin{itemize}
    \item $(\Lf,\ \N,\ \Rf) = (G^{<\mu},\ G_\mu,\ G^{>\mu})$
    \item $(\nL,\ \nN,\ \nR) = (n_1 \cdot \ldots \cdot n_{\mu-1},\ n_\mu,\ n_{\mu+1}\cdot\ldots\cdot n_d)$
    \item $(\sigL,\ \sigR) = (\sigma^{(\mu-1)},\ \sigma^{(\mu)})$ and $(\SigL,\ \SigR) = (\Sigma^{(\mu-1)},\ \Sigma^{(\mu)}) = (\diag(\sigma^{(\mu-1)}),\ \diag(\sigma^{(\mu)}))$
    \item $(\rL,\ \rR) = (r_{\mu-1},\ r_\mu)$
    \item $\B = M_{(\mu)}$ and $\Ps = P_{(\mu)}$
  \end{itemize}
  \end{remark}
Hence, the important variables are
$L$ (\textbf{l}eft part), $R$ (\textbf{r}ight part), $N$ (\textbf{n}ew part), $B$ (right hand side) and $S$ (\textbf{s}ampling). 
  The microsteps $\mathcal{M}^{(1)},\ldots,\mathcal{M}^{(d)}$ of ALS for
  the tensor train format only change the respective $G_\mu$ and are given by
  \begin{align}
    \nonumber \mathcal{M}^{(\mu)}_r(G) & := (G_1,\ldots,G_{\mu-1},G^+_\mu,G_{\mu+1},\ldots,G_d) \\
    \label{unregM} G^+_{\mu} & := \argmin_{G_\mu} \|\tau_r(G) - M\|_P = \argmin_{\widetilde{N}} \|\Lf \cdot {\widetilde{N}} \cdot \Rf - \B\|_\Ps
  \end{align}
  or equivalently $G^+_{\mu}(j) = \argmin_{{\widetilde{N}}(j)} \|\Lf \cdot {\widetilde{N}}(j) \cdot \Rf - \B(j)\|_{\Ps(j)}$ --- an equation in which only matrices are involved.
  We only need to consider three-dimensional tensors $A \in \R^{\nL \times \nN \times \nR}$, with TT-rank $(\rL,\rR)$ and TT-singular values $(\sigL, \sigR)$. For simplicity, we
  redefine $\tau_r$ for this case via $A = \tau_r(\Lf,\N,\Rf)$.

\section{Stable Alternating Least Squares Microsteps for Tensor Completion}\label{sec:moofthsicorefu}

With the derivation in Section \ref{sec:3d}, we have seen that it is sufficient to only consider ALS for order $3$ tensors.
\begin{definition}[Variational residual function (cf. Definition \ref{sicorefuinmosmo2matrix})]\label{sicorefuinmosmo2}
  Let $\omega \geq 0, \ \gammaL,\gammaR > 0$ and $\B,\Ps,\Lf,\N,\Rf$ as in Definition \ref{reduction}. We define
  the variational residual function $C := C_{\B,\Ps,\Lf,\N,\Rf}$ for $\mathbb{V}_{\omega} := \mathbb{V}_{\omega}(\Lf,\N,\Rf)$ by
  \begin{align} \label{corf}
    C(\widetilde{\N}) & := \int_{\mathbb{V}_{\omega}(\Lf,\N,\Rf)} \| (\Lf + \gammaL \Delta \Lf) (\widetilde{\N} + \Delta \N) (\Rf + \gammaR \Delta \Rf) - B\|^2_{\Ps} \intd   \Delta \Lf \intd  \Delta \N \intd  \Delta \Rf,  \\
 \mathbb{V}_{\omega} & := \{(\Delta \Lf,\Delta \N,\Delta \Rf) \mid \|\Delta \Lf \N \Rf\|^2_{\newwo{F}} + \|\Lf \Delta \N \Rf\|^2_{\newwo{F}} + \|\Lf \N \Delta \Rf\|^2_{\newwo{F}}\leq \omega^2 \}, \nonumber
  \end{align}
  where $\gammaL, \gammaR$ are scalings that only depend on the proportions of the representation, to
  be specified by Lemma \ref{retafu}.
\end{definition}
As in the matrix case, $\Delta \N$ does not influence the minimizer, so we omit it from now on. It should further be noted that
$\mathbb{V}_{\omega}$ does \textbf{not} depend on the unknown $\widetilde{\N}$. The scalings $s_1$, $s_2$ will become relevant
in Section \ref{sec:backtod} in order to achieve a similar effect as in Remark \ref{constants_matrices}.
\subsection{Standard Representation of a TT-Tensor}
A representation $G = (\Lf,\N,\Rf)$ can be changed without changing the generated tensor $A = \tau_r(G)$ \cite{RoUs13_OnL,HoThRe12_Onm},
more specifically
\begin{align}\label{equi} \tau_r(G) = \tau_r(\widetilde{G}) \quad \Leftrightarrow \quad \widetilde{G} = (\widetilde{\Lf},\widetilde{\N},\widetilde{\Rf}) = (\Lf T_1^{-1}, T_1 \N T_2^{-1}, T_2 \Rf) \end{align}
for two regular matrices $T_1 \in \R^{\rL \times \rL}, T_2 \in \R^{\rR \times \rR}$.
One can define an extended standard representation that explicitly contains the unique TT-singular values $(\gamma,\theta)$
which is essentially unique (in terms of uniqueness of the truncated matrix SVD\footnote{Both $U \Sigma V^T$ and $\widetilde{U} \Sigma \widetilde{V}^T$ are truncated SVDs of $A$ if and only if there exists
an orthogonal matrix $W$ that commutes with $\Sigma$ and for which $\widetilde{U} = U W$ and $\widetilde{V} = V W$.
For any subset of pairwise distinct nonzero singular values, the corresponding submatrix of $W$ needs to be diagonal with entries in $\{-1,1\}$.}). 
For the construction, a slightly modified TT-SVD \cite{Os11_Ten} is used\footnote{Although called TT-SVD, unlike the matrix SVD, decompositions constructed by the algorithm do not explicitly contain the singular values.}. An analogous decomposition also appeared earlier in \cite{Vi2003_Eff}  
and is, within the quantum computing community, sometimes referred to as \textit{canonical MPS form}
(not to be confused with canonical polyadic decomposition).
This normalization is needed to obtain the same simplification as in Theorem \ref{miofthalsrerefufoma}.
%
\begin{lemma}[Standard representation]\label{stre}
  Let $A \in \R^{\nL \times \nN \times \nR}$ be a tensor. \\
  There exists an essentially unique (extended) representation
  \begin{align} \mathcal{G} = (\fL,\SigL,\fN,\SigR,\fR) \label{sr} \end{align}
  for which $A = \tau_r(\fL, \ \SigL \fN \SigR, \ \fR)$ as well as 
  $\fL \ \SigL \ \rhb{\fN \SigR \fR}$ and $\lhb{\fL \SigL \fN} \ \SigR \ \fR$
  are (truncated) SVDs of $A^{(\{1\})}$ and $A^{(\{1,2\})}$, respectively.
  This in turn implies that
  $\fL$ and $\lhb{\SigL \fN}$ are column orthogonal,  as well as $\fR$ and $\rhb{\fN \SigR}$ are row orthogonal. 
\end{lemma}
%
\begin{proof} 
  \textit{$1$. Uniqueness:} Let there be two such representations $\widetilde{\mathcal{G}}$ and $\mathcal{G}$. 
  Since the left-singular vectors of $A^{(\{1\})}$ are essentially unique, we conclude $\widetilde{\fL} = \fL W_1$ for an
  orthogonal matrix $W_1$ that commutes with $\SigL$. Via an SVD of $A^{(\{1,2\})}$
  it follows that $\widetilde{\fR} = W_2^T \fR$ for an
  orthogonal matrix $W_2$ that commutes with $\SigR$. Furthermore $\lhb{\fL \SigL W_1 \widetilde{\fN}} = \lhb{\widetilde{\fL} \SigL \widetilde{\fN}} = \lhb{\fL \SigL \fN} W_2$. The map $x \mapsto \lhb{\fL \SigL x}$
  is linear and, in this case, of full rank. This implies $\widetilde{\fN} = W_1^T \fN W_2$. \\
  \textit{$2.$ Existence (constructive):}
  Let $A = \tau_r(\widetilde{\Lf},\widetilde{\N},\widetilde{\Rf})$ where $\rhb{\widetilde{\N}}$ and $\widetilde{\Rf}$ are column orthogonal (this
  can always be achieved using \eqref{equi}). An SVD of $\widetilde{\Lf}$ yields $\widetilde{\Lf} = \fL \ \SigL\ V_1^T$, since
  $\fL \ \SigL\ \rhb{V_1^T \widetilde{\N} \widetilde{\Rf}}$ is a truncated SVD of $A^{(\{1\})}$. 
  A subsequent
  SVD of $\lhb{\SigL V_1^T \widetilde{\N}}$ yields $\SigL V_1^T \widetilde{\N} = \widehat{\N}\ \SigR\ V_2^T$, since 
  $\lhb{\fL \widehat{\N}}\ \SigR \ (V_2^T \widetilde{\Rf})$ is a truncated SVD of $A^{(\{1,2\})}$. We can finish
  the proof defining $\fN := \SigL^{-1} \widehat{\N}$ and $\fR = V_2^T \widetilde{\Rf}$. Note that, by construction,
  $\lhb{\SigL \fN}$ is column-orthogonal. \\
  \textit{$3.$ Implied orthogonality:} 
  Using the essential uniqueness, it follows that $\lhb{\SigL \fN}$ must indeed be column-orthogonal. By analogously constructing the extended representation
  from right to left we would obtain that $\rhb{\fN \SigR}$ is row-orthogonal. By uniqueness it follows again that this is always the case.
\end{proof}
\begin{remark}[Conventional form of standard representation]\label{cofoofstre}
  Throughout the rest of the article, the standard representation will mostly appear in form
  of a specific, conventional representation
  \begin{align} \label{csp} (\Lf,\ \N,\ \Rf) = (\fL, \ \SigL \fN \SigR, \ \fR), \end{align}
  hence with interface matrices $\fL$ and $\fR$ given by corresponding singular vectors.
\end{remark}
\subsection{Minimizer of the Variational Residual Function}
We define (from now on) our method as
\begin{align} \label{method} \mathcal{M}^{\ast}(\Lf,\ \N,\ \Rf) 
  = (\Lf,\ \argmin_{\widetilde{\N}} C(\widetilde{\N}),\ \Rf)
\end{align}
with $C = C_{\B,\Ps,\Lf,\N,\Rf}$ as in \eqref{corf}. Although Theorem \ref{miofthalsrerefu},
or more specifically the regularity of $Y(j)$ given by \eqref{Yj}, later provide the uniqueness
of the minimizer, we up to that point formally use the minimization of $\|\tau_r(\Lf,\widetilde{\N},\Rf)\|_F$ as
secondary and representation independent criterion. The special cases $\mu \in \{1,d\}$ can be
derived from the general case (Remark \ref{casemu1d}) and comply with the matrix case.
\begin{lemma}[Representation independent]\label{in}
  The method $\mathcal{M}^{\ast}$ is representation independent.
\end{lemma}
\begin{proof}
  Let $\N^+ := \argmin_{\widetilde{N}} C$, $C = C_{\B,\Ps,\Lf,\N,\Rf}(\widetilde{N})$ and 
  $\widehat{\N}^+ := \argmin_{\widetilde{N}} \widehat{C}$,\\$\widehat{C} = C_{\B,\Ps,\widehat{\Lf},\widehat{\N},\widehat{\Rf}}(\widetilde{N})$
  for representations $\tau_r(\Lf,\N,\Rf) = \tau_r(\widehat{\Lf},\widehat{\N},\widehat{\Rf})$ as well as 
  $\widehat{\mathbb{V}}_{\omega} = \mathbb{V}_{\omega}(\widehat{\Lf}, \widehat{\N}, \widehat{\Rf})$
  and $\mathbb{V}_{\omega} = \mathbb{V}_{\omega}(\Lf, \N, \Rf)$.
  According to \eqref{equi}, there exist two matrices $T_1,T_2$ such that
  \[(\Lf T_1, T_1^{-1} \N T_2, T_2^{-1} \Rf) = (\widehat{\Lf}, \widehat{\N}, \widehat{\Rf}).\]
  Hence
  \begin{align*}  \widehat{C}(\widetilde{N}) & = \int_{\widehat{\mathbb{V}}_{\omega}} \left\| (\Lf + \gammaL \Delta \widehat{\Lf} T_1^{-1}) T_1 \widetilde{\N} 
    T_2^{-1} (\Rf + \gammaR T_2 \Delta \widehat{\Rf}) - \B \right\|^2_{\Ps} \intd   \Delta \widehat{\Lf} \intd  \Delta \widehat{\N} \intd  \Delta \widehat{\Rf}, \\
    \mbox{with } \widehat{\mathbb{V}}_{\omega} & = \left\{(\Delta \widehat{\Lf},\Delta \widehat{\N},\Delta \widehat{\Rf}) \mid 
    \|\Delta \widehat{\Lf} T_1^{-1} \N \Rf\|^2_{\newwo{F}}
    + \|\Lf T_1^{-1} \Delta \widehat{\N} T_2 \Rf\|^2_{\newwo{F}} + \|\Lf \N T_2 \Delta \widehat{\Rf}\|^2_{\newwo{F}}\leq \omega^2 \right\}
  \end{align*}
  The substitution $(\Delta \widehat{\Lf},\Delta \widehat{\N},\Delta \widehat{\Rf}) \overset{\iota}{\rightarrow} (\Delta \Lf T_1, T_1^{-1} \Delta \N T_2, T_2^{-1} \Delta \Rf)$
  introduces a constant Jacobi Determinant $|\det(J_\iota)| = 1$. We obtain
  \begin{align*} \widehat{C}(\widetilde{N}) &
    := \int_{\mathbb{V}_{\omega}} \left\| (\Lf + \gammaL \Delta \Lf) \ (T_1 \widetilde{\N} T_2^{-1}) 
     (\Rf + \gammaR \Delta \Rf) - \B \right\|^2_{\Ps} \intd   \Delta \Lf \intd  \Delta \N \intd  \Delta \Rf \\
    & = C(T_1 \widetilde{\N} T_2^{-1})
  \end{align*}
  Hence $\widehat{\N}^+ = T_1^{-1} \N^+ T_2$. This is the same relation
  given for $\N$ and $\widehat{\N}$ and therefore $\tau_r(\Lf, \N^+, \Rf) = \tau_r(\widehat{\Lf},\widehat{\N}^+,\widehat{\Rf})$
  (which is a set equality if the minimizer is not assumed to be unique).
\end{proof}
%
\begin{corollary}[Integral over Kronecker product]\label{inovkrpr}
  Let $\omega_1 > 0$.\\Further, let $H \in \R^{(n_X n_Y) \times (n_X n_Y)}$ as well as $Y \in \R^{n_Y \times n_Y}$ be matrices and
  \[ V_{\omega_1}^{(n_X,m_X)} = \{ X \in \R^{n_X \times m_X} \mid \|X\|_F = \omega_1\}. \]
  Then
  \begin{align*} \int_{V_{\omega_1}^{(n_X,m_X)}} (X \kp Y)^T H (X \kp Y) \intd X = \frac{\omega_1^2 |V_{\omega_1}^{(n_X,m_X)}|}{n_X m_X} I_{m_X} \kp Y^T H^{\ast} Y \end{align*}
  for $H^{\ast} = \mathrm{tr}_1(H) \in \R^{n_Y \times n_Y}$.\footnote{$\mathrm{tr}_1$ is the partial trace, i.e. $(H^{\ast})_{i,j} = \mbox{tr}(h_{i,j}), \ H = \sum_{i,j} h_{i,j} \kp e_i e_j^T,\ h_{i,j} \in \R^{n_X \times n_X}$}
  For an analog $V_{\omega_2}^{(n_Y,m_Y)}$, $\omega_2 > 0$, we further have
  \begin{align*} \underset{V_{\omega_1}^{(n_X,m_X)}, V_{\omega_2}^{(n_Y,m_Y)}}{\iint} \kern-1em (X \kp Y)^T H (X \kp Y) \intd X \intd Y = \frac{\omega_1^2 \omega_2^2 |V_{\omega_1}^{(n_X,m_X)}||V_{\omega_2}^{(n_Y,m_Y)}|}{n_X m_X n_Y m_Y} tr(H) I_{m_X m_Y}. \end{align*}
\end{corollary}
\begin{proof}
  Using the splitting $H = \sum_{i,j} h_{i,j} \kp e_i e_j^T$, $h_{i,j} \in \R^{n_X \times n_X}$,
  Lemma \ref{inovalva} can be applied to each summand, separately for $X$ and $Y$.
\end{proof}
We now derive the minimizer of the variational residual function \eqref{corf}.
Due to Lemma \ref{in}, we can use the standard representation in form of Remark \ref{cofoofstre}
for simplification. In this case, $\mathbb{V}_{\omega}$ takes the convenient form
\begin{align} \label{convenient} \mathbb{V}_{\omega}(\fL, \ \SigL \ \fN \ \SigR, \ \fR) & = \{(\Delta \Lf,\Delta \N,\Delta \Rf) \mid 
  \|\Delta \Lf \SigL \|^2_{\newwo{F}} + \|\Delta \N \|^2_{\newwo{F}} + \|\SigR \Delta \Rf\|^2_{\newwo{F}}\leq \omega^2 \}.
\end{align}
\begin{theorem}[Minimizer of the ALS variational residual function]\label{miofthalsrerefu}
  Let $(\fL, \ \SigL ,\ \fN ,\ \SigR, \ \fR)$ be the standard representation \eqref{sr} for a tensor $A$.
  The minimizer $\N^+$ of the residual function $C_{\B,\Ps,\fL,\SigL \fN \SigR,\Rf}$ as in \eqref{corf} is given by
\begin{align}
\nonumber
 \N^+(j) & \ = \ \argmin_{\widetilde{N}(j)} \quad \underbrace{\|\ \fL \ \widetilde{N}(j) \ \fR - \B(j)\ \|^2_{\Ps(j)}}_{\mbox{standard ALS}} \quad + \underbrace{\quad  \Omega(j), \quad}_{\mbox{(regularization)}} j = 1,\ldots,\nN \\
 \label{Yj}
 \Omega(j) & \ = \ \left\{
  \begin{array}{rllr}
  \nL^{-1} \zeta_1 & \omega^2  \|\ \SigL^{-1}  & \widetilde{N}(j) & \fR_{:,\Ps(j)_2}\ \|^2_F \\
  + \ \nR^{-1} \zeta_2  & \omega^2 \|\ \fL_{\Ps(j)_1,:} & \widetilde{N}(j) & \SigR^{-1}\ \|^2_F \\
  + \ \newwo{|\Ps(j)| (\nR\nL)^{-1}} \zeta_{(1,2)}  & \omega^4 \|\ \SigL^{-1}  & \widetilde{N}(j) & \SigR^{-1}\ \|^2_F
\end{array} \right.
\end{align}
  with $\Ps(j)_u = (x^{(1)}_u,x^{(2)}_u,\ldots)$, \movedwo{$u = 1,2$,} for $\Ps(j) = \newwo{\{}x^{(1)},x^{(2)},\ldots\newwo{,x^{(|S(j)|)}\}} \subset \mathbb{N}^2$. 
  The constants $\zeta\oldwojg{, \rho}$ only depend
  on the proportions of the representation and sampling set (cf. Remark \ref{constants}) as well as
  the constant scalings $\gammaL$, $\gammaR$.
\end{theorem}
\oldjg{30,43}{
\begin{proof}
  The proof is rather technical and can be found in Appendix A.
\end{proof}}
\oldwo{This result may appear to be intricate. However, to calculate the minimizer is
of the same order (with near same constant)}
\newwo{The regularization term is more straightforward than it might appear. The computational complexity
for the calculation of the minimizer is of the same order (with near same constant)}
as for standard ALS, for which
the matrices $\fL_{\Ps(j)_1,:} \in \R^{\oldwojg{a_j}\newwo{|S(j)|} \times \rL}$ and $\fR_{:,\Ps(j)_2} \in \R^{\rR \times \oldwojg{a_j}\newwo{|S(j)|}}$ 
\oldwo{($a_j = |\{ p \mid p \in P,\ p_\mu = j\}|$)} are
required anyhow (for further explanation, see \eqref{fLPsj},\eqref{fRPsj}). The scalings $\nL^{-1},\nR^{-1}$ go in hand
with these to adjust the regularization terms to the magnitudes of the
corresponding shares of the sampling set and disappear\oldwo{, together with $\rho_{(1,2)}$,} for the approximation of a fully available tensor (Corollary \ref{fipr}).
The constants $\zeta$ are weighted according to the proportions of the left and right interface matrices.
For example, for $\mu = 1$ (the position of the core being updated), we have $\zeta_1 = 0$ and accordingly
for $\mu = 2$, $\zeta_1 \ll \zeta_2$. In practice, instead of evaluating the minimizer exactly, a few steps
(that is for very coarse tolerance) of a suitably computed, \textit{preconditioned cg} are performed as
described in Section \ref{CCG}. This can be achieved without changing results.

\begin{proof}\BCn{f3_}{f_}
We omit the scalings $\gammaL, \gammaR$ for simplicity since they only have to be carried along the lines.
We search for $\N^+ := \argmin_{\widetilde{\N}} C_{\B,\Ps,\fL,\SigL \fN \SigR,\fN}(\widetilde{\N})$. 
Substituting
\[ 
 (\Delta \Lf,\Delta \N,\Delta \Rf) \rightarrow (\Delta \fL \SigL^{-1}, \Delta \fN ,\SigR^{-1} \Delta \fR)
\]
we can (up to a constant factor) restate $C$ as
\begin{align}
C_{\B,\Ps,\fR,\SigL \fN \SigR,\fL}(\widetilde{\N}) & \propto \int_{\mathbb{V}_{\omega}} \|
(\fL + \Delta \fL \SigL^{-1}) \widetilde{\N} 
\nonumber (\fR + \SigR^{-1} \Delta \fR) - \B\|^2_{\Ps} \intd   \Delta \fL \intd \Delta \fN \intd \Delta \fR, \\
\mathbb{V}_{\omega} & = \{(\Delta \fL,\Delta \fN,\Delta \fR) \mid 
\|\Delta \fL \|^2 + \|\Delta \fN \|^2 + \| \Delta \fR\|^2\leq \omega^2 \}.
\end{align}
Each of the independent matrices of the minimizing core is restated as
\begin{align}
 \N^+(j) & = \argmin_{\widetilde{\N}(j)} \int_{\mathbb{V}_\omega} \| ((\fR + \SigR^{-1} \Delta \fR)^T \otimes_K (\fL + \Delta \fL \SigL^{-1})) \ \mbox{vec}(\widetilde{\N}(j)) \\
 & - \mbox{vec}(\B(j)) \|^2_{\mbox{vec}(\Ps(j))} \intd  \Delta \fL \intd \Delta \fN \intd \Delta \fR, 
\end{align}
where $\otimes_K$ is the matrix Kronecker product.
Let $j$ be arbitrary but fixed from now on.
For any $x$, it is $\|x\|_{\mbox{vec}(\Ps(j))} = \|H(j) x\|_{F} = x^T H(j) x$ for a diagonal, 
square matrix $H(j) \in \R^{|\Ind|/\nN \times |\Ind|/\nN}$ 
with $H(j)_{(s),(s)} = \delta_{s \in \Ps(j)}$ (hence $H(j)^2 = H(j)$).
Using the normal equation, we obtain $\N^+(j) = Y^{-1} b$, where
\begin{align*} Y = & \int_{\mathbb{V}_\omega} (\fR + \SigR^{-1} \Delta \fR) \otimes_K (\fL + \Delta \fL \SigL^{-1})^T
\\ & H(j) \ (\fR + \SigR^{-1} \Delta \fR)^T \otimes_K (\fL + \Delta \fL \SigL^{-1})
 \intd   \Delta \fL \intd \Delta \fN \intd \Delta \fR 
\end{align*}
and
\begin{align*} b = & \left( \int_{\mathbb{V}_\omega} (\fR + \SigR^{-1} \Delta \fR) \otimes_K  (\fL + \Delta \fL \SigL^{-1})^T \right) \\
 & \ H(j) \ \mbox{vec}(\B(j))
\intd   \Delta \fL \intd \Delta \fN \intd \Delta \fR .
\end{align*}
In both $Y$ and $b$, any perturbation that appears only one-sided vanishes due to symmetry of $\mathbb{V}_\omega$. 
Hence $b = |\mathbb{V}_\omega| \ {({\fR}^T \otimes_K \fL)_{\mbox{vec}(\Ps(j)),:}}^T \ \mbox{vec}(\B(j))_{\mbox{vec}(\Ps(j))}$
and for $\intd \delta := \intd \Delta \fL \intd \Delta \fN \intd \Delta \fR$
\begin{align*} Y = & \int_{\mathbb{V}_\omega} ({\fR}^T \otimes_K {\fL})^T \ H(j) \ ({\fR}^T \otimes_K {\fL}) \intd  \delta \\
 + & \int_{\mathbb{V}_\omega} ({\fR}^T \otimes_K \Delta {\fL} \SigL^{-1})^T \ H(j) \ ({\fR}^T \otimes_K \Delta {\fL} \SigL^{-1}) \intd  \delta \\
 + & \int_{\mathbb{V}_\omega} (\Delta {\fR}^T \SigR^{-1} \otimes_K {\fL})^T \ H(j) \ 
 (\Delta {\fR}^T \SigR^{-1} \otimes_K {\fL}) \intd  \delta \\
 + & \int_{\mathbb{V}_\omega} (\Delta {\fR}^T \SigR^{-1} \otimes_K \Delta {\fL} \SigL^{-1})^T \ H(j) \ 
 (\Delta {\fR}^T \SigR^{-1} \otimes_K \Delta {\fL} \SigL^{-1}) \intd  \delta
\end{align*}
Now, let $\ell = \#_{\fR},\ n = \#_\fN,\ k = \#_{\fL}$. Since $\mathbb{V}$ is a version of the $(\ell + n + k)$-sphere,
we can use the following integration formula: 
Let $f: \R^{n+m} \rightarrow \R^k$ be a sufficiently smooth function and $S^{v-1}_{\omega}$ be the $v$-sphere of radius $\omega$.
 Then
 \[ \int_{S^{n+m-1}_{\omega}} f(x_n,x_m) \intd x = \int_0^{\pi/2} \omega \int_{S^{n-1}_{\omega \sin(u)}} \int_{S^{m-1}_{\omega \cos(u)}} f(x_n,x_m) \intd  x_m \intd  x_n \intd  u. \]
We use it twice and thereby split the integral. For a function $f$ we then obtain
\begin{align*}
 & \int_{\mathbb{V}} f \intd  \delta = 
 \int_{\lambda = 0}^{\omega} \int_{S_{\lambda}^{n + \ell + k-1}} f \intd  \delta \intd \lambda =
 \int_{\lambda = 0}^{\omega} \lambda \int_{g = 0}^{\pi/2} \int_{S_{\lambda \sin(g)}^{n -1}} \int_{S_{\lambda \cos(g)}^{\ell + k-1}} f \intd  \delta \intd g \intd \lambda = \\
 & \int_{\lambda = 0}^{\omega} \lambda \int_{g = 0}^{\pi/2} \int_{S_{\lambda \sin(g)}^{n -1}} \lambda \cos(g) 
 \int_{u = 0}^{\pi/2} \int_{S_{\lambda \cos(g) \sin(u)}^{\ell-1}} \int_{S_{\lambda \cos(g) \cos(u)}^{k-1}} f \intd \Delta \fL \intd \Delta \fR  \intd u \intd \Delta \fN \intd g \intd \lambda 
\end{align*}
If $f$ is independent of $\Delta \fN$, this then simplifies to
 \begin{align*} 
 = \int_{\lambda = 0}^{\omega} \lambda^2 \int_{g = 0}^{\pi/2} |S_{\lambda \sin(g)}^{n-1}| \cos(g) 
 \int_{u = 0}^{\pi/2} \int_{S_{\lambda \cos(g) \sin(u)}^{\ell-1}} \int_{S_{\lambda \cos(g) \cos(u)}^{k-1}} f \intd \Delta \fL \intd \Delta \fR  \intd u \intd g \intd \lambda 
\end{align*}
We further use the identity (where the function $\Gamma(\cdot)$ is not to be confused with the given diagonal matrix $\Gamma$)
\[ \int_0^{\pi/2} \cos(x)^p \sin(x)^q \intd x = \frac{\Gamma((p+1)/2) \ \Gamma((q+1)/2)}{2 \Gamma((p+q+2)/2)} =: \nu(p,q) \]
We apply these and Corollary \ref{inovkrpr} for different $f = (X \otimes_K Y)^T  H(j) (X \otimes_K Y)$.
For $\delta_1, \delta_2 \in \{0,1\}$ we set $X$ as ${\fR}^T$ ($\delta_1 = 0$) or $\Delta {\fR}^T \SigR^{-1}$ ($\delta_1 = 1$)
and analogously $Y$ as $\fL$ ($\delta_2 = 0$) or $\Delta \fL \SigL^{-1}$ ($\delta_2 = 1$).
For the summands $Y(0,0) + Y(1,0) + Y(0,1) + Y(1,1) = Y$ this then yields
\begin{align*} Y(\delta_1,\delta_2) = &
\int_{\lambda = 0}^{\omega} \lambda^2 \int_{g = 0}^{\pi/2}
\cos(g) \frac{2 \pi^{n/2} (\lambda \sin(g))^{n-1}}{\Gamma(n/2)} \\
& \int_{u = 0}^{\pi/2}  \frac{2 \pi^{\ell/2} (\lambda \cos(g) \sin(u))^{\ell-1}}{\Gamma(\ell/2)} \left(\lambda^2 \cos(g)^2 \sin(u)^2 \right)^{\delta_1}  \\
& \frac{2 \pi^{k/2} (\lambda\cos(g) \cos(u))^{k-1}}{\Gamma(k/2)} \left(\lambda^2 \cos(g)^2 \cos(u)^2 \right)^{\delta_2} \intd  u \intd g \intd \lambda \cdot C_H({\delta_1},{\delta_2}) \\
= & \ c \cdot \int_{\lambda = 0}^{\omega} \lambda^{n+\ell+k-1 + 2 {\delta_1} + 2 {\delta_2}} \intd \lambda  \\
& \cdot \int_{g = 0}^{\pi/2} \cos(g)^{\ell+k-1 + 2 {\delta_1} + 2 {\delta_2}} \sin(g)^{n-1} \intd g \\
& \cdot \int_{u = 0}^{\pi/2} \cos(u)^{k-1 + 2 {\delta_2}} \sin(u)^{\ell-1 + 2 {\delta_1}} \intd u \cdot C_H({\delta_1},{\delta_2}) \\
= & \ c \frac{\omega^{n+\ell+k + 2 {\delta_1} + 2 {\delta_2}}}{n+\ell+k + 2 {\delta_1} + 2 {\delta_2}} \\
& \ \nu(\ell+k-1 + 2 {\delta_1} + 2 {\delta_2},n-1) \ \nu(k-1 + 2 {\delta_2},\ell-1 + 2 {\delta_1}) \ C_H({\delta_1},{\delta_2})
\end{align*}
for $c = \frac{8 \pi^{(n+k+\ell)/2}}{\Gamma(n/2)\Gamma(\ell/2)\Gamma(k/2)}$. The constant matrices $C_H$ are given by
\begin{align*}
 C_H(0,0) & = \widetilde{C}_H(0,0) = K(0,0)^T \ K(0,0), \ K(0,0) = ({\fR}^T \otimes_K \fL)_{\mbox{vec}(\Ps(j)),:} \\
 \nL \rL C_H(1,0) & = \widetilde{C}_H(1,0) = K(1,0)^T \ K(1,0), \ K(1,0) = {\fR_{:,\Ps(j)_2}}^T \otimes_K \SigL^{-1} \\
 \nR \rR C_H(0,1) & = \widetilde{C}_H(0,1) = K(0,1)^T \ K(0,1), \ K(0,1) = \SigR^{-1} \otimes_K \fL_{\Ps(j)_1,:} \\
 |S(j)|^{-1} \nL \nR \rL \rR C_H(1,1) & = \widetilde{C}_H(1,1) = K(1,1)^T \ K(1,1), \ K(1,1) = \SigR^{-1} \otimes_K \SigL^{-1}
\end{align*} 
Furthermore, it is $|\mathbb{V}_\omega| = c \frac{\omega^{n+\ell+k}}{n+\ell+k} \nu(\ell+k-1,n-1) \ \nu(k-1,\ell-1)$. 
Factoring out this base volume in $Y = |\mathbb{V}_\omega| \widetilde{Y}$
by using properties of the $\Gamma$ function, one derives:
\begin{alignat*}{3}
 \widetilde{Y}(0,0) & = \widetilde{C}_H(0,0),\quad  && \widetilde{Y}(1,0) = \nL^{-1} \zeta_{1} \omega^2 \widetilde{C}_H(1,0),\\
 \widetilde{Y}(0,1) & = \nR^{-1} \zeta_{2} \omega^2 \widetilde{C}_H(0,1),\quad && \widetilde{Y}(1,1) = |S(j)| \nL^{-1} \nR^{-1} \zeta_{(1,2)} \omega^4 \widetilde{C}_H(1,1),
\end{alignat*}
where the constants $\gammaL$ and $\gammaR$ have been added again. Restating the result
again as a least squares problem finishes the proof.
\end{proof}
\movedwo{%
\begin{remark}[Specification of constants]\label{constants}
 Let $\#_\fR := size(\fR),\ \#_\fN := size(\fN),\ \#_\fL := size(\fL)$ be the sizes of the tensor components.
  The constants in Theorem \ref{miofthalsrerefu} are given by \oldwo{$\rho_{(1,2)} = |\Ps(j)| \nL^{-1} \nR^{-1} = |\Ps(j)| (|\Ind|/\nN)^{-1}$
  and}
  \begin{align*}  \zeta_1 & = \gammaL^2 \frac{\#_\fL}{\rL (\#_\fL+\#_\fN+\#_\fR+2)}, \\
  \zeta_2 & = \gammaR^2 \frac{\#_\fR}{\rR (\#_\fL+\#_\fN+\#_\fR+2)}, \\   
    \zeta_{(1,2)} & = \gammaL^2 \gammaR^2 \frac{\#_\fR \#_\fL}{\rL \rR (\#_\fL+\#_\fN+\#_\fR+2)(\#_\fL+\#_\fN+\#_\fR+4)}.
  \end{align*}
\end{remark}}
\subsection{Evaluation with Coarse Conjugate Gradient}\label{CCG} %

\newwo{
For each slice $j$, the solution to the least squares problem in Theorem \ref{miofthalsrerefu} is described through the normal
equation  
\begin{align}
 Z(j)^T Z(j)\ \mathrm{vec}(N^+(j)) = Z(j)^T \begin{pmatrix} B(j)|_{\mathrm{vec}(S(j))} & 0\end{pmatrix} \label{normaleq}
\end{align}
with
\begin{align}
  Z(j) & := \begin{pmatrix} (\fR^T \otimes \fL)|_{\mathrm{vec}(S(j)),:} \\ Y(j) \end{pmatrix}, \quad
   Y(j) := \begin{pmatrix}
      \sqrt{\nL^{-1} \zeta_1} \ {\fR_{:,\Ps(j)_2}}^T \kp \omega \SigL^{-1} \\       
      \sqrt{\nR^{-1} \zeta_2} \ \omega \SigR^{-1} \kp \fL_{\Ps(j)_1,:} \\
      \sqrt{|\Ps(j)| \nR^{-1}\nL^{-1} \zeta_{(1,2)}} \ \omega \SigR^{-1} \kp \omega \SigL^{-1} \nonumber.
    \end{pmatrix}.
\end{align}}
\new{19,35}{In practice, since within each microstep we do not benefit from an exact solution of this transitory system,
 it is much more economic to perform a preconditioned conjugate gradient method and terminate when a coarse, relative tolerance (e.g. $\mathrm{tol} = 10^{-2}$) is reached.
This tolerance is empirically chosen such that the number of required cg steps is minimized,
however under the condition that the approximation quality does not notably change in either direction\BCn{e1_}{e_} --- such that neither loss of accuracy
nor additional regularization can be observed (cf. Section \ref{sec:seimannounraad}).
With the following consideration, we can construct a preconditioner (cf. \eqref{filterex}).}

\begin{corollary}[Filter properties]\label{fipr}
  For full sampling, $P = \Ind$, the update is given by the so called {\normalfont{filter}} (a diagonal matrix)
  \begin{align} 
    \nonumber \filter & := (I \kp I + \zeta_1 \cdot I \otimes \omega^2 \SigL^{-2} + \zeta_2 \cdot \omega^2 \SigR^{-2} \otimes I + \zeta_{(1,2)} \cdot \omega^2 \SigR^{-2} \otimes \omega^2 \SigL^{-2} )^{-1}, \\
    \label{Gamma}  \mathrm{vec} (\N^+(j)) & = \filter \ \mathrm{vec} ({\fL}^T \ \B(j) \ {\fR}^T).
  \end{align}
\end{corollary}
\begin{proof}
From $P = \Ind$, it follows that $\fR_{:,\Ps(j)_2}$ is an $n_L$-order copy of $\fR$ and $\fL_{\Ps(j)_1,:}$ is an $n_R$-order copy of $\fL$ (cf. \eqref{Yj}).
Hence the regularization terms are the same for all $j = 1,\ldots,\nN$. The minimizer $N^+(j)$ is
given by 
\[ (\newwo{Z(j)}^T \newwo{Z(j)})^{-1} \newwo{Z(j)}^T \begin{pmatrix} 
                     \mbox{vec}(\B(j)) \\
                     0 \\
                     \vdots
                    \end{pmatrix} \quad
\mbox{for}
\quad
\newwo{Z(j)} = \begin{pmatrix} 
\fR^T \otimes \fL \\
\begin{array}{c}
\sqrt{\nL^{-1} \zeta_1} \ \fR^T \otimes \omega \SigL^{-1} \\
\vdots 
\end{array} \\ 
\begin{array}{c}
\sqrt{\nR^{-1} \zeta_2} \ \omega \SigR^{-1} \otimes \fL \\
\vdots 
\end{array} \\ 
\sqrt{\zeta_{(1,2)}} \ \omega \SigR^{-1} \kp \omega \SigL^{-1}
\end{pmatrix}
\begin{matrix}
\vphantom{
\fR^T \otimes \fL} \\
\left. \vphantom{
\begin{array}{c}
\sqrt{\nL^{-1} \zeta_1} \ \fR^T \otimes \SigL^{-1} \\
\vdots 
\end{array} } \right\} n_L\mbox{-times} \\ 
\left.
\vphantom{
\begin{array}{c}
\sqrt{\nR^{-1} \zeta_2} \ \SigR^{-1} \otimes \fL \\
\vdots 
\end{array} } \right\} n_R\mbox{-times} \\ 
\vphantom{
\sqrt{\zeta_{(1,2)}} \ \SigR^{-1} \kp \SigL^{-1}
}
\end{matrix}.
\]
The factors $\sqrt{\nL^{-1}}$ and $\sqrt{\nR^{-1}}$
vanish in $\newwo{Z(j)}^T \newwo{Z(j)}$ due to the multiple rows involving the orthogonal matrices $\fR$ and $\fL$. Furthermore,
$(\newwo{Z(j)}^T \newwo{Z(j)})^{-1}$ is diagonal.
\end{proof}
\begin{remark}[Application of cg algorithm and order of computational complexity]\label{applofcg}
\oldjg{19}{In practice, since within each microstep we do not benefit by an exact solution of the transitory least squares problem as in Theorem \ref{miofthalsrerefu},
 it is much more economic to perform preconditioned cg steps and terminate when a coarse, relative tolerance (e.g. $\mathrm{tol} = 10^{-2}$) is reached. 
 We observe that
 \begin{align}
   \nL^{-1} \nR^{-1} |S(j)| \filter^{-1} & \approx (\fR^T \otimes \fL)|_{\mathrm{vec}(S(j)),:}^T (\fR^T \otimes \fL)|_{\mathrm{vec}(S(j)),:} + Y(j)^T Y(j), \label{precond} \\
   Y(j) & := \begin{pmatrix}
      \sqrt{\nL^{-1} \zeta_1} \ {\fR_{:,\Ps(j)_2}}^T \kp \omega \SigL^{-1} \\       
      \sqrt{\nR^{-1} \zeta_2} \ \omega \SigR^{-1} \kp \fL_{\Ps(j)_1,:} \\
      \sqrt{\rho_{(1,2)} \zeta_{(1,2)}} \ \omega \SigR^{-1} \kp \omega \SigL^{-1} \nonumber
    \end{pmatrix},
 \end{align}
 such that $\filter$ is an excellent preconditioner. Hence very few iterations (and at most $\rL \rR$) are sufficient. In that sense, the complexity 
  is reduced to $\mathcal{O}(\rL \rR |P|)$ without any noticeable loss of the final approximation quality.}
\new{19}{The matrix $\filter$ serves as excellent preconditioner in the sense that $\nL^{-1} \nR^{-1} |S(j)| \filter^{-1} \approx Z(j)^T Z(j)$,
 which holds as equality for $P = \mathcal{I}$ (as in the previous Corollary \ref{fipr}). 
 This relation can as well be quantified through upper bounds on the condition number of $Z(j) \filter^{1/2}$ as in Lemma \ref{siTRIP}
 by which we expect (and observe in practice) very few iterations (and at most $\rL \rR$) to be sufficient to reach a given coarse tolerance
 (for standard ALS, this is however not necessarily the case, cf. Lemma \ref{lik}).
 Each single cg step then has complexity $\mathcal{O}(\rL \rR |P|)$, while the full least squares problem has complexity $\mathcal{O}(\rL^2 \rR^2 |P|)$ (per slice)}.
\end{remark}
%
%
\subsection{Stability and Restricted Isometry Properties}
%
Before we derive the central theoretical statement of this paper, Theorem \ref{coofthfiadfmist}, some preparation is necessary.
The tensor restricted isometry property (e.g. \cite{RaScSt15_Ten}) does not hold for any non trivial sampling set $P \varsubsetneq \Ind$.
We however only need to work with a modified version as follows, in which
left and right interface matrices are fixed. Apart from that, the shape is exactly the same\BCn{s2}{s}. %

\begin{definition}[Internal tensor restricted isometry property (iTRIP)]\label{iTRIP}
  We say a rank $r$ tensor $A = \tau_r(\Lf, \N, \Rf)$ has the {\normalfont internal} tensor restricted isometry property for the sampling set $\Ps \newwo{= P_{(\mu)}}$\BCn{s1}{s},
  if there exist $0 \leq c < 1$ and $\rho > 0$ with
  \[ (1-c)\|\widetilde{A}\|^2_F \leq \rho \|\widetilde{A}\|^2_\Ps \leq (1+c) \|\widetilde{A}\|^2_F \]
  for all $\widetilde{A} \in \mathcal{A}(\Lf,\Rf) := \{\tau_r(\Lf, \widetilde{\N}, \Rf) \mid \widetilde{\N} \mbox{ arbitrary }\}$.
\end{definition}
\newwo{Given a tensor $A$, if the iTRIP does not hold, then we can not expect 
the next update to have good completion properties, since changes on the sampling subset are unrelated to changes on the whole space.
As indicated in Example \ref{inofallesqco}, if $c \rightarrow 1$, then the next iterate can be arbitrarily bad.}
Note that the constants \oldwo{are independent of the specific, chosen representation}\newwo{only depend on the tensor $A$, and not
on its representation,} and that this property is easy to check. 
In particular (as proven by Lemma \ref{siTRIP} for $\omega = 0$), the iTRIP with constant $c$ is equivalent to 
\[ \kappa_2(\diag( (\fR^T \otimes \fL)|_{\mathrm{vec}(S(1)),:},\ldots,(\fR^T \otimes \fL)|_{\mathrm{vec}(S(\nN)),:}))^2 \leq \frac{1+c}{1-c},\]
\newwo{where $(\fL, N, \fR)$ corresponds to the standard representation as in \eqref{csp}. Hence, the property}
\oldwo{and hence} holds as long as that matrix has full rank (where $\kappa_2$ is the condition number regarding the spectral norm $\|\cdot\|_2$)
or equivalently $\N \mapsto (\Lf \N \Rf)_\Ps$ is injective. Note that, by a slice wise consideration, the condition
number of a single $(\fR^T \otimes \fL)|_{\mathrm{vec}(S(j)),:}$ can be improved since different magnitudes
of sampling for each slice can be compensated (after all, the slices are solved independently).
\begin{lemma}[Likelihood of the iTRIP]\label{lik}
Let $\mathcal{T}_{\newwo{r}}$ be the subset of $3$ dimensional tensors with rank $r = (\rL,\rR)$.
Let $P$ be a (random) sampling that fulfills $|\Ps(j)| \geq \rL \rR$ 
for all $j = 1,\ldots,\nN$. \old{20}{Then almost every $A \in \mathcal{T}$ has the iTRIP.}
\new{20}{Then, for (only) almost every representation $(L,N,R) \in \R^{n_L \times r_\gamma} \times \R^{r_\gamma \times n_N \times r_\theta} 
\times \R^{r_\theta \times n_R}$ (with respect to the Lebesgue measure), the tensor $A = \tau_r(L,N,R) \in \mathcal{T}_{\newwo{r}}$ has the iTRIP.} \BCn{t1}{t}%
If for one $j$, $|\Ps(j)| < \rL \rR$, then 
no $A \in \mathcal{T}_{\newwo{r}}$ has the iTRIP.
\end{lemma}
\begin{proof}
A tensor $A = \tau_r(\Lf, \N, \Rf)$ has the iTRIP
(for some valid constants) if and only if the linear map $\N \mapsto (\Lf \N \Rf)_\Ps$ is injective, 
or equivalently, $(\Rf^T \kp \Lf)_{\mbox{vec}(\Ps(j)),:}$ has full rank for each $j$.
Due to the provided slice density of $P$, each matrix $(\fR^T \kp \fL)_{\mbox{vec}(\Ps(j)),:}$ is 
of size $|\Ps(j)| \times \rL \rR$. Hence generically, it is of full rank.
If $|\Ps(j)| < \rL \rR$, then the matrix cannot have full rank.
\end{proof}
Tensors themselves that do not have the iTRIP, assuming sufficient sampling, pose just a marginal phenomenon for high dimension $d$.
\oldwo{(in the matrix case for example, some columns or rows may indeed have very few samples).}
A quite simple construction however shows that the iTRIP does not behave well under perturbation: \BCn{m1}{m}%
\begin{lemma}[iTRIP under perturbation]\label{pertrip}
 Let $B \in \R^{n_L \times n_N \times n_R}$ with singular values $(\gamma^{(B)},\theta^{(B)})$.
 Assume further that for one $j$ it holds $|S(j)_1| < n_L$.
 Then for every $\sigma^{\ast} > 0$, there exists a tensor
 $A$ with rank $(r_\gamma,r_\theta)$ and $\|A - B\|^2_F \leq \sum_{i = r_\gamma}^\infty (\gamma^{(B)}_i)^2 + 2 \sum_{i = r_\theta+1}^\infty (\theta^{(B)}_i)^2 + (\sigma^{\ast})^2$ 
 such that $A$ does not have the iTRIP (and $\gamma^{(A)}_{r_\gamma} = \sigma^\ast$). If $B$ already has rank $(r_\gamma,r_\theta)$, then $\|A - B\|^2_F \leq (\gamma^{(B)}_{r_\gamma})^2 + (\sigma^{\ast})^2$ suffices.
\end{lemma}
\begin{proof}
 Truncation of $B$ yields a tensor $\widetilde{A}$ with rank $(\ast,r_\theta)$ and $\|\widetilde{A} - B\|^2_F \leq 
 \delta := \sum_{i = r_\theta+1}^\infty (\theta^{(B)}_i)^2$. The tensor $\widetilde{A}$ hence has perturbed
 singular values such that $\|\gamma^{(\widetilde{A})} - \gamma^{(B)}\|_2 \leq \delta$ (Mirsky's Theorem \cite{Mi1960_Sym}).
 Let $(\widetilde{\fL},\SigL^{(\widetilde{A})},\widetilde{\fN},\SigR^{(\widetilde{A})},\widetilde{\fR})$ be the standard representation of $\widetilde{A}$.
 Without loss of generality, we may assume that $j = 1$ and that only points in the first $k := |S(1)|$ rows of
 $B(1)$ are contained in the sampling $S(1)$. Let now
 \begin{align*}
  \widetilde{\fL}_{:,\{1,\ldots,r_\gamma\}} =: \begin{pmatrix} X & \widetilde{x} \\ Y & \widetilde{y} \end{pmatrix}, \quad \fL := \begin{pmatrix} X & x \\ Y & y \end{pmatrix}, \quad X \in \R^{k \times r_\gamma-1},
 \end{align*}
 If $X$ is already singular, then we may choose $A = \widetilde{A}$. Otherwise, then we may choose
 $x = \alpha X v$, $y = \alpha \widehat{y}$ for $\alpha = \|(X v;\widehat{y})\|_2^{-1}$ and $v = -(X^T X)^{-1} Y^T \widehat{y}$,
 for an arbitrary vector $\widehat{y} \neq 0$. 
 In all three cases, $\fL$ is orthogonal and for $A := \tau_r(\fL,\mathrm{diag}(\sigL^{(\widetilde{A})}_1,\ldots,\sigL^{(\widetilde{A})}_{r_\gamma-1},\sigma^{\ast})\widetilde{\fN} \SigR^{(\widetilde{A})},\widetilde{\fR})$ it holds
 $\|A - B\|_F \leq \|A - \widetilde{A}\|_F + \|\widetilde{A} - B\|_F \leq \sum_{i = r_\gamma}^\infty (\sigma^{(\widetilde{A})}_\gamma)_i^2
 + (\sigma^{\ast})^2 + \delta  \leq \sum_{i = r_\gamma}^\infty (\sigma^{(B)}_\gamma)_i^2 + \delta + (\sigma^{\ast})^2 + \delta$.
 Yet $(\fR^T \otimes \fL)|_{\mathrm{vec}(S(1)),:}$ is a singular matrix, since $\fL_{\{1,\ldots,k\},:}$ is already singular.
\end{proof}
The statement analogously holds true for $r_\theta$ and can easily be transferred to matrix completion as well.
Tensors that do not have the iTRIP are hence densely scattered depending on $\gamma^{(B)}_{r_\gamma}$, as are,
more importantly, surroundings in which the constant $c$ is close to $1$ and overfitting becomes more likely
\oldwo{If the iterate is close to such a tensor the likelihood grows to encounter overfitting} (cf. Example \ref{inofallesqco}).
\oldwo{but the regularization \eqref{Yj} already compensates this.}
\newwo{In case of the regularized update, the additional term $Y(j)$ \eqref{Yj} does not allow the condition number to change that easily (cf. Lemma \ref{siTRIPequiv}).}
\begin{lemma}[Partial matrix inverse by divergent parts]\label{pamainbydipa}
  We partition $\{1,\ldots,n\} = \omega_j \cup \omega_j^c$ $(\omega_j^c = \{1,\ldots,n\} \setminus \omega_j)$, $j = 1,2$
  and define
  $\Omega := \omega_1 \times \omega_2$, $\widetilde{\Omega} := \omega^c_1 \times \omega^c_2$.
  Let $\{A^{(k)}\}_k, \ \{J^{(k)}\}_k \subset \R^{n \times n}$ be series of symmetric matrices, $supp(J^{(k)}) \subset \Omega$. \\
  If $\mbox{lim}_{k \rightarrow \infty} A^{(k)}|_{\widetilde{\Omega}} = A|_{\widetilde{\Omega}}$, $A|_{\widetilde{\Omega}}$ s.p.d, and
  $\sigma_{\mathrm{min}}(J^{(k)}|_{\Omega}) \rightarrow \infty$, then
  $\newwo{V} := \mbox{lim}_{k \rightarrow \infty} (A^{(k)} + J^{(k)})^{-1}$ exists and we
  have $\newwo{V}|_{\widetilde{\Omega}} = (A|_{\widetilde{\Omega}})^{-1}$ and $\newwo{V}|_{\widetilde{\Omega}^c} = 0$ $(\widetilde{\Omega}^c = \{1,\ldots,n\}^2 \setminus \widetilde{\Omega})$.
\end{lemma}
\begin{proof}
  First, w.l.o.g., let $\Omega = \{m+1,\ldots,n\}^2$. Otherwise we can apply permutations. Further, let
  $\newwo{V}^{(k)} := A^{(k)} + J^{(k)}$. We partition our (symmetric) matrices $M$ for $M_{1,1} \in \R^{m \times m}$ block-wise as
  \[ M = \begin{pmatrix} M_{1,1} & M_{1,2} \\ M_{1,2}^T & M_{2,2} \end{pmatrix}. \]
  Note that $J^{(k)}_{1,1}, J^{(k)}_{1,2} \equiv 0$. Since $A^{(k)}_{1,1} = \newwo{V}^{(k)}_{1,1}$ and $A_{1,1} = A|_{\widetilde{\Omega}}$ is s.p.d,
  $A^{(k)}_{1,1}$ is invertible for all $k > K$ for some $K$ and hence $\mbox{lim}_{k \rightarrow \infty} (\newwo{V}^{(k)}_{1,1})^{-1} = A_{1,1}^{-1}$.
  Further, $\sigma_{\mathrm{min}}(B_{2,2}^{(k)}) > \sigma_{\mathrm{min}}(J_{2,2}^{(k)}) - \sigma_{\mbox{max}}(A_{2,2}^{(k)}) \rightarrow \infty$
  and hence $\|(\newwo{V}^{(k)}_{2,2})^{-1}\| \rightarrow 0$.
  Therefore, for $k > \widetilde{K}$ and $H^{(k)} := \newwo{V}^{(k)}_{1,1} - \newwo{V}^{(k)}_{1,2} (\newwo{V}^{(k)}_{2,2})^{-1} (\newwo{V}^{(k)}_{1,2})^T$, it is $\sigma_{\mathrm{min}}(H^{(k)}) > \sigma_{\mathrm{min}} (A_{1,1})/2$.
  By block-wise inversion of $\newwo{V}^{(k)}$, it then follows $((\newwo{V}^{(k)})^{-1})_{1,1} = (H^{(k)})^{-1} \rightarrow (A^{(k)}_{1,1})^{-1}$. Similarly, $((\newwo{V}^{(k)})^{-1})|_{\Omega} \rightarrow 0$.
\end{proof}
One last step remains, since we cannot allow $\zeta$ to depend on the rank $r$. 
For now, we redefine the method $\mathcal{M}^{\ast}$ 
to directly yield the result in Theorem \ref{miofthalsrerefu} for arbitrary constants $\zeta$, i.e.
\begin{align}
  \mathcal{M}_\zeta^{\ast}(\Lf,\ \N,\ \Rf) := (\fL,\ \N^+,\ \fR). \label{Mzeta} 
\end{align}
We explain in Section \ref{sec:backtod} and Lemma \ref{retafu} 
how the scalings $\gammaL,\gammaR$ as well as $\omega$ are used to obtain
one specific $\mathcal{M}_{\zeta}^{\ast}$ from $\mathcal{M}^{\ast}$, for which $
\zeta$ is indeed independent of $r$. \oldwo{It is easy to see that $\mathcal{M}_\zeta^{\ast}$ is no less
representation independent.}\BCn{u1}{u}%
\begin{theorem}[Stability of the method $\mathcal{M}_\zeta^{\ast}$]\label{coofthfiadfmist} 
  Let $\B$ be the target tensor, $\Ps \newwo{\ \subsetneq \mathcal{I}}$ the sampling set, arbitrary but fixed, and $\mathcal{M}_\zeta^{\ast}$ as in \eqref{Mzeta}.
  \begin{itemize}
  \item The regularized method $\mathcal{M}_\zeta^{\ast}$ $(\omega > 0)$ as defined by \eqref{Mzeta}  
  (for $\zeta_1,\zeta_2 \geq 0$ and $\zeta_{(1,2)} > 0$ that do not depend on $r$) is stable 
  at all points $A^{\ast}$ (and hence also \old{2}{fix-rank}\new{2}{fixed-rank} stable).
  \item The unregularized method $(\omega = 0)$ \eqref{unregM} provides stability only for fixed rank (cf. Example \ref{alsanadfarnocowitr}), and only at those points $A^{\ast}$  that have the iTRIP (cf. Def. \ref{iTRIP}).
  \end{itemize}  
\end{theorem}
\begin{proof}
  Let $A^{\ast}$ be a fixed tensor with TT-ranks $r^{\ast}$. \\
  \textit{1. \old{2}{Fix-rank}\new{2}{Fixed-rank} stability:} We first show that $\mathcal{M}^{\ast}$ is stable for fixed rank. Let $A_i$ be a sequence 
  with $\mbox{rank}(A_i) = r^{\ast}$ and $A_i \rightarrow A^{\ast}$. 
  Let $\mathcal{G}^{\ast} =(\fL^{\ast},\SigL^{\ast},\fN^{\ast},\SigR^{\ast},\fR^{\ast})$ be the standard representation of $A^{\ast}$ as well
  as $\mathcal{G}_i$ correspond to $A_i$. We partition the indices for $\sigL^{\ast}$ and $\sigR^{\ast}$ by $k$ and $\ell$ according to
  equality of entries, such that $\sigL^{\ast}_1 = \ldots = \sigL^{\ast}_{k_1} > \sigL^{\ast}_{k_1+1} = \ldots = \sigL^{\ast}_{k_2} > 
  \ldots > \sigL^{\ast}_{k_{K-1}+1} = \ldots = \sigL^{\ast}_{k_K} > 0$ and likewise for $\ell_1,\ldots,\ell_L$.
  Since $A_i \rightarrow A^{\ast}$, their singular values also converge (e.g. \cite{We1912_Das}).
  We can hence conclude from \cite{Do20_Ano,We72_Per} that there exist sequences of block diagonal, orthogonal matrices $W_i$ and $M_i$ with block sizes $k_1, k_2-k_1, \ldots, k_K-k_{K-1}$ and
  $\ell_1, \ell_2-\ell_1, \ldots, \ell_L-\ell_{L-1}$, respectively, such that
  \begin{align}
    \label{ast} \|\fL_i W_i - \fL^{\ast}\|_F \rightarrow 0 \quad \mbox{and} \quad \|M_i \fR_i - \fR^{\ast}\|_F \rightarrow 0, \quad 
  \end{align}
  since the standard representation includes left and right singular vectors. 
  We have to show that the tensors $\newwo{H_i} = \tau_r(\fL_i, \N_i, \fR_i) = \tau_r(\mathcal{M}^{\ast}(\fL_i, \SigL_i \fN_i \SigR_i, \fR_i))$ 
  converge to the analogously defined $\newwo{H^{\ast}}$.
  For fixed $j$, we define for each single $\mathcal{G}_i$ the matrix $Y_i = Y(j)$ (cf. Theorem \ref{miofthalsrerefu}, Remark \ref{normaleq})
  and $z_i := (\fR_i^T \kp \fL_i)$ such that
  \begin{align} N_i(j) & = \argmin_{\widetilde{N}(j)} \left\|\begin{pmatrix} (z_i)_{\mbox{vec}(\Ps(j)),:} \\
      Y_i \end{pmatrix} \mbox{vec}(\widetilde{N}(j)) - 
    \begin{pmatrix} \mbox{vec}(\B(j))|_{\mbox{vec}(\Ps(j))} \\ 0 \end{pmatrix} \right\|, \label{Ni} \\
    \mbox{vec} (\newwo{H_i}(j)) & = z_i \ \mbox{vec}(N_i(j)). \nonumber 
  \end{align}
  We define the shifted matrices
  \begin{align*} z^{M,W}_i & := (M_i \ \fR_i)^T \kp (\fL_i \ W_i) \\
    Y^{M,W}_i & := \begin{pmatrix}
      \sqrt{\oldwojg{\nu_{s-1} \zeta^{(\mu)}_1}\newwo{\nL^{-1} \zeta_1}} \ (M_i \ {\fR_{:,\Ps(j)_2}})^T \kp (\SigL_i^{-1} \ W_i) \\
      \sqrt{\oldwojg{\nu_{s} \zeta^{(\mu)}_2}\newwo{\nR^{-1} \zeta_2}} \ (\SigR_i^{-1} \ M_i^T) \kp (\fL_{\Ps(j)_1,:} \ W_i) \\
      \sqrt{\oldwojg{\nu_{s-1,s} \zeta^{(\mu)}_{(1,2)}}\newwo{|\Ps(j)| \nR^{-1}\nL^{-1} \zeta_{(1,2)}}} \ (\SigR_i^{-1} \ M_i^T) \kp (\SigL_i^{-1} \ W_i)
    \end{pmatrix} 
  \end{align*}
  Due to \eqref{ast}, it holds $(z^{M,W}_i)_{\mbox{vec}(\Ps(j)),:} \rightarrow z^{\ast}_{\mbox{vec}(\Ps(j)),:}$.
  Inserting $I = (M_i^T \kp W_i) (M_i^T \kp W_i)^T$ into \eqref{Ni}, we obtain
  \begin{align*}
    \mbox{vec} (\newwo{H_i}(j)) = z^{M,W}_i \ & \left( (z^{M,W}_i)_{\mbox{vec}(\Ps(j)),:}^T \ (z^{M,W}_i)_{\mbox{vec}(\Ps(j)),:} 
    + {Y^{M,W}_i}^T \ Y^{M,W}_i \right)^{-1} \\
    & \cdot (z^{M,W}_i)_{\mbox{vec}(\Ps(j)),:}^T \ \mbox{vec}(\B(j))|_{\mbox{vec}(\Ps(j))}.
  \end{align*}
  Since $W_i^T \SigL^{\ast} W_i = \SigL^{\ast}$ for all $i$, it follows $W_i^T \SigL_i W_i \rightarrow \SigL^{\ast}$.
  Likewise $M_i^T \SigR_i M_i \rightarrow \SigR^{\ast}$ and thereby also 
  ${Y^{M,W}_i}^T \ Y^{M,W}_i \rightarrow {{Y}^{\ast}}^T \ {Y}^{\ast}$. We treat the cases $\omega = 0$ and $\omega > 0$ separately: \\
  \textit{(i) $\omega = 0$}: In this case, $Y^{M,W}_i = 0 = Y^{\ast}$.
  If the iTRIP holds for $A^{\ast}$, then \\ $\sigma_{\mathrm{min}}(z^{\ast}_{\mbox{vec}(\Ps(j)),:}) > 0$ and therefore \\
  \[ \left( (z^{M,W}_i)_{\mbox{vec}(\Ps(j)),:}^T \ (z^{M,W}_i)_{\mbox{vec}(\Ps(j)),:}\right)^{-1} \rightarrow \left( (z^{\ast})_{\mbox{vec}(\Ps(j)),:}^T \ (z^{\ast})_{\mbox{vec}(\Ps(j)),:}\right)^{-1}. \]
  This directly yields convergence of $(\newwo{H_i}(j)) \rightarrow (\newwo{H^{\ast}}(j))$ since all involved factors converge. \\
  \textit{(ii) $\omega > 0$}: Here, we use that $\sigma_{\mathrm{min}}(Y^{\ast}) > 0$ and $\sigma_{\mathrm{min}}(z^{\ast}_{\mbox{vec}(\Ps(j)),:}) \geq 0$.
  We then obtain convergence since
  \begin{align*} 
    & \left( (z^{M,W}_i)_{\mbox{vec}(\Ps(j)),:}^T \ (z^{M,W}_i)_{\mbox{vec}(\Ps(j)),:} 
    + {Y^{M,W}_i}^T \ Y^{M,W}_i \right)^{-1} \\
    & \rightarrow  \left( (z^{\ast})_{\mbox{vec}(\Ps(j)),:}^T \ (z^{\ast})_{\mbox{vec}(\Ps(j)),:} 
    + {Y^{\ast}}^T \ Y^{\ast} \right)^{-1}.
  \end{align*}
  This proves \old{2}{fix-rank}\new{2}{fixed-rank} stability. \\ 
  \textit{2. Stability:} Let now $A_i$ have arbitrary ranks. Without loss of generality by consideration of a finite amount of infinite subsequences, 
  we can assume that $\mbox{rank}(A_i) \equiv r$ for all $i$. Then, since $TT(r^{\ast})$ is a manifold, it follows $\sigL \geq \sigL^{\ast}$ and $\sigR \geq \sigR^{\ast}$.
  We can therefore have singular values $(\sigL_i)_{k_K+1},\ldots,(\sigL_i)_{k_{K+1}} \rightarrow 0$ 
  as well as $(\sigR_i)_{\ell_L+1},\ldots,(\sigR_i)_{\ell_{L+1}} \rightarrow 0$.
  We expand the matrices $W_i$ and $M_i$ by identities of appropriate sizes
  to account for the vanishing singular values: $W_i \leftarrow diag(W_i,I_{k_{K+1} - k_K})$,
  $M_i \leftarrow diag(M_i,I_{\ell_{L+1} - \ell_L})$.
  In regard of Proposition \ref{pamainbydipa}, let $\Omega$ be the smallest cross product set, such that 
  $({Y^{M,W}_i}^T \ Y^{M,W}_i)|_{\widetilde{\Omega}}$ converges (which is the set that corresponds to vanishing singular values). Then, due to the
  definition of $Y^{M,W}_i$, $\sigma_{\mathrm{min}}(({Y^{M,W}_i}^T \ Y^{M,W}_i)|_{\Omega}) \rightarrow \infty$.
  We can conclude that 
  \begin{align*} 
    & \left.\left(\left( (z^{M,W}_i)_{\mbox{vec}(\Ps(j)),:}^T \ (z^{M,W}_i)_{\mbox{vec}(\Ps(j)),:} 
    + {Y^{M,W}_i}^T \ Y^{M,W}_i \right)^{-1}\right)\right|_{\widetilde{\Omega}} \\
    & \rightarrow  \left( (z^{\ast})_{\mbox{vec}(\Ps(j)),:}^T \ (z^{\ast})_{\mbox{vec}(\Ps(j)),:} 
    + {Y^{\ast}}^T \ Y^{\ast} \right)^{-1}.
  \end{align*}
  and 
  \begin{align*} 
    & \left.\left(\left( (z^{M,W}_i)_{\mbox{vec}(\Ps(j)),:}^T \ (z^{M,W}_i)_{\mbox{vec}(\Ps(j)),:} 
    + {Y^{M,W}_i}^T \ Y^{M,W}_i \right)^{-1}\right)\right|_{\widetilde{\Omega}^c} \rightarrow 0.
  \end{align*}
  Because of this restriction, we in turn again get convergence to the limit $(\newwo{H_i}(j)) \rightarrow (\newwo{H^{\ast}}(j))$,
  since all parts that correspond to vanishing singular values, also vanish within the update. This finishes the proof.
\end{proof}

\begin{definition}[Stabilized internal tensor restricted isometry property (siTRIP)]\label{siTRIP}
  We say a rank $r$ tensor $A = \tau_r(\Lf, \N, \Rf)$ has the stable internal tensor restricted isometry property for the sampling set $\Ps$,
  if there exist $0 \leq c < 1$ and $\rho > 0$ such that for all $\widetilde{N}$ holds
  \begin{align} (1-c) \int_{\mathcal{V}_\omega(\Lf,\N,\Rf)} \|\widetilde{A}_\Delta\|^2_F \intd \Delta
  \leq \rho \int_{\mathcal{V}_\omega(\Lf,\N,\Rf)} \|\widetilde{A}_\Delta\|^2_{\Ps}  \intd \Delta 
  \leq (1+c) \int_{\mathcal{V}_\omega(\Lf,\N,\Rf)} \|\widetilde{A}_\Delta\|^2_F  \intd \Delta \label{siTRIPineq} \end{align}
  where $\widetilde{A}_\Delta := \tau_r(\Lf+\Delta \Lf, \widetilde{\N} + \Delta \N, \Rf + \Delta \Rf)|^2_F$
  and $\intd \Delta = \intd \Delta \Lf \intd \Delta \N \intd \Delta \Rf$.
\end{definition}
The constants are independent of the specific representation (cf. Lemma \ref{in}).
\begin{lemma}\label{siTRIPequiv}
Let $Z(j) = \begin{pmatrix} (\fR^T \otimes \fL)|_{\mathrm{vec}(S(j)),:} \\ Y(j) \end{pmatrix}$ as in \eqref{normaleq}.
 The siTRIP with constant $c$ for $A \in \R^{\Ind}$ is equivalent to
 \[ \exists \ c>0: \quad\kappa_2(\diag(Z(1) \filter^{1/2},\ldots,Z(\nN) \filter^{1/2}) ^2 \leq \frac{1 + c}{1 - c}, \]
 where $A = \tau_r(\fL,\N,\fR)$ is its standard representation and $\filter$ is as in Lemma \ref{fipr}.
\end{lemma}
\begin{proof}
Let $Z = \diag( Z(1),\ldots,Z(\nN) )$ and $\mathrm{vec}(\widetilde{\N})^T = (\mathrm{vec}(N(1))^T,\ldots,\mathrm{vec}(N(\nN))^T)$.
 We abbreviate the siTRIP \eqref{siTRIPineq} as $(1-c) \beta \leq \rho \xi \leq (1+c) \beta$ ($\beta = \beta(\widetilde{\N}), \ \xi = \xi(\widetilde{\N})$).
 Let $a:= \max_{x \neq 0} \frac{\|Z x\|^2}{\|(I \otimes \mathcal{F}^{-1/2}) x \|^2}$ and $b:=\min_{x \neq 0} \frac{\|Z x\|^2}{\|(I \otimes \mathcal{F}^{-1/2}) x \|^2}$.
 Since the perturbation $\Delta \N$ is independent of $\widetilde{N}$, this term can be neglected in consideration of that $\|\widetilde{N}\|$
 is not bounded. With $P = \Ind$ it holds $\beta = \sum_{j=1}^{\nN} \mathrm{vec}(\widetilde{\N}(j))^T \filter^{-1} \mathrm{vec}(\widetilde{\N}(j)) = 
 \|(I \otimes \mathcal{F}^{-1/2}) \mathrm{vec}(\widetilde{\N})\|^2_F$ (cf. proof of Theorem \ref{miofthalsrerefu}). For the actual sampling $P$,
 we have $\xi = \| Z \mathrm{vec}(\widetilde{\N}) \|^2_2$. Thereby $a = \max_{\widetilde{\N} \mid \beta = 1} \xi$
 and $b = \min_{\widetilde{\N} \mid \beta = 1} \xi$. Now, given the siTRIP, it follows
 \[ \mathrm{cond} (Z (I \otimes \mathcal{F}^{1/2}))^2 = \frac{a}{b} \leq \frac{\rho^{-1} (1+c)}{\rho^{-1} (1-c)} = \frac{1+c}{1-c}. \]
 For the opposite implication, define $\rho = \frac{1-c}{b}$. Then
 \[ \rho \xi \leq a \frac{1-c}{b} \beta \leq (1+c) \beta \quad \mbox{and} \quad \rho \xi \geq b \frac{1-c}{b} \beta = (1-c) \beta. \] 
\end{proof}
The siTRIP holds for \textit{every} tensor (possibly with $c$ close to $1$). For $\omega \rightarrow 0$, the constant $c$ converges
to the one of the iTRIP and for $\omega \rightarrow \infty$, \oldwo{$c \rightarrow 1$}\newwo{$c \rightarrow 0$}.
As well as for Definition \ref{iTRIP}, a slice wise consideration yields a better condition number for just $Z(j) \filter^{1/2}$. 
Note that this value not only \oldwo{determines}\newwo{bounds} the number of steps required for the cg method, but appears to be important for the reconstruction
quality \oldwo{of the tensor}\newwo{obtained through one microstep}. \newwo{However, a further investigation into the siTRIP, how exactly it behaves under perturbation
and if it requires modifications, is a matter of future research.}\\\\
\newwo{
As for the matrix case, we have to limit the singular values from below by a decreasing value proportional to the current residual, cf. Section \ref{sec:filter}. This leads to a slight complication, which
is resolved through the following, simple corollary to Theorem \ref{miofthalsrerefu}.
\begin{corollary}
 Let $A$ and $\widetilde{A}$ be tensors with standard representations $(\fL,\Gamma,\fN,\Theta,\fR)$
 and $(\fL,\Gamma,\widetilde{\fN},\Theta,\fR)$, respectively. Then both yield the same update $N^+$.
\end{corollary}
Hence, if we want to modify the singular values $\gamma$ and $\theta$, we may do so without 
knowledge about an appropriate core $\widetilde{\fN}$ (for example in the sense of some unknown best approximation).
In the least squares problem in Theorem \ref{miofthalsrerefu}, we therefor
simply set $\gamma_i := \max(\gamma_i,\sigma_{\mathrm{min}})$, $i = 1,\ldots,r_{\mu-1}$ and 
$\theta_i := \max(\theta_i,\sigma_{\mathrm{min}})$, $i = 1,\ldots,r_{\mu}$ (and can
thereby also ignore that the combination of the new $\gamma$ and $\theta$ might not be feasible, cf. \cite{GrKrxx_The}} 

%
\section{Behavior of the SALSA Filter}\label{sec:filter}\BCn{j1}{j}
A deeper understanding of the regularization utilized by SALSA and the reason for the lower bound $\sigma_{\min}$ is provided by the \textit{filter} as indicated
by \eqref{filterex} for the matrix case and as defined by Corollary \ref{fipr} for tensors.
Throughout this section, we assume that the sampling is such that \newwo{for}
the minimizer in Theorem \ref{miofthalsrerefu} \oldwo{is basically equal to}\newwo{it (approximately) holds}
\begin{align} 
  \mathrm{vec}(\N^+(j)) & = \filter \ \mathrm{vec}( ({\fL}^T \ \B(j) \ {\fR}^T) ) \label{filtermin},
\end{align}
which \oldwo{at last holds}\newwo{is true at last} for $P = \Ind$ (cf. Corollary \ref{fipr}).
Since $\zeta_1 \zeta_2 = \zeta_{(1,2)}$ (see the later equation \eqref{zetadef}), we can rewrite
\begin{align*} 
  \N^+ & = D_{\omega^2\zeta_1}({\SigL}) \ ({\fL}^T \ \B \ {\fR}^T) \ D_{\omega^2\zeta_2}({\SigR}) \\
  D_c({\Sigma}) & := (I + c \Sigma^{-2})^{-1}.
\end{align*}
We are interested in the fixpoints of this update, i.e. we postulate
$\N^+ = \SigL \ \fN \ \SigR$. Then, since $\rhb{\fN \ \SigR}$ is row-orthogonal (cf. Lemma \ref{stre}),
it holds
\begin{align}
 \label{fixp} D_{\omega^2\zeta_1}(\SigL) \ Z & = \SigL, \\
 \nonumber Z & = \rhb{({\fL}^T \ \B \ {\fR}^T) \ D_{\omega^2\zeta_2}(\SigR)} \ \rhb{\fN \ \SigR}^T,
\end{align}
where $Z =: \diag(\sigma^{(Z)})$ is necessarily a diagonal matrix (certainly, an analogous argument holds for $\SigR$ as well).
Because \eqref{fixp} can only hold if $d_{\sigma^{(Z)},\zeta_1}(\sigL_i) = \sigL_i$ for all $i$, 
the focus of our analysis lies on the fixpoints of the function $d_{\sigma^{(Z)},c}: \sigma \mapsto (1 + c \sigma^{-2})^{-1} \sigma^{(Z)}$. 
For each pair $(\sigma^{(Z)},c)$, the only attractive fixpoint (if existent) is given by $f_{\mathrm{stab}} = \frac{1}{2} \sigma^{(Z)} + \frac{1}{2} \sqrt{(\sigma^{(Z)})^2 - 4 c}$
and the repelling one by $f_{\mathrm{rep}} = \frac{1}{2} \sigma^{(Z)} - \frac{1}{2} \sqrt{(\sigma^{(Z)})^2 - 4 c}$.
At the point where $f_{\mathrm{stab}} = f_{\mathrm{rep}}$, it holds $\sigma = \sqrt{c} = \frac{1}{2} \sigma^{(Z)}$. The minimal value which
the term $(1 + c \sigma^{-2})^{-1}$ can hence take in any attractive fixpoint, is $F = 1/2$. This
behavior is shown in Figure \ref{filter_behaviour}. 
\begin{figure}
  \begin{center}
      
      \ifuseprecompiled
      \includegraphics[width=0.98\textwidth]{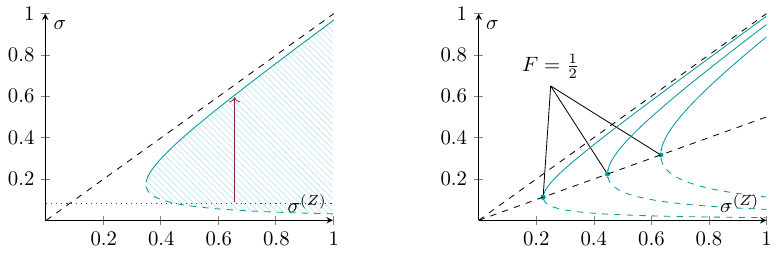}
      \fi
  \end{center}
  \caption{\label{filter_behaviour} Left: Plotted are the fixpoints (continuous for attractive, dashed for repelling ones, in teal) 
  of $d_{\sigma^{(Z)},c}$ for one fixed $c$ with respect to $\sigma^{(Z)}$. Within the hatched area, singular values rise until
  they reach the upper boundary. A lower limit to the singular values is indicated as dotted, magenta line. 
  Right: Different values of $c$ are considered. The turning point 
  $\sigma = c = \frac{1}{2} \sigma^{(Z)}$ corresponds to a filter value of $1/2$. }
\end{figure}
\oldwojg{The relation to the filter is given by
\begin{align*}
\filter_{i,1} & = (D_{\omega^2\zeta_1}({\SigL}))_{i,i} \cdot \underbrace{(D_{\omega^2\zeta_2}({\SigR}))_{1,1}}_{\approx 1}. 
\end{align*}}%
A stabilized singular value corresponds to some attractive fixpoint of $d_{\sigma^{(Z)},c}$ \newwo{(cf. Definition \ref{virtstabsv}),
and it necessarily holds $(D_{\omega^2\zeta_1}({\SigL}))_{i,i} > 0.5 \Leftrightarrow \sigL_i > \omega \sqrt{\zeta_1}$.}\oldwo{Therefore, it 
necessarily holds $(D_{\omega^2\zeta_1}({\SigL}))_{i,i} > 0.5$. In practice, the value $\lowfil_{\mathrm{stab}}$ should be chosen larger,
as well as $\lowfil_{\mathrm{virt}}$ lower, not only to reduce the computational cost, but also to avoid premature reactions within the optimization.
Since the singular values $\SigL$ take part in another, neighboring microstep as well,
the accordant value is also taken into account (cf. \ref{Gamma}). \\}
\oldwo{It is now easy to understand why a lower limit to all singular values is required.}\newwo{This explains why the lower limit $\sigma_{\min}$ is 
necessary.} As displayed in Figure \ref{filter_behaviour} (left),
for any fixed $\sigma^{(Z)}$, a singular value $\sigma$ must be above a certain threshold (that corresponds to the repelling fixpoint) to
be increased by \oldwo{an accordant. So we cannot allow it to converge to zero.}\newwo{the microstep. It therefore must not converge to zero.} 
\section{Results Transferred Back to a d-Dimensional Tensor}\label{sec:backtod}
In this Section, we \oldwo{return to a $d$-dimensional tensor}\newwo{transfer the previous results for $S = P_{(\mu)}$ and $B = M_{(\mu)}$
to the $d$-dimensional setting}. 
In Remark \ref{constants}, we have
$\#_\fR = size(\fR) = \rR \prod_{i=s+1}^{d} n_i$, $\#_\fN = size(\fN) = \rL n_s \rR$, $\#_\fL = size(\fL) = \rL \prod_{i=1}^{s-1} n_i$. 
By combining modes (cf. Definition \ref{reduction}), the \textit{sizes} of the left as well as right side
have been \oldwo{drastically overrated and }distorted,
considering that the degrees of freedom of $\fL = \mathcal{G}^{<s}$ and $\fR = \mathcal{G}^{>s}$
are given by a sum, not a product, of the degrees of freedom of the single modes \newwo{(ignoring minor gauge conditions)}. 
We choose one of the few remaining options through which the method becomes stable.
We artificially set 
\[ \#_\fR \leftarrow r_\mu \sum_{i=\mu+1}^{d} n_i, \quad \#_\fL \leftarrow r_{\mu-1} \sum_{i=1}^{\mu-1} n_i \]
using appropriate scalings $\gammaL = \gammaL^{(\mu)}$, $\gammaR = \gammaR^{(\mu)}$ (differently for each mode $\mu$).
Otherwise, we will not obtain a stable microstep. 
Furthermore, the near common parts of the denominators, $\#_\fR + \#_\fN + \#_\fR + 2 (+2)$, can be incorporated into $\omega^2$, 
so we omit them in the following sense:
\begin{lemma}[Rescaled target function]\label{retafu}
  The previously discussed rescaling is achieved by choosing
  \begin{align*} (\gammaL^{(\mu)})^2 & = E \frac{\sum_{s=1}^{\mu-1} n_s}{\left(\prod_{s=1}^{\mu-1} n_s \right) \sum_{s=1}^{d} n_s}, 
    \quad (\gammaR^{(\mu)})^2 = E \frac{\sum_{s=\mu+1}^{d} n_s}{\left( \prod_{s=\mu+1}^{d} n_s \right) \sum_{s=1}^{d} n_s}, \\
    E & = r_\mu \prod_{s=\mu+1}^{d} n_s +  r_{\mu-1} n_\mu r_\mu + r_{\mu-1} \prod_{s=1}^{\mu-1} n_s
  \end{align*}
  Thereby,
  \begin{align}
    \label{zetadef} \zeta^{(\mu)}_1 = \frac{\sum_{s=1}^{\mu-1} n_s}{\sum_{s=1}^{d} n_s}, \quad \zeta^{(\mu)}_2 = \frac{\sum_{s=\mu+1}^{d} n_s}{\sum_{s=1}^{d} n_s}, \quad 
   \zeta^{(\mu)}_{(1,2)} = \zeta^{(\mu)}_1 \zeta^{(\mu)}_2 (1 + \mathcal{O}(E^{-1})).
  \end{align}
\end{lemma}
\begin{proof}\BCn{f4_}{f_}%
  \oldwo{See appendix A.}\newwo{First, 
$\zeta^{(\mu)}_1 = \gammaL^2 \frac{\prod_{s=1}^{\mu-1} n_s}{E} = \frac{\sum_{s=1}^{\mu-1} n_s}{\sum_{s=1}^{d} n_s}$,
with an analog result for $\zeta^{(\mu)}_2$. For the mixed term, we have
$\zeta^{(\mu)}_{(1,2)} = \zeta^{(\mu)}_{1} \zeta^{(\mu)}_{2} \frac{E}{E+2} = \zeta^{(\mu)}_{1} \zeta^{(\mu)}_{2} (1 - \frac{2}{E+2})$.}
\end{proof}
The value $E^{-1}$ is in general far below machine accuracy, such that we (from now on) ignore the factor $(1 + \mathcal{O}(E^{-1}))$.
There might be a more suitable realization of this result and
it should be remarked that the exact scalings are not important for the validity of Theorem \ref{coofthfiadfmist}. \BCn{o3}{o}%
In this context, for fixed $\mu$, the matrices $\fL_{\Ps(j)_1,\cdot\cdot} \in \R^{\oldwojg{a_j}\newwo{|P_{(\mu)}(j)|} \times \rL}$ and $\fR_{:,\Ps(j)_2} \in \R^{\rR \times \oldwojg{a_j}\newwo{|P_{(\mu)}(j)|}}$ 
\oldwo{, $a_j = |\{ p \mid p \in P,\ p_\mu = j\}| = |P_{(\mu)}(j)|$} (cf. \eqref{rrpex}), are given by 
\begin{align}
  \left(  \fL_{\Ps(j)_1,\cdot\cdot}  \right)_{\ell,\cdot\cdot} & = G_1(p^{(i_\ell)}_1) \cdot \ldots \cdot \ G_{\mu-1}(p^{(i_\ell)}_{\mu-1}) \label{fLPsj} \newwo{\ = G_{1,\ldots,\mu-1}((p^{(i_\ell)}_1,\ldots,p^{(i_\ell)}_{\mu-1}))}, \\
  \left(  \fR_{:,\Ps(j)_2}  \right)_{:,\ell} & = G_{\mu+1}(p^{(i_\ell)}_{\mu+1}) \cdot \ldots \cdot \ G_{d}(p^{(i_\ell)}_{d}) \label{fRPsj} \newwo{\ = G_{\mu+1,\ldots,d}((p^{(i_\ell)}_{\mu+1},\ldots,p^{(i_\ell)}_d))},
\end{align}
\movedwo{for $p^{(i_\ell)} \in P_{(\mu)}(j)$, $\ell = 1,\ldots,\oldwo{a_j}\newwo{|P_{(\mu)}(j)|}$} \newwo{and} a representation $G$ for which $\fL = G^{<s}$ and $\fR = G^{>s}$ \newwo{(cf. \eqref{superdupercores})}.
\begin{remark}[Case $\mu = 1,d$]\label{casemu1d}
  For $\mu = 1,d$ in Theorem \ref{miofthalsrerefu}, the same formula can be used by formally setting $G^{<1} = \fL = 1$, $G^{>d} = \fR = 1$ and
$\zeta^{(1)}_1 = 0$, $\zeta^{(d)}_2 = 0$, $\zeta^{(1)}_{(1,2)}, \zeta^{(d)}_{(1,2)} = 0$, respectively. These comply with the result in the matrix case (cf. Remark \ref{constants_matrices}).
\end{remark}
Since all microsteps $\mathcal{M}^{\ast}$ are stable, we call this regularized ALS method stable - hence the name SALSA (Stable ALS Approximation).
We summarize in Algorithm \ref{fiadfsw} one full left sweep $\mu = 1 \rightarrow d$ of SALSA for fixed rank $r$.
Note that in practice, the complexity is reduced to the minimal necessary order in the optimal case (cf. Remark \ref{applofcg}).
The simpler matrix case ($d = 2$) is carried out in Algorithm \ref{MatrixCompletion}.
\begin{algorithm}
  \caption{SALSA Sweep \label{fiadfsw}}
  we here identify $\widetilde{\Sigma} = \diag(\widetilde{\sigma})$
  \begin{algorithmic}[1]
    \REQUIRE limit $\sigma_{\mathrm{min}}$, parameter $\omega$, initial guess $A = \tau_r(G)$ for which $\rhb{G_2},\ldots,\rhb{G_d}$ are row-orthogonal
     and data points $M|_P$
    \FOR{$\mu=1,\ldots,d$}
    \IF{$\mu \neq 1$}
    \STATE compute the SVD $U \widetilde{\Sigma} V^T := \lhb{G_{\mu-1}}$ and set $\sigma_i^{(\mu-1)} := \max(\widetilde{\sigma}_i,\sigma_{\mathrm{min}})$, $i = 1,\ldots,r_{\mu-1}$
    \STATE set $G_{\mu-1}$ via $\lhb{G_{\mu-1}} = U$ and $G_{\mu} := \widetilde{\Sigma} V^T G_{\mu}$
    \ENDIF
    \IF{$\mu \neq d$}
    \STATE compute the SVD $U \widetilde{\Sigma} V^T := \lhb{G_\mu}$ and set $\sigma_i^{(\mu)} := \max(\widetilde{\sigma_i},\sigma_{\mathrm{min}})$, $i = 1,\ldots,r_\mu$
    \STATE update $G_{\mu+1} := V^T G_{\mu+1}$ and $G_{\mu}$ via $\lhb{G_\mu} = U \widetilde{\Sigma}$
    \ENDIF
    \FOR{$j=1,\ldots,n_\mu$}
    \STATE update $G_\mu(j) := \N(j)$ by solving the least squares problem
    in Theorem \ref{miofthalsrerefu} for $\fL = G^{<s}, \ \fR = G^{>s}, \ \gamma = \sigma^{(\mu-1)}, \ \theta = \sigma^{(\mu)}$ (cf. Remark \ref{casemu1d})
    using coarse cg (cf. Remark \ref{applofcg})
    \ENDFOR        
    \ENDFOR
  \end{algorithmic}
\end{algorithm}
\section{Semi Implicit and Non Uniform Rank Adaption}\label{sec:seimannounraad} \BCn{m3}{m}%
The rank adaption for tensor completion is carried out analogously to the matrix case, Section \ref{sec:matrixra}.
The exact choices of the following parameters are not important, such that we only indicate them roughly. The specific values which
we used in all numerical tests are provided in Section \ref{sec:tupa}.
The number of minor singular values (cf. Definition \ref{virtstabsv}) for each matricization is kept constant, such that
\begin{align}
 |\{ i \mid 0 < \sigma^{(\mu)}_i < f_{\mathrm{minor}} \cdot \omega \}| \overset{!}{=} k_{\mathrm{minor}}, \label{rankincr}
\end{align}
for each $\mu = 1,\ldots,d-1$
for certain constants $f_{\mathrm{minor}} < 1$ and $k_{\mathrm{minor}} \in \mathbb{N}$, subject to the theoretical bound $r_{\mu} \leq \min(n_{\mu} r_{\mu-1}, n_{\mu+1} r_{\mu+1})$. 
Furthermore, a common, upper limit $r_{\mu} \leq r_{\mathrm{lim}}$ is applied, which is chosen large enough, but likewise in order
to avoid unnecessary computation time. The factor $f_{\mathrm{minor}}$ is related to the filter in Corollary \ref{fipr}
and the analysis in Section \ref{sec:filter}, but has ultimately been chosen empirically. Whenever necessary, then the rank $r_\mu$ is decreased through a simple truncation $\sigma^{(\mu)}_{r_{\mu}} \gets 0$,
while it is increased using a minor singular value $0 < \sigma^{(\mu)}_{r_{\mu}+1} \ll \sigma_{\mathrm{min}}$.
In the latter case, the required, corresponding singular vectors can for example be chosen randomly.
\\\\
As in the matrix case, the lower limit $\sigma_{\min}$ is a fraction $f_{\sigma_{\min}} \ll 1$ of the residual on the sampling set (cf. Algorithm \ref{CALS_alg}).
The parameter $\omega > 0$ is reduced by a factor $f_\omega$ in each iteration, slowly reducing the magnitude of regularization. 
The factor $1 < f_\omega \in (f^{(\min)}_\omega,f^{(\max)}_\omega)$, in turn, is increased or decreased after each iteration 
through a simple heuristic, in such a way that
\begin{align}
 \max_{X \in P,P_2} \frac{\|A^{({\tt iter})} - M\|_X}{\|A^{({\tt iter}-1)} - M\|_X} \overset{!}{=} 1 + \varepsilon_{\mathrm{progr}}, \quad \varepsilon_{\mathrm{progr}} > 0, \label{rop}
\end{align}
or rather, that it stays close to this fixed value. The tensor
$A^{({\tt iter})} = \tau_{r^{({\tt iter})}}(G^{({\tt iter})})$ is the iterate at iteration number ${\tt iter}$. This adaption ensures that $f_\omega$ is
not too large as to impair the approximation quality, but neither so small that 
the required runtime becomes unreasonable.
\\\\
The algorithm will terminate if one of the following stopping criteria is fulfilled:
\begin{itemize}
 \item stagnation: $\omega \ll \sigma_{\min}$ and $f_\omega = f^{(\max)}_\omega$
 \item convergence: $\omega \rightarrow 0$ or $\|A^{({\tt iter})} - M\|_P / \|M\|_P \rightarrow 0$
 \item early stop: $\|A^{({\tt iter})} - M\|_{P_2} \gg \min_{i < {\tt iter}} \|A^{(i)} - M\|_{P_2}$
\end{itemize}
We have neglected minor implementation details and practical tweaks in this subsection to focus on the essence of the above criteria,
such that we refer to the Matlab code for remaining parts.
%
%
\oldwojg{\\\\
The stability of SALSA is used to establish an in principle simple rank adaption.
For a more detailed analysis and motivation, we refer to Section \ref{sec:filter}.
We capture the magnitude of regularization caused by the individual singular vectors $\sigma^{(\mu)}$:
\begin{definition}[Minimal filter values]\label{minfilvalues}
Define the entries of $\lowfil^{(\mu)} \in (0,1)^{r_\mu}$ via
\begin{align} \label{theta} \lowfil^{(\mu)}_i := \max(\filter^{(\mu)}_{1,i},\filter^{(\mu+1)}_{i,1}), \end{align}
where $\filter^{(0)} = \filter^{(d)} := 0$ and $\filter^{(\mu)}$ (for each $\mu$) is defined in Corollary \ref{fipr}.
\end{definition}
This magnitude is then used to define certain thresholds for all singular values.
\begin{definition}[Virtual ranks and virtual singular values]\label{virtstabsv}
Let $0 < \lowfil_{\mathrm{virt}} < \lowfil_{\mathrm{stab}} < 1$ be fixed. A singular value $\sigma^{(\mu)}_i$ is called
\textit{virtual}, if $\lowfil^{(\mu)}_i < \lowfil_{\mathrm{virt}}$ and denoted \textit{stabilized} (with respect to $\lowfil_{\mathrm{stab}}$) 
if $\lowfil^{(\mu)}_i > \lowfil_{\mathrm{stab}}$. The \textit{virtual rank} of $A = \tau_r(G)$ is given by its
exact rank $r = r(A)$, while the \textit{stabilized rank} $r^{(S)}$ only includes the stabilized singular values.
\end{definition}
The trick is to overestimate all ranks by $1$ and to gradually decrease $\omega$ (as well as the singular value limit). 
During several iterations, each last singular value $\sigma^{(\mu)}_{r_{\mu}}$ just equals $\sigma_{\mathrm{min}}^{(\mu)}$ (cf. Algorithm \ref{CALS_alg}). 
It does thereby only marginally influence the optimization,
which is why we use the term \textit{virtual}. 
However, at a certain point, the according singular values exceed
the minimum and then stabilize. Each time this happens and certain criteria hold, the technical rank is increased by $1$ 
(by adding a virtual singular value using random terms).
Vice versa, a rank is cut if the stabilized rank is by $2$ lower than the virtual rank. \\
In principle, one could use maximal ranks from the start, but we have encountered that the
current substitute for the replenishment term is not entirely suitable for that. As such an approach would involve higher computational costs anyway,
we do not investigate further into a fully implicit rank adaption. The rest
of this subsection will deal with remaining details.\\
The minimal values $\sigma^{(\mu)}$ substitute the replenishment term in \eqref{conv} in order to prevent virtual singular
values from quickly converging to zero, which would cause them not to be picked up in subsequent steps.
\begin{definition}[Singular value limit]\label{sivali}
The lower limit to the singular values is defined as fixpoint of
\begin{align}
 \sigma_{\mathrm{min}}^{(\mu)} \mapsto \frac{1}{\sum_{\mu=1}^d n_{\mu}} \ (1 - \lowfil_{\mathrm{min}}^{(\mu)}(\sigma_{\mathrm{min}}^{(\mu)})) \ \mathrm{Res}^{\mathrm{est}} \label{fakerepl}
\end{align}
where $\lowfil_{\mathrm{min}}^{(\mu)}(\sigma_{\mathrm{min}}^{(\mu)})$ is defined the same way as $\lowfil^{(\mu)}$ (see \eqref{theta}, \eqref{Gamma}), 
but assuming that all last singular values
equal the minimal $\sigma_{\mathrm{min}}^{(\mu)}$. The value $\mathrm{Res}^{\mathrm{est}}  > 0$ is
a pessimistic estimator for the full residual, 
\[ \mathrm{Res}^{\mathrm{est}}  := (\sqrt{|\Ind|/|P_2|} \mathrm{Res}_{P_2})^{3/2} \ (\sqrt{|\Ind|/|P|} \mathrm{Res}_{P})^{-1/2}. \]
\end{definition}
In practice, it is sufficient to perform a damped fixpoint iteration parallel to the decreases of $\widetilde{\omega}$
to obtain $\sigma_{\mathrm{min}}^{(\mu)}$.
As usual for machine learning, we sacrifice a fraction of the sampling set (the training set) $P$ to
obtain a control set $P_2$. The performance on this set serves as main criterion for the final output
of the algorithm.
\begin{definition}[Control set]\label{cose}
For a given index set $P$, we define $P_2 \subset P$ as control set. This set may be chosen
randomly or specifically distributed as well. The actual set used for the optimization is replaced by $P \leftarrow P \setminus P_2$ (keeping
the same symbol). 
\end{definition}
It is not easy to give a general criterion when to terminate the algorithm. 
Often, an estimate for an upper limit to all ranks proves efficient. 
Since as final result, the one representation with minimal control residual is chosen (cf. Remark \ref{te}), 
the approximation quality can be impaired only through a too early stop. Only in case that the algorithm
does not allow any more rank increases (which it only does for substantial reasons)
and the progress is low as well as $\omega$ sufficiently small, it will terminate.
%
\begin{definition}[Unblocked ranks]\label{unra}
We define the set of unblocked ranks $\mathcal{U} \subset \{2,\ldots,d\}$.
If $\sum_{\mu = 1}^d r_{\mu-1} r_{\mu} n_{\mu} - \sum_{i = 1}^d r_{\mu}^2 > |P|/1.2$ (degrees of freedom too high), then we set $\mathcal{U} = \emptyset$.
Otherwise,
\begin{align*} \mu \in \mathcal{U} \
\Leftrightarrow \ r_{\mu}+1 \leq \min(n_{\mu} r^{(S)}_{\mu-1}, n_{\mu+1} r^{(S)}_{\mu+1}, r_{\mathrm{lim}}) \ \wedge \ \lowfil_{\mathrm{min}}^{(\mu)}(\sigma_{\mathrm{min}}^{(\mu)}) < \lowfil_{\mathrm{stab}},
\end{align*}
%
where $r_{\mathrm{lim}} \in \mathbb{N}$ is a given, upper limit to any rank. The value $\lowfil_{\mathrm{min}}^{(\mu)}$ is defined by \eqref{fakerepl}.
\end{definition}
The reason for this definition are simple: either a rank is limited
by the neighboring ranks (independent of the tensor completion problem), it is maximal as defined by the user or it is by definition impossible to add a virtual rank, 
since due to the value $\sigma_{\mathrm{min}}^{(\mu)}$, an additional singular values would already be stabilized (cf. Remark \ref{chra}).
\begin{remark}[Decline of $\omega$]\label{deofom}
Let $G^{(\tt iter)}$ be the representation after iteration number ${\tt iter} = 1,2,\ldots$. Define
 \[\gamma^i_X := \frac{\mathrm{Res}_X(G^{(i)})}{\mathrm{Res}_X(G^{(i-1)})},\qquad  i={\tt iter}-4 \ldots {\tt iter}, \ X \in \{P,P_2\} \]
the arithmetic mean of the last $5$ residual reduction factors for the sampling and control residual. Let
$\widetilde{\lowfil}_{\mathrm{stab}} > \lowfil_{\mathrm{stab}}$ be fixed and close to, yet less than $1$.
We say $\omega$ is {\normalfont{minimal}}, if either
 \newcommand{\cfbox}[2]{%
    \colorlet{currentcolor}{.}%
    {\color{#1}%
    \fbox{\color{currentcolor}#2}}%
}
\def\ccfbox#1{
\begin{itemize}
 \item[]  \noindent\cfbox{gray}{\parbox{\linewidth - 2\fboxsep}{ #1 }}
\end{itemize}
}
\def\ccnofbox#1{
\begin{itemize}
 \item[]  \noindent\parbox{\linewidth - 2\fboxsep}{ #1 }
\end{itemize}
}
\ccfbox{there exists a stabilized (with respect to $\widetilde{\lowfil}_{\mathrm{stab}}$) rank equal to $r_{\mathrm{lim}}$}
or
\ccfbox{
\ccnofbox{$\mathcal{U} = \{ \}$ (Definition \ref{unra})}
and
\ccnofbox{all singular values are stabilized (with respect to $\widetilde{\lowfil}_{\mathrm{stab}}$)
}
}
\vspace{0.1cm}
%
%
\noindent
The parameter is regulated as follows:
Initialize $\widetilde{\omega} = \omega_0$. After each iteration ${\tt iter}$, if 
\ccfbox{ $\omega$ is not minimal }
and
\ccfbox{ 
\ccfbox{
  \ccnofbox{the singular spectrum does not currently change too much}
  and
  \ccnofbox{$\gamma^i_P < \gamma^{\ast}$ or $\gamma^i_{P_2} < \gamma^{\ast}$}
  }
  or
  \ccnofbox{$\mathrm{Res}_P(G^{({\tt iter})}) > \mathrm{Res}_P(G^{({\tt iter}-1)})$}
  }
  \vspace{0.1cm}
  %
%
%
%
%
\noindent
then $\widetilde{\omega}$ is decreased by a constant factor of $f_{\omega}$.
Set then $\omega = \widetilde{\omega} \|\tau_r(G^{({\tt iter})})\|_{\Ind}$. 
Furthermore, the decrease of $\widetilde{\omega}$ is accelerated if $\max(\lowfil^{(\mu)}) \ll \lowfil_{\mathrm{virt}}^{(\mu)}$.
\end{remark}
\begin{remark}[Changing ranks]\label{chra}
The $\mu$-th rank is increased if all of the following conditions hold:
\begin{itemize}
 \item $\mu \in \mathcal{U}$ (Definition \ref{unra})
 \item $\widetilde{\omega}$ has been decreased in the previous iteration
 \item $\sigma^{(\mu)}_{r_{\mu}}$ is stabilized
 \item there is no other stabilized rank $r_\mu^{(S)}$ equal to $r_{\mathrm{lim}}$
\end{itemize}
The representation is then expanded randomly, such that for the new singular value holds $\sigma^{(\mu)}_{r_{\mu}+1} = \sigma_{\mathrm{min}}^{(\mu)}$
and all other singular values remain (near) equal. \\ If, in contrast, at any time $\sigma^{(\mu)}_{r_{\mu-1}}$ is virtual 
(and hence $\sigma^{(\mu)}_{r_{\mu}}$ as well), the rank is decreased by $1$ and the tensor truncated.
\end{remark}
By this kind of rank adaption, only virtual singular values are ever introduced or removed.
This is to be understood as the main idea behind SALSA. The exact rank is not relevant anymore within the optimization, only
the magnitude of $\omega$ compared to the singular values matters.
\begin{remark}[Termination]\label{te}
Let $i^{\ast} = \mathrm{argmin}_i \mathrm{Res}_{P_2}(G^{(i)})$ and $f_{P_2} > 1$ be fixed.
If one of the following criteria holds, then the algorithm terminates.
\begin{itemize}
 \item $\omega$ is minimal (Remark \ref{deofom}) (convergence)
 \item ${\tt iter} > 10$ and $\mathrm{Res}_{P_2} > f_{P_2} \cdot \mathrm{Res}_{P_2}(G^{i^{\ast}})$ (Definition \ref{cose}) (divergence)
\end{itemize}
As final result, $G^{i^{\ast}}$ is chosen (it may be cut to its stabilized rank).
\end{remark}
A typical approximation terminates as follows. At some point, the residual on the control set (Definition \ref{cose}) declines slower than the
one on the sampling points. If we do not already encounter divergence when further on decreasing $\omega$ (Remark \ref{te}), the lower limit to all
singular values (Definition \ref{sivali}) will increase up to the point where all ranks are blocked (Definition \ref{unra}). This prohibits
further rank increases (Remark \ref{chra}) and ultimately $\omega$ becomes \textit{minimal} (Remark \ref{deofom}), that is it will no longer be decreased.
At that point the algorithm terminates (Remark \ref{te}) and picks a former representation with smallest residual on the control set.}
%
%
%
\subsection{The SALSA Algorithm}
SALSA (Stable ALS Approximation) for tensors is summarized in Algorithm \ref{ra}. \BCn{m2}{m}%
For the (recommended) choices of tuning parameters also used in the numerical tests, see Subsection \ref{sec:tupa}.
The Matlab implementation as well as a video showing the rank adaption by means of plotting the singular values during a runtime
can be found on the personal webpage of the author Sebastian Kr\"amer, along with all sources that were used to create the presented results.%
\footnote{by the time the paper is written, the address is {\tt www.igpm.rwth-aachen.de/team/kraemer}}
The notion of stability we address in this paper does however not mean a low sensibility to roundoff errors accumulated over multiple iterations. 
We encountered that even different processor architectures or rearrangement of brackets with regard to associativity in products can change the intermediate approximations.
\oldwo{in case the algorithm}\newwo{As in almost all cases the algorithm} does not find the global minimum,
this also holds for the final output, both in case of ALS and SALSA. For once, this is not an actual drawback, since even exact arithmetic would
not consistently cause the algorithm to find better local minima, but it should be kept in mind when reconstructing results.
The order of computational complexity does not exceed $\mathcal{O}(dr^2\oldwojg{\# P}\newwo{|P|})$ using coarse \oldwo{CG}\newwo{cg}, where $r = \max_{\mu} r_{\mu}$. Note that the computational
complexity per sweep can actually be lower, since not all ranks are kept equal, but some are lower than others. \\

 \begin{algorithm}
  \caption{SALSA Algorithm \label{CALS_alg}}
    \newwo{
  \begin{algorithmic}[1]  \label{ra}
  \REQUIRE $P \subset \Index$, $M|_P$
  \STATE initialize $G$ s.t. $\tau_r(G) \equiv \|M|_P\|_1/|P|$ for $r \equiv 1$ and $\omega = 1/2 \|\tau_r(G)\|_F$
  \STATE split off a small control set $P_2 \subset P$ (Definition \ref{cose})
  \FOR{${\tt iter}=1,2,\ldots$}
    \STATE proceed SALSA sweep$^{\ast}$ (Algorithm \ref{fiadfsw})  
    \STATE$^{\ast}$: and renew lower limit $\sigma_{\mathrm{min}} := f_{\min} \cdot \frac{|\Index|}{|P|} \|\tau_r(G) - M\|_P$
    \STATE$^{\ast}$: adapt $f_\omega$ according to progress (cf. \eqref{rop})
    \STATE$^{\ast}$: adapt and decrease $\omega$ by factor of $f_\omega$
    \STATE adapt rank according to \eqref{rankincr} (start this when the first few iteration have passed)
    \IF{a stopping criterion applies (Section \ref{sec:seimannounraad})}
      \STATE terminate algorithm
      \STATE \textbf{return} iterate for which $\|\tau_r(G) - M\|_{P_2}$ was lowest
    \ENDIF
  \ENDFOR
  \end{algorithmic}}
  \oldwojg{
  \begin{algorithmic}[1]  \label{ra}
  \REQUIRE $P \subset \Index$, $M|_P$ (and parameters)
  \STATE initialize $G$ s.t. $\tau_r(G) \equiv \mathrm{const}$, $|P| \|\tau_r(G)\|_F^2 = |\Ind|\|M|_P\|_P^2$ for $r \equiv 1$ and $\widetilde{\omega} = 1/2$
  \STATE split off a small validation set $P_2 \subset P$ (Definition \ref{cose})
  \STATE proceed one or a few ordinary ALS sweeps (Algorithm \ref{fiadfsw} for $\omega \equiv 0$)
  \FOR{${\tt iter}=1,2,\ldots$}
    \STATE ONCE: after a few iterations, introduce virtual ranks ($\Rightarrow r \equiv 2$)
    \STATE proceed SALSA sweep$^{\ast}$ (Algorithm \ref{fiadfsw})    
    \STATE$^{\ast}$: decrease $\widetilde{\omega}$ if progress low (Remark \ref{deofom} applies)
    \IF{$^{\ast}$: a singular value becomes stabilized/virtual (Remark \ref{chra} applies)} 
      \STATE increase/decrease the virtual rank
    \ENDIF
    \IF{final breaking criteria apply (Remark \ref{te})}
      \STATE terminate algorithm
    \ENDIF
  \ENDFOR
  \end{algorithmic}}  
\end{algorithm} 
\section{Numerical Experiments}\label{sec:ne}
We consider the following \oldwo{two}\newwo{three} algorithms:
\begin{itemize}
 \item \oldwo{(Greedy adaptive)}\textbf{ALS} (\newwo{modified }Algorithm \ref{fiadfsw} for $\omega \equiv 0$)
 \item \textbf{SALSA} (Algorithm \ref{CALS_alg}\newwo{, the algorithm proposed in this work})
 \item \newwo{\textbf{RTTC} (Riemannian cg for tensor train completion \cite{St16_Rie})}\BCn{j1_}{j_}%
\end{itemize}
We explain how ranks are adapted for ALS in Section \ref{sec:raadfostals}, shortly present the idea behind RTTC in Section \ref{sec:rttc},
give details for data acquisition and measurements in Section \ref{sec:daaqanme} as well as
tuning parameters in Section \ref{sec:tupa}. We analyze the results in the latter Section \ref{sec:anofre}. 
For each test, we give a (too large) upper bound $r_{\mathrm{lim}}$ 
for the maximal rank of the iterates\newwo{, in order to rule out excessive computation times (although this bound is seldomly reached).\BCn{w1}{w}
We would like to emphasize that, in contrast to rank adaption itself, the one dimensional problem of choosing such a bound
is easily controlled for example based on the validation set}.%
\oldwo{We like to emphasize that, in contrast to rank adaption itself, such a bound 
can subsequently be increased if it might yield improvements - 
since this does only pose a one dimensional problem.
Such a limit is not obligatory, but in specific cases the necessarily coarse criteria in Remark \ref{unra}
only hold for very large rank, such that the algorithms would use up a lot of time without changing the results (cf. Remark \ref{te}).}
For simplicity, we use a common mode size $n = n_1 = \ldots = n_d$. 

\subsection{Rank Adaption for Standard ALS}\label{sec:raadfostals}
Since ALS itself is not rank adaptive, the (so far) most promising approach, that is greedy
rank adaption, is chosen. When the progress stagnates, the algorithm searches for the highest
(new) singular value $\sigma^{(\mu)}_+$ which any of the rank increases may yield. These values are estimated as follows.
Let $\mu$ be fixed and $G$ be a representation for which $G^{<\mu-1}$ is column-orthogonal and $G^{>\mu}$ is row-orthogonal.
Further, let
\begin{align*} 
 T & := (G^{<\mu-1})^T \ \left((M - \tau_r(G))|_P \right)_{(\mu-1,\mu)} \ (G^{>\mu})^T, \\
 \alpha_{i_{\mu-1},i_\mu} & = \argmin_{\widetilde{\alpha}_{i_{\mu-1},i_\mu}} \|G^{<\mu-1} \ (G_{\mu-1}(i_{\mu-1}) 
  \cdot G_{\mu}(i_{\mu}) \\
  & \quad + \widetilde{\alpha}_{i_{\mu-1},i_\mu} T(i_{\mu-1},i_\mu)) \ G^{>\mu} - M_{(\mu-1,\mu)} \|_{P_{(\mu-1,\mu)}(i_{\mu-1},i_\mu)}.
\end{align*}
We define the core $H(\cdot,\cdot)$, $H(i_{\mu-1},i_\mu) 
= \alpha_{i_{\mu-1},i_\mu} T(i_{\mu-1},i_\mu) \in \R^{r_{\mu-2} \times r_{\mu}}$
and stack its entries to form the matrix $\mathfrak{H} \in \R^{r_{\mu-2} n_{\mu-1} \times r_{\mu} n_{\mu}}$.
\oldwo{Then}\newwo{This yields the candidate} $\sigma^{(\mu)}_+ := \|\mathfrak{H}\|_2$, the largest singular value of $\mathfrak{H}$.\BCn{x1}{x}
This approach is very similar \oldwojg{to two-fold DMRG microsteps as defined in \cite{HoRowSc12_The},}\newwo{to the two-fold microsteps as defined in \cite{HoRowSc12_The} and the rank adaption in AMEn \cite{DoSa14_Alt}, which are both based on DMRG.} It however prevents overfitting,
since it is equivalent to performing only one, preconditioned CG step (similarly to a Landweber iteration). Furthermore, our experiments suggest that it is more reliable.
The corresponding rank $\mu = \argmin_{\widetilde{\mu}} \sigma^{(\widetilde{\mu})}_+$ is increased by $1$, 
using a rank $1$ approximation of $\mathfrak{H}$.
\oldwo{Basically the same termination criteria as for SALSA are used, although some criteria 
that are based on $\omega$ need to be compensated for}\newwo{Since ALS works differently than SALSA,\BCn{d1_}{d_}
only some stopping criteria can be overtaken, while additional ones are introduced in order to prevent premature termination but also
to avoid unnecessary runtime}. No rank decreases are proceeded since
this involves tremendous difficulties, of which the most important one is the sheer incapability to decide when and
which rank actually to decrease.

\subsection{\newwo{The RTTC Algorithm}}\label{sec:rttc}\BCn{j3_}{j_} 
The article \cite{St16_Rie}, in which RTTC is derived and explained in detail, focuses exclusively on
tensor completion using the tensor train format as well (but it can likewise be assumed that it is generalizable to other problems).
Instead of alternating optimization, RTTC provides a nonlinear conjugate gradient scheme based on Riemannian optimization, which has comparable computational complexity
per sweep. Naturally, the problem of rank adaption also poses a challenge in that setting. Therefore, a heuristic rank adaption is introduced (Algorithm 3 in \cite{St16_Rie})
which successively tests if a single rank increase yields a tolerable change of the residual on the validation set (based on a parameter $\rho \geq 0$).
If so, it continues normally with the next test; otherwise, the algorithm priorly resets to the iterate with previous rank. \\\\
We observed however that for example the choice $\rho \neq 0$ worked better given the assignments in Section \ref{sec:generic_sec}, 
but choices other than $\rho = 0$ caused the algorithm to not recover a single instance in case of the rank adaption test tensor in Section \ref{sec:recratt}.
We therefore used three different choices $\rho \in \{0,0.2,1\}$ ($1$ is the default in \cite{St16_Rie}) and granted RTTC, as opposed to SALSA,
the advantage to choose the best result (based on the test set) for each of the following problem classes. \\\\
Some minor modification to RTTC were necessary in order to provide fair tests. The allowed number of iterations per rank increase test
was increased, since $10$ turned out to be too few. In exchange, the relative improvement parameter was raised to $10^{-3}$, as lower
values did not yield improvements. Since RTTC does not have an actual stopping criterion,
we stopped whenever the normalized validation residual increased and at the same time was $1000$ times higher than the normalized sampling residual, or whenever the current validation residual
was much higher than any previously obtained one (as described in Section \ref{sec:seimannounraad}). When terminating, each time the iterate with lowest validation residual was chosen 
(as described in Algorithm \ref{CALS_alg}). These adaptations were done carefully in order to obtain only improvements in the quality of approximation.
We did further not compare the required number of iterations of RTTC, and the algorithm was granted as much time as necessary.
\subsection{Data Acquisition and Measurements}\label{sec:daaqanme}
\textit{Sampling:} In order to obtain a sufficient sampling for each slice of $M$, we generate the set
$P$ in a quasi-random way as follows: For each direction $\mu = 1,\ldots,d$ and each index 
$i_{\mu}\in\Index_\mu$ we pick $c_{\mathrm{sf}} \cdot r_P^2$ indices $i_1,\ldots,i_{\mu-1},i_{\mu+1},\ldots,i_d$ 
at random (uniformly). This gives in total $|P| \oldwojg{\lesssim}\newwo{=} c_{\mathrm{sf}} \cdot d n  r_P^2$ samples (excluding duplicate samples). 
The rank $r_P$ is artificial, such that $c_{\mathrm{sf}}$ can be interpreted as sampling factor
since the number of degrees of freedom of a TT-tensor of common rank $r$ is slightly less than $d n r^2$.\\
\textit{Testing:}
As a \oldwo{verification}\newwo{test} set $C$, we use a set of the same cardinality as $P$ that is generated in the same way.
\newwo{Of course, neither this set nor the values $M|_C$ are known by the algorithm.
The residuals are then measured with respect to the chosen iterate which had the lowest validation residual (cf. Algorithm \ref{CALS_alg}).
}
\\ 
\textit{Order of optimization:} Instead of the sweep we gave before ($\mu = 1,\ldots,d$) for simplicity, we
alternate between two sweeps ($\mu = 1,\ldots,h, \quad \mu = d,\ldots,h, \quad h = \lfloor d/2 \rfloor$) to enhance symmetry.\\
\textit{Averaging:} With $\arithmean{\cdot}$ we denote the arithmetic mean and by $\geomean{\cdot}$ the geometric mean
which we use for logarithmic scales.
\subsection{Implementation Details and Tuning Parameters}\label{sec:tupa} 
All tests for ALS and SALSA were done using a (pure) Matlab implementation. This includes the toolbox \textit{multiprod} \cite{Pao2010_Mul},
which allows a reasonably swift evaluation of products between arrays of matrices, $H(i) J(i)$, $i = 1,\ldots,k$, and is much faster than a plain loop.
In contrast, some subfunctions of RTTC are based on .mex routines.\BCn{k2_}{k_}\\
Instead of solving full problems in each microstep, both ALS and SALSA use coarse cg (cf. Remark \ref{applofcg}), for which
the tolerance was empirically chosen low enough such that it did not influence the quality of approximation \BCn{e2_}{e_}. Note that the cg steps
of RTTC are not comparable, since they are performed on low rank manifolds and used to update all cores at once. \\
We only list time performances in the appendix, which should be interpreted carefully, while the iteration numbers may
provide a clearer picture due to similar computational complexities. All parameters
have been chosen equally for all experiments (except $\rho$ for RTTC) with respect to best results, not speed, and could be relaxed for easier problems
(or in practice for first trials) to reduce timing considerably. 
Straightening the tolerances for ALS or RTTC, hence allowing more iterations, did however not lead to 
notable improvements. \\
\oldwo{We use the same tuning parameters for all tests, given by:
$\gamma^{\ast} = 10^{-3}$, $f_{\omega} = 1.1$, $\lowfil_{\mathrm{virt}} := 0.33$, $\lowfil_{\mathrm{stab}} := 0.99$, $\widetilde{\lowfil}_{\mathrm{stab}} := 0.999$,
$f_{P_2} := 2.5$, $|P_2|/|P| = 1/20$. The specific choices are heuristic (based on experience), but likewise recommendable
for other problems.}
\newwo{The parameter choices (cf. Section \ref{sec:seimannounraad}) for SALSA are given by
$\varepsilon_{\mathrm{progr}} = 5\cdot10^{-3}$, $f_{\mathrm{minor}} = 0.5$, $k_{\mathrm{minor}} = 2$, $f^{(\min)}_{\omega} = 1+5\cdot10^{-4}$, $f^{(\max)}_{\omega} = 1.1$, $f_{\sigma_{\min}} = 0.1$.
The size of the validation set is $|P_2|/(|P|+|P_2|) = 1/20$. These have in parts been chosen empirically and are recommendable
for other problems.}
We observed that any reasonably close values work as well,
the more so for larger sampling sets. The performance is in that sense not based on how close
the parameters are to some unknown optimal choices. We also refer to the implementation for all details.

\subsection{Approximation of a Tensor with Near Uniform Singular Spectrum}\label{sec:domino_sec}

At first, we consider the completion of the following tensor:
\[ D(i_1,\ldots,i_d) := \left( 1 + \sum_{\mu = 1}^{d-1} \frac{i_{\mu}}{i_{\mu+1}} \right)^{-1}, \quad i_\mu = 1,\ldots,n, \ \mu = 1,\ldots,d \]
This tensor is not low rank, but has well ordered modes and uniformly exponentially decaying singular values.  
It can therefore very well be approximated with uniform ranks (for a black box, rank adaptive algorithm however, this is not trivial to recognize) 
and the low variance of results suggests that mostly \oldwo{the}\newwo{a near} best approximation is found. Hence, standard ALS can barely be outperformed. 
The results are plotted in Figure \ref{domino_plot} (see Appendix C for Table \ref{domino_table}).

\begin{figure}
  \begin{center}
      \ifuseprecompiled
      \includegraphics[width=0.98\textwidth]{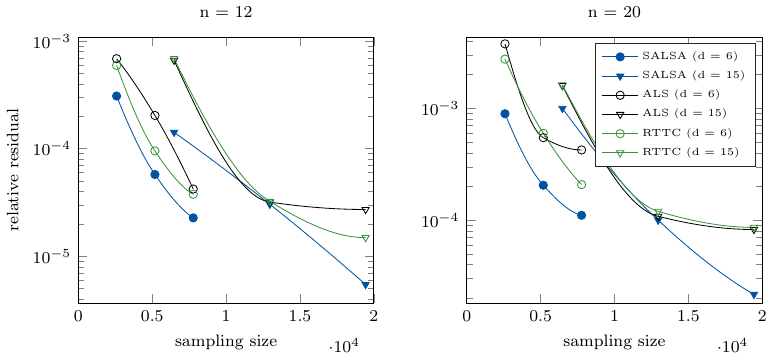}
      \fi
  \end{center}
  \caption{\label{domino_plot} ($d=6\oldwojg{,9},15$, $r_P=6$, $r_{\mathrm{lim}} = \oldwojg{10\ }\newwo{14}$, $n=12,20$, $c_{\mathrm{sf}} = 2,4,6$) Plotted are, \newwo{for
  the tensor $D$,} for varying dimension and mode size, the averaged relative residuals $\geomean{\|A - M\|_C/\|M\|_C}$
    \oldwo{and accordant standard deviations} as functions of the sampling size $|P|$
    as result of \newwo{each} $20$ trials, for ALS (black), SALSA (blue, filled symbols) \newwo{and RTTC using $\rho = 1$ (green). The markers are exact;
    the intermediate lines are \textit{shape-preserving piecewise cubic Hermite interpolations of such}.}}
\end{figure}
\subsection{Approximation of Three Generic Tensors with non Uniform Singular Spectrum}\label{sec:generic_sec}
We want to demonstrate how different results can
be through proper rank adaption, considering the following three tensors, generated by generic functions:
\begin{align*} f^{(1)}(i_1,\ldots,i_8) & := \frac{{i_1}}{4} \cos({i_3}-{i_8}) + \frac{{i_2}^2}{{i_1}+{i_6}+{i_7}} + {i_5}^3 \sin({i_6}+{i_3}) \\
 f^{(2)}(i_1,\ldots,i_7) & := \left(\frac{{i_4}}{{i_2}+{i_6}} + {i_1}+{i_3}-{i_5}-{i_7} \right)^2, \quad i_\mu = 1,\ldots,n, \ \mu = 1,\ldots,d \\
 f^{(3)}(i_1,\ldots,i_{11}) & := \sqrt{i_2 +i_3+ \frac{1}{10}(i_4+i_5+i_7+i_8+i_9) + \frac{1}{20}(i_1 - i_6 - i_{10} +i_{11}  )^2};
\end{align*}
In contrast to the tensor in Section \ref{sec:domino_sec}, the modes are not (and hardly can be) ordered
in accordance with the TT format. A different ordering may of course yield other results,
but we cannot assume to find a better ordering if the approximation fails in
the general case. The results are plotted in Figure \ref{generic_plot} (see Appendix C for Table \ref{generic_table}).
\begin{figure}
  \begin{center}
      \ifuseprecompiled
      
     \includegraphics[width=0.98\textwidth]{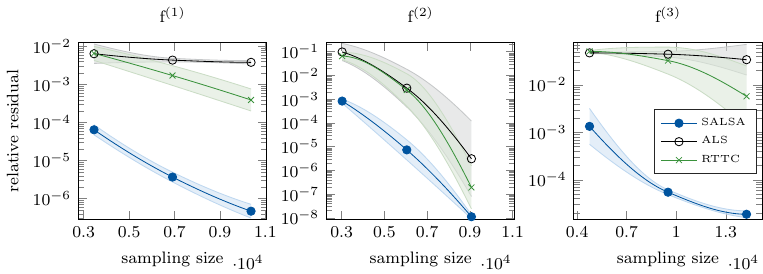}
     \fi
  \end{center}
  \caption{\label{generic_plot} ($d_1 = 8, d_2 = 7, d_3 = 11$, $r_P=6$, $r_{\mathrm{lim}} = 10$, $n=8$, $c_{\mathrm{sf}} = 2,4,6$) Plotted are, 
    for the tensors $f^{(1)}$ (left), $f^{(2)}$ (middle) and $f^{(3)}$ (right), the averaged relative residuals $\geomean{\|A-M\|_C/\|M\|_C}$
    and \newwo{shadings proportional to the}\oldwo{accordant} standard deviations as functions of the sampling size $|P|$
    as result of \newwo{each} $20$ trials, for ALS (black), SALSA (blue, filled symbols) \newwo{and RTTC using $\rho = 0.2$ (green). The markers are exact;
    the intermediate lines are \textit{shape-preserving piecewise cubic Hermite interpolations of such}.}}
\end{figure}
\subsection{Recovery of Random Tensors with Exact Low Rank}\label{sec:reofratewilora}
We next consider the recovery of quasi-random tensors with exact low ranks.
Although this in practice will never occur, it is a very neutral test\footnote{Note that in some papers,
uniform distributions on $[0,1]$ are used such that all entries of the target tensor are
positive, causing each first singular value to be huge compared to all following ones.
This leads to a tremendous simplification of the completion problem. There seems to be no indication yet that
the sampling required for the completion of a random tensor is in general close to what is stated for the matrix case \cite{CaTa10_The}.}. 
The ranks are generated randomly, but it is ensured that $\arithmean{r} \geq 2/3 k$ and $\max(r) \leq k$ for
some bound $k \in \mathbb{N}$. \\
Each of these is generated via a TT representation $A=\tau_r(G)$ where we assign to each entry of each
block $G_1,\ldots,G_d$ a uniformly distributed random value in $[-0.5,0.5]$. Subsequently,
the singular values $\Sigma^{(1)},\ldots,\Sigma^{(d-1)}$ are forced to take 
uniformly distributed random values in $[0,1]$ (up to scaling). This is achieved
by successive replacements of the current values in $G$. \\
As results, we plot the number of successful recoveries ($\|A-M\|_C/\|M\|_C < 10^{-5}$)
for different mode sizes $n$ (each single tuple uniform), dimensions $d$ and maximal ranks $k$ of the
target tensor (Figures \ref{rt6}, \ref{rt8}).
\def\plottickfontsize{\tiny} 
\begin{figure}
  \begin{center}
\ifuseprecompiled
\includegraphics[width=0.98\textwidth]{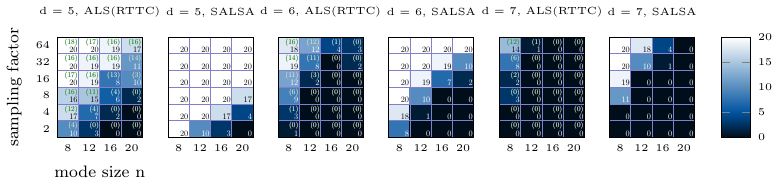}
\fi
  \end{center}
  \caption{\label{rt6} ($d=5,6,7$, $r_P = 6$, $r_{\mathrm{lim}} = 9$, $n=8,12,16,20$, $c_{\mathrm{sf}} = 2,4,8,16,32,64$) 
    Displayed as $20$ shades of 
    blue (black $(0)$ to white $(\mbox{all } 20)$) are the number of successful reconstructions 
    for random tensors with maximal rank $k = 6$ for ALS and SALSA. \newwo{The bracketed green numbers are the results for RTTC using $\rho = 0.2$ and are independent of the shading.
    We recommend use of the digital version for better readability.}}
\end{figure}

\begin{figure}
  \begin{center}
\ifuseprecompiled
\includegraphics[width=0.98\textwidth]{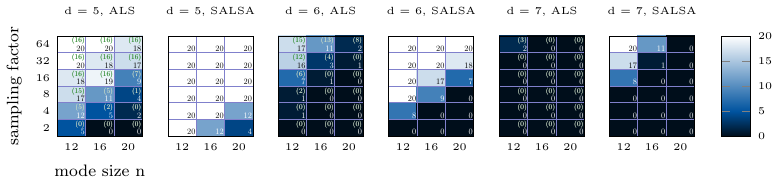}
\fi
  \end{center}
\caption{\label{rt8} ($d=5,6,7$, $r_P = 8$, $r_{\mathrm{lim}} = 11$, $n=12,16,20$, $c_{\mathrm{sf}} = 2,4,8,16,32,64$) 
    Displayed as $20$ shades of 
    blue (black $(0)$ to white $(\mbox{all } 20)$) are the number of successful reconstructions 
    for random tensors with maximal rank $k = 8$ for ALS and SALSA. \newwo{The bracketed green numbers are the results for RTTC using $\rho = 0.2$ and are independent of the shading.
    We recommend use of the digital version for better readability.}}
\end{figure}

\subsection{Recovery of the Rank Adaption Test Tensor}\label{sec:recratt}

Last but not least, we consider the recovery
of tensors as in Example \ref{raadtete}, for
which $Q_1, Q_4, Q_5$ and $Q_6$ are generated quasi-randomly for
each trial. For an explanation of the results in Figure \ref{spc}, we refer to Section \ref{sec:reofratewilora}.

\begin{figure}
  \begin{center}
\ifuseprecompiled
\includegraphics[width=0.98\textwidth]{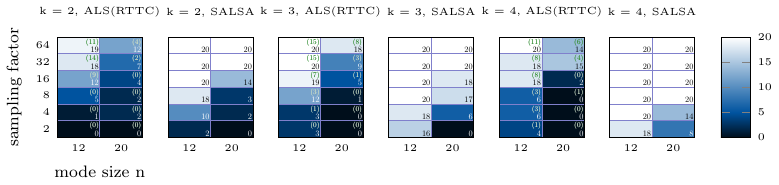}
\fi
  \end{center}
  \caption{
\label{spc} ($d=6$, $r_P = 2k$, $r_{\mathrm{lim}} = 2k+3$, $n=12,20$, $c_{\mathrm{sf}} = 2,4,8,16,32,64$) 
    Displayed as $20$ shades of 
    blue (black $(0)$ to white $(\mbox{all } 20)$) are the number of successful reconstructions 
    for the rank adaption test tensor with rank $(1,k,k,k,1,2k,1)$ for ALS and SALSA. \newwo{The bracketed green numbers are the results for RTTC using $\rho = 0$ and are independent of the shading.
    We recommend use of the digital version for better readability.}}
\end{figure}

\subsection{Analysis of Results}\label{sec:anofre}

SALSA is superior in nearly all observed cases\BCn{k1_}{k_}. For tensors which 
could as well be approximated with uniform ranks, the differences are marginal \newwo{as is to be expected}\oldwo{, but
SALSA yields better results (the timing however is slightly worse)}. The three generic
functions show that the residuals can be multiple orders of magnitude better,
and although the functions were chosen quite randomly, we do of course not want to
over-interpret these specific results. Finally, for the more neutral test of
random tensor recovery, the required sampling seems to be overall
$4$ to $8$ times lower. For the rank adaption test tensor, the performance
of SALSA becomes even better for larger rank $k$ (this is due to
the larger total sampling), while greedy ALS runs into the predicted trouble\newwo{, as does RTTC}.
\newwo{In general, RTTC\BCn{j2_}{j_} performs slightly better than greedy ALS in the approximation of
tensors with exponentially declining singular values, while the latter is slightly better in the random recovery tests.
This difference is unlikely because of their optimization techniques, which are known to achieve similar results,
but rather due to the different rank adaption heuristics. This insight is also based on closer inspections of single tests,
which suggest that the rank adaption of RTTC is inferior even to the one of ALS in the majority of cases.
}

\section{Conclusions}
In this article, we have demonstrated that the most successful completion
algorithms are very sensitive to rank changes and 
that existing rank adaption methods suffer from this. \\
In order to correct this, as proven for SALSA, we suggested a regularization motivated by averaged
microsteps in order to uncouple the optimization of a discrete, technical rank.
While there is likely room for improvements \newwo{and rigorous convergence bounds remain subject to future work},
we take the \newwo{noteworthy} numerical results as indication that \textit{stability (under
truncation)} is a worthwhile property.
The computational complexity of SALSA is further reduced to the minimal order through use
of a coarse cg method.
Although we focused on tensor
completion (with possibly small sampling sets), 
the derivations given in this paper allow for a generalization
to other semi-elliptic problems. Furthermore, it may be possible to adapt the
presented ideas to manifold based methods \newwo{such as RTTC}.


\bibliographystyle{spmpsci} 
\footnotesize
\bibliography{tensor_CALS}
\normalsize

\section{Appendix (Experimental Data)}\label{sec:appendixB}
Following are the precise values for Figures \ref{domino_plot} and \ref{generic_plot}, \newwo{for $R_C := \|A - M\|_C$ and $R_P := \|A - M\|_P$.}
%
%
\begin{table}[H]
\centering
\resizebox{\columnwidth}{!}{%
 \begin{tabular}{|c|c||c|c|c|c||c|c|c|c|}
 \hline
 \multicolumn{2}{|c|}{$n=12$} & \multicolumn{4}{c||}{ALS} & \multicolumn{4}{c|}{SALSA} \\
 \hline
 $d$ & $c_{\mathrm{sf}}$ & $\geomean{R_C/\|M_C\|}$ & $\geomean{R_P/\|M_P\|}$ & $\arithmean{\mbox{time}}$ & $\arithmean{\mbox{iter}}$ & $\geomean{R_C/\|M_C\|}$ & $\geomean{R_P/\|M_P\|}$ & $\arithmean{\mbox{time}}$ & $\arithmean{\mbox{iter}}$\\
\hline \multirow{3}{*}{6} & 2 & 6.9e-04(2.1) & 1.1e-04(2.9) & 85(57) & 467(151) &3.1e-04(1.1) & 6.8e-06(3.7) & 81(10) & 412(33)\\
& 4 & 2.0e-04(1.9) & 3.9e-05(3.8) & 131(86) & 527(154) &5.8e-05(1.5) & 1.4e-06(2.0) & 170(29) & 546(58)\\
& 6 & 4.2e-05(1.8) & 4.7e-06(2.8) & 215(76) & 624(108) &2.3e-05(1.4) & 1.5e-06(1.4) & 209(23) & 581(41)\\
\hline \multirow{3}{*}{9} & 2 & 7.6e-04(2.1) & 2.3e-04(4.4) & 145(88) & 606(198) &1.8e-04(1.1) & 6.9e-06(2.8) & 129(9) & 455(23)\\
& 4 & 8.0e-05(1.1) & 2.0e-05(1.1) & 276(57) & 804(81) &2.7e-05(1.2) & 7.4e-07(2.0) & 298(34) & 644(40)\\
& 6 & 6.6e-05(1.2) & 2.5e-05(1.2) & 355(94) & 834(105) &9.6e-06(1.2) & 5.5e-07(1.5) & 457(50) & 732(48)\\
\hline \multirow{3}{*}{15} & 2 & 6.6e-04(1.5) & 3.3e-04(2.1) & 384(113) & 951(159) &1.4e-04(1.4) & 1.1e-05(1.8) & 266(20) & 498(31)\\
& 4 & 3.2e-05(1.1) & 6.8e-06(1.1) & 961(62) & 1396(44) &3.0e-05(1.1) & 4.3e-06(1.6) & 540(35) & 642(32)\\
& 6 & 2.7e-05(1.1) & 8.3e-06(1.1) & 1278(85) & 1416(47) &5.5e-06(1.8) & 2.9e-07(1.6) & 1133(92) & 827(48)\\
 \hline
 \end{tabular}
 }
 \newline
\vspace*{0.2 cm}
\newline
 \resizebox{\columnwidth}{!}{%
 \begin{tabular}{|c|c||c|c|c|c||c|c|c|c|}
 \hline
 \multicolumn{2}{|c|}{$n=20$} & \multicolumn{4}{c||}{ALS} & \multicolumn{4}{c|}{SALSA} \\
 \hline
 $d$ & $c_{\mathrm{sf}}$ & $\geomean{R_C/\|M_C\|}$ & $\geomean{R_P/\|M_P\|}$ & $\arithmean{\mbox{time}}$ & $\arithmean{\mbox{iter}}$ & $\geomean{R_C/\|M_C\|}$ & $\geomean{R_P/\|M_P\|}$ & $\arithmean{\mbox{time}}$ & $\arithmean{\mbox{iter}}$\\
\hline \multirow{3}{*}{6} & 2 & 3.8e-03(1.1) & 1.7e-03(1.1) & 89(38) & 378(93) &8.9e-04(1.3) & 3.4e-05(2.9) & 119(15) & 375(34)\\
& 4 & 5.5e-04(1.2) & 1.7e-04(1.2) & 157(45) & 468(74) &2.1e-04(1.3) & 4.0e-06(1.5) & 256(30) & 506(36)\\
& 6 & 4.2e-04(1.2) & 1.6e-04(1.4) & 198(62) & 483(80) &1.1e-04(1.4) & 5.0e-06(1.3) & 329(27) & 536(32)\\
\hline \multirow{3}{*}{9} & 2 & 2.6e-03(1.0) & 1.4e-03(1.1) & 204(101) & 536(141) &4.6e-04(1.1) & 2.4e-05(2.2) & 224(15) & 452(26)\\
& 4 & 2.9e-04(1.1) & 1.0e-04(1.1) & 426(71) & 758(73) &1.3e-04(1.3) & 2.7e-06(1.5) & 452(57) & 583(38)\\
& 6 & 2.0e-04(1.1) & 7.2e-05(1.4) & 618(151) & 818(103) &4.2e-05(1.1) & 2.0e-06(1.3) & 748(53) & 692(23)\\
\hline \multirow{3}{*}{15} & 2 & 1.6e-03(1.0) & 8.6e-04(1.0) & 698(208) & 960(166) &1.0e-03(1.9) & 1.8e-04(6.0) & 361(64) & 417(46)\\
& 4 & 1.1e-04(1.1) & 2.3e-05(1.2) & 1766(123) & 1402(53) &1.0e-04(1.1) & 1.5e-05(1.5) & 899(50) & 608(27)\\
& 6 & 8.2e-05(1.0) & 2.8e-05(1.1) & 2299(139) & 1392(45) &2.2e-05(1.2) & 9.2e-07(1.4) & 1873(188) & 782(44)\\
\hline
\end{tabular}
}
\vspace{0.1cm}
\caption{\label{domino_table} Results for Subsection \ref{sec:domino_sec} (with arithmetic and geometric variances in brackets) using a (pure) Matlab implementation. 
For ALS, exact least squares solution are computed, whereas for SALSA, coarse CG is used. Note that most iterations are performed while the rank is not at its maximum yet.}
\end{table}
%
%
\begin{table}[H]
\centering
\resizebox{\columnwidth}{!}{%
 \begin{tabular}{|c|c||c|c|c|c||c|c|c|c|}
 \hline
 \multicolumn{2}{|c|}{$n=12$} & \multicolumn{4}{c||}{ALS} & \multicolumn{4}{c|}{SALSA} \\
 \hline
 $d$ & $c_{\mathrm{sf}}$ & $\geomean{R_C/\|M_C\|}$ & $\geomean{R_P/\|M_P\|}$ & $\arithmean{\mbox{time}}$ & $\arithmean{\mbox{iter}}$ & $\geomean{R_C/\|M_C\|}$ & $\geomean{R_P/\|M_P\|}$ & $\arithmean{\mbox{time}}$ & $\arithmean{\mbox{iter}}$\\
 \hline
& 2 & 6.5e-03(3.2) & 5.3e-03(3.4) & 72(46) & 423(139) &6.5e-05(1.8) & 1.1e-06(5.8) & 69(16) & 359(66)\\
& 4 & 4.4e-03(1.2) & 4.1e-03(1.3) & 34(34) & 258(114) &3.7e-06(1.8) & 1.9e-08(5.5) & 133(32) & 497(87)\\
& 6 & 3.8e-03(1.2) & 3.6e-03(1.2) & 47(40) & 288(122) &4.7e-07(2.3) & 4.5e-09(4.0) & 157(34) & 520(88)\\
 \hline
 \end{tabular}
 }
 \newline
\vspace*{0.2 cm}
\newline
 \resizebox{\columnwidth}{!}{%
 \begin{tabular}{|c|c||c|c|c|c||c|c|c|c|}
 \hline
 \multicolumn{2}{|c|}{$n=12$} & \multicolumn{4}{c||}{ALS} & \multicolumn{4}{c|}{SALSA} \\
 \hline
 $d$ & $c_{\mathrm{sf}}$ & $\geomean{R_C/\|M_C\|}$ & $\geomean{R_P/\|M_P\|}$ & $\arithmean{\mbox{time}}$ & $\arithmean{\mbox{iter}}$ & $\geomean{R_C/\|M_C\|}$ & $\geomean{R_P/\|M_P\|}$ & $\arithmean{\mbox{time}}$ & $\arithmean{\mbox{iter}}$\\
 \hline
& 2 & 9.7e-02(5.2) & 4.0e-02(6.5) & 50(54) & 322(167) &8.3e-04(1.5) & 9.7e-06(3.1) & 65(11) & 345(37)\\
& 4 & 2.9e-03(38.7) & 4.7e-04(151.0) & 104(102) & 482(260) &7.3e-06(7.0) & 1.3e-07(6.7) & 130(40) & 506(117)\\
& 6 & 3.1e-06(1506.0) & 1.2e-06(2386.6) & 125(79) & 576(271) &1.1e-08(1.4) & 3.6e-09(1.0) & 175(17) & 579(48)\\
\hline
\end{tabular}
}
 \newline
\vspace*{0.2 cm}
\newline
 \resizebox{\columnwidth}{!}{%
 \begin{tabular}{|c|c||c|c|c|c||c|c|c|c|}
 \hline
 \multicolumn{2}{|c|}{$n=12$} & \multicolumn{4}{c||}{ALS} & \multicolumn{4}{c|}{SALSA} \\
 \hline
 $d$ & $c_{\mathrm{sf}}$ & $\geomean{R_C/\|M_C\|}$ & $\geomean{R_P/\|M_P\|}$ & $\arithmean{\mbox{time}}$ & $\arithmean{\mbox{iter}}$ & $\geomean{R_C/\|M_C\|}$ & $\geomean{R_P/\|M_P\|}$ & $\arithmean{\mbox{time}}$ & $\arithmean{\mbox{iter}}$\\
 \hline
& 2 & 4.7e-02(1.1) & 4.6e-02(1.1) & 66(72) & 291(166) &1.3e-03(5.7) & 1.9e-04(14.1) & 53(16) & 185(55)\\
& 4 & 4.4e-02(1.3) & 4.1e-02(1.5) & 74(76) & 288(166) &5.5e-05(1.3) & 2.9e-06(2.4) & 138(19) & 328(34)\\
& 6 & 3.4e-02(4.4) & 3.0e-02(6.4) & 115(167) & 319(256) &1.9e-05(1.4) & 7.2e-07(2.5) & 248(50) & 408(56)\\
\hline
\end{tabular}
}
\vspace{0.1cm}
\caption{\label{generic_table} Results for Subsection \ref{sec:generic_sec} (with arithmetic and geometric variances in brackets) using a (pure) Matlab implementation. 
For ALS, exact least squares solution are computed, whereas for SALSA, coarse CG is used. Note that most iterations are performed while the rank is not at its maximum yet.}
\end{table}

%
%
\end{document}